\documentclass[ejs]{imsart}

\RequirePackage[numbers]{natbib}

\usepackage[utf8]{inputenc}
\usepackage{graphicx}
\usepackage{graphics}
\usepackage{amscd,amsmath,amstext,amsfonts,amsbsy,amssymb,amsthm,eufrak}
\usepackage{mathrsfs}
\usepackage[unicode,colorlinks,linkcolor = blue,citecolor = red]{hyperref}
\usepackage{nccmath}
\usepackage{framed}
\usepackage{hyperref, xcolor}
\usepackage{amsthm}
\usepackage{subcaption}

\usepackage{multicol,graphicx,framed,enumerate}
\usepackage{array,multirow,longtable,fancyhdr,booktabs}
\usepackage{makeidx}
\usepackage{afterpage}
\usepackage{appendix}

\pubyear{2020}
\volume{0}
\issue{0}
\firstpage{1}
\lastpage{8}
\arxiv{2204.02726}

\startlocaldefs

\newtheorem{thm}{Theorem}[section]
\newtheorem{lem}[thm]{Lemma}
\newtheorem{proposition}[thm]{Proposition}
\newtheorem{corollary}[thm]{Corollary}
\newtheorem{assumption}[thm]{Assumption}
\newtheorem{definition}[thm]{Definition}
\newtheorem*{definition*}{Definition}
\newtheorem*{thm*}{Theorem}

\newtheorem{remark}[thm]{Remark}

\newcommand{\bequ}{\begin{equation}}
\newcommand{\eequ}{\end{equation}}
\def\bn{\begin{eqnarray*}}\def\en{\end{eqnarray*}}
\def\bu{\begin{equation}}\def\eu{\end{equation}}

\newcommand{\card}{\mbox{Card}}

\newcommand{\supp}{\mbox{supp}}
\newcommand{\Var}{\mbox{Var}}
\newcommand{\Atan}{\mbox{atan2}}

\def\N{\mathbb{N}}

\def\R{\mathbb{R}}
\def\L{\mathbb{L}}

\def\E{\mathbb{E}}

\def\ind{1}

\makeatother
\endlocaldefs

\begin{document}

\begin{frontmatter}
\title{Adaptive warped kernel estimation for \\ nonparametric regression with \\ circular responses \support{T.D.N. was supported by a public grant as part of the Investissement d’avenir project, reference ANR-11-LABX-0056-LMH, LabEx LMH.}
}
\runtitle{Adaptive warped kernel estimation for nonparametric circular regression}

\begin{aug}
\author{\fnms{Tien Dat} \snm{Nguyen}
\ead[label=e1]{tiendat.nguyen.math@gmail.com}} 
\address{Laboratoire  de Math\'ematiques, UMR 8628, Universit\'e Paris Sud, 91405 Orsay Cedex France \\
\printead{e1}}
\and

\author{\fnms{Thanh Mai} \snm{Pham Ngoc}
\ead[label=e2]{phamngoc@math.univ-paris13.fr}}

\address{LAGA, Institut Galil\'ee, CNRS, UMR 7539, Universit\'e Sorbonne Paris Nord, 93430, Villetaneuse,  France \\
\printead{e2}}

\and

\author{\fnms{Vincent} \snm{Rivoirard}
\ead[label=e3]{Vincent.Rivoirard@dauphine.fr}}

\address{Ceremade, CNRS, UMR 7534, Universit\'e Paris-Dauphine, PSL Research University, 75016 Paris,
	France \\
\printead{e3}\\
}

\runauthor{Nguyen et al.}

\end{aug}

\begin{abstract}
In this paper, we deal with nonparametric regression for circular data, meaning that observations are represented by points lying on the unit circle. We propose a kernel estimation procedure with data-driven selection of the bandwidth parameter. For this purpose, we use  a warping strategy combined with a Goldenshluger-Lepski type estimator. To study optimality of our methodology, we consider the minimax setting and prove, by establishing upper and lower bounds, that our procedure is nearly optimal on anisotropic H\"older classes of functions for pointwise estimation. The obtained rates also reveal the specific nature of regression for circular responses. Finally, a numerical study is conducted, illustrating the good performances of our approach.
\end{abstract}


\tableofcontents
\end{frontmatter}
\section{Introduction}\label{regression.circular:sec:Introduction}
Directional statistics is the branch of statistics which deals with observations that are directions. In this paper, we will consider more specifically \textit{circular data} which arises whenever using a periodic scale to measure observations. These data are represented by points lying on the unit circle of $\mathbb{R}^2$ denoted in the sequel by $\mathbb{S}^1$.  Circular data are collected in many research fields,  for example in ecology (animal orientations), earth sciences (wind, ocean current directions, cross-bed orientations to name a few), medicine (circadian
rhythm), forensics (crime incidence) or social science (clock or calendar effects). Various comprehensive surveys on statistical methods for circular data can be found in Mardia and Jupp~\cite{book:Mardia-Jupp}, Jammalamadaka and
SenGupta~\cite{book:Jammalamadaka-SenGupta}, Ley and Verdebout~\cite{book:Ley-Verdebout} and recent advances are collected in Pewsey and Garc\'ia-Portugu\'es~\cite{Pewsey-GarciaPortuge}. Note that the term \textit{circular data} is also used to distinguish them from data supported on the real line $\R$ (or some subset of it), which henceforth are referred to as \textit{linear data}.

In the present work, we focus on a nonparametric regression model with a circular response and linear predictor. Assume that we have an independent identically distributed (i.i.d in the sequel) sample $\left\{ (X_{i},\Theta_{i}) \right\}_{i=1}^{n}$ distributed as $(X,\Theta)$, where $\Theta$ is a circular random variable, i.e. $\Theta\in\mathbb{S}^1,$ and $X$ is a random variable  with density $f_{X}$ supported on $ \mathbb{R}$.
We assume that the cumulative distribution function of $X$, denoted $F_{X},$ 
is known. We also assume that $F_X$ is invertible on $\R$ meaning that $f_{X}$ is positive on $\R$ and $F_{X}(\R) = (0,1)$. We aim at estimating a function $m$ which contains the dependence structure between the predictors $X_i$ and the observations $\Theta_i$. For our setting, described in Section~\ref{sec:framework}, the regression function $m$ is derived in Equation~\eqref{formula:regress-function.m}.
 

Regression with circular response and linear covariates has been first and mostly explored from a parametric point of view. Pioneered contributions are due to Gould~\cite{Gould}, Johnson and Wehlry~\cite{Johnson-Wehlry} or Fisher and Lee~\cite{Fisher-Lee:1992}. The latter proposed the most popular link-based function (namely the function $2 \arctan$) to model the conditional mean. Major difficulties, among others of such link-based models  involve computational drawbacks to estimate parameters  as identified by Presnell et al.~\cite{Presnell-Morrison-Littel}. Presnell et al.~\cite{Presnell-Morrison-Littel}  in turn suggested alternatively a spherically projected multivariate linear model. Since then, numerous parametric approaches have been proposed, we refer the reader to all the references in Pewsey and Garc\'ia-Portugu\'es~\cite{Pewsey-GarciaPortuge}.  In order to get a more flexible approach, nonparametric paradigm has been considered, first in the pioneering work by Di Marzio et al.~\cite{Marzio-Panzera-Taylor} and more recently in Meil\'an-Vila et al.~\cite{MVila-FFernandez-Crujeiras-Panzera} for the multivariate setting. Surprisingly enough, the nonparametric point of view has only been considered in very few papers. 
Note that contrary to all works aforementioned which classically focus on the conditional mean (which is our goal as well) Alonso-Pena and Crujeiras~\cite{Alonso-Crujeiras} proposed a nonparametric multimodal regression method for estimating the conditional density when for instance the latter is highly skewed or multimodal. Estimation procedures developed in \cite{Marzio-Panzera-Taylor} or \cite{MVila-FFernandez-Crujeiras-Panzera} consist in  estimating the arctangent function of the ratio of the trigonometric moments of  $\Theta$ (more details about this approach are given in the next section as it is the starting point of our procedure). More precisely, in the case of pointwise estimation and covariates supported on $[0,1]$, Di Marzio et al.~\cite{Marzio-Panzera-Taylor} investigated the performances of a Nadaraya-Watson and a local linear polynomial estimators. Theoretically, for regression functions being twice continuously differentiable, they obtained expressions for asymptotic bias and variance. Their proofs are based on linearization of the function arctangent  by using Taylor expansions, but no sharp controls of the remainder terms in the expansions are obtained. Actually obtaining such controls would be very tedious with such an  approach based on Taylor expansions. As for the more recent work of  Meil\'an-Vila et al.~\cite{MVila-FFernandez-Crujeiras-Panzera}, they studied the multivariate setting $[0,1]^d$ with the same estimators and proofs technics. In both papers, neither rates of convergence nor adaptation are obtained and cross-validation is used to select the kernel bandwidth in practice.  By adaptation, we mean that the estimators do not require the specification of the regularity of the regression function which is crucial from a practical point of view. In view of this, we were motivated to fill the gap in the literature. Our goal is twofold:  obtaining optimal rates of convergence for predictors supported on $\R$ and adaptation for estimating $m$ the regression function. To achieve this, we propose a new strategy based on concentration inequalities along with warping methods. 

\medskip

\textit{Our contributions.} Under the assumption that the cumulative distribution function (c.d.f.) of the design $X$ is known and invertible, warping methods used in this paper consist in introducing the auxiliary function $g := m \circ F_X^{<-1>}$, with $ F_X^{<-1>}$ the inverse of $F_X$. We then use classical kernel rules to estimate the function $g$ in the specific framework of circular data. Our procedure needs to select two bandwidths. Fully data-driven selection of bandwidths is performed by using a Goldenshluger-Lepski type procedure \cite{Goldenshluger-Lepski:2011}. Then, theoretical performances are studied. We consider the minimax setting and prove by establishing upper and lower bounds that our procedure is nearly optimal on anisotropic H\"older classes of functions for pointwise estimation. These results are stated in Theorems \ref{adaptive:thm:pointwise-risk.upper-bound:g} and \ref{lower-bound} respectively. Then, we conduct a numerical study whose goal is twofold. We first investigate the best tuning parameters of our procedure. Once tuned, our estimates are used on artificial data and compared to other classical methods. The numerical study reveals the good performances of our methodology.

\medskip

\textit{Plan.} In section \ref{sec:case-known-design-distribution}, we explain how to take into account the circular nature of the response and then propose our data-driven kernel estimator of the regression function $m$  based on warping strategy and the Goldenshluger-Lepski bandwidth selection rule. Section \ref{sec:theoretical} contains the theoretical results. Section \ref{sec:regression.circular:numerical.simulation} presents numerical results including simulations. Finally, all the proofs are deferred to Section \ref{sec:proofs}.

\medskip

\textit{Notations.} It is necessary to equip the reader with some notations. In the sequel, a point on $\mathbb{S}^1$ will not be represented as a two-dimensional vector $\bold{w}=(w_2,w_1)^\top$ with unit Euclidean norm but as an angle $\theta=atan2(w_1,w_2) $ defined as follows:
\begin{definition}\label{def:full.formula.atan2} The function $\Atan: \R^2\setminus (0,0)\mapsto [-\pi;\pi]$ is defined for any $(w_1,w_2)\in\R^2\setminus (0,0)$ by \emph{
		$$
		\Atan (w_1,w_2):= \left \{
		\begin{array}{ll}
		\arctan\big( \frac{w_1}{w_2} \big) & \quad \mbox{if } \quad w_2 \geq 0 , w_{1} \neq 0 \\
		{0} & \quad {\mbox{if } \quad w_2 > 0 , w_{1} = 0} \\
		\arctan\big( \frac{w_1}{w_2} \big) + \pi & \quad \mbox{if } \quad w_2 < 0, w_1 >0 \\
		\arctan\big( \frac{w_1}{w_2} \big) - \pi & \quad \mbox{if } \quad w_2 < 0, w_1 \leq 0,
		\end{array}
		\right.
		$$
		with $\arctan$ taking values in $[- \pi/ 2, \pi/ 2]$. In particular for $w_1>0$, $\Atan (w_1,0)=\arctan(+\infty)=\pi/2$ and $\Atan (-w_1,0)=\arctan(-\infty)=-\pi/2$.
}\end{definition} 
In this definition, one has arbitrarily fixed the origin of $\mathbb{S}^1$ at $(1,0)^\top$ and uses the anti-clockwise direction as positive. Thus, a circular random variable can be represented as angle over $[-\pi, \pi)$. Observe that $\Atan (0,0)$ is not defined. Hereafter, $\left\| \cdot \right\|_{\L^{1}(\R)}$ and $\left\| \cdot \right\|_{\L^{2}(\R)}$ respectively denote the $\L^{1}$ and $\L^{2}$ norm on $\R$ with respect to the Lebesgue measure:
\begin{equation*}
\left\| f \right\|_{\L^{1}(\R)}  =  \int_{\R} |f(y)| dy , \quad \left\| f \right\|_{\L^{2}(\R)}  = \Big( \int_{\R} |f(y)|^{2} dy  \Big)^{1/2} .
\end{equation*}
The $\L^{\infty}$ norm is defined by $\left\|f\right\|_{\infty} = \sup_{y \in \R} |f(y)|$. Moreover, we denote $*$ the classical convolution product defined for functions $f, g$ by $f * g(x) := \int_{\R} f(x-y). g(y) dy $, for $x \in \R$. Finally, for $\alpha \in \R$, $[\alpha]_{+} := \max\left\{ \alpha ; 0 \right\}$, and for $\beta > 0$, $\lfloor \beta \rfloor$ denotes the largest integer strictly smaller than $\beta$.
\section{The estimation procedure}
\label{sec:case-known-design-distribution}
After recalling the framework of circular data in Section~\ref{sec:framework}, Section~\ref{sec:warping-strategy} is devoted to the construction of an estimator for $m(x)$, at a given point $x \in~\R$  which will be fixed along this paper, using warped kernel methods. Then, Section~\ref{sec:regression.circular:data-driven.estimator} presents a data-driven procedure for bandwidth selection by using the Goldenshluger-Lepski methodology. 
\subsection{The framework of circular data}\label{sec:framework}
There is no doubt that, due to their periodic nature, circular data are fundamentally different from linear ones, and thus need specific tools. To measure the closeness between two angles $\theta_1$ and $\theta_2$, we do not consider the natural distance
$$d(\theta_{1} , \theta_{2}) := \min \left\{ \big| \theta_{1}  - \theta_{2}  +  2k\pi \big| : k \in \mathbb{Z} \right\},\quad \theta_1, \theta_2 \in [-\pi, \pi),$$
but we focus on $d_c$ with $$d_{c} (\theta_{1}, \theta_{2}) := 1 - \cos(\theta_{1} - \theta_{2}),\quad \theta_1, \theta_2 \in [-\pi, \pi),$$   which is extensively used in the literature of directional statistics (see for instance Section 2 in the seminal monograph by Mardia and Jupp \cite{book:Mardia-Jupp}, Section 3.2.1 of \cite{book:Ley-Verdebout}, \cite{Marzio-Panzera-Taylor} or \cite{MVila-FFernandez-Crujeiras-Panzera}). Note that the divergence $d_c$ corresponds to the usual squared Euclidean norm in $\R^{2}$. Indeed, the angles $\theta_{1}$ and $\theta_{2}$ determine the corresponding points $( \cos \theta_{1} ,  \sin \theta_{1})$ and $( \cos \theta_{2} ,  \sin \theta_{2})$ respectively on the unit circle $\mathbb{S}^{1}$.  Then, the usual squared Euclidean norm  in $\R^{2}$ reads
\begin{align*}
\big( \cos \theta_{1} - \cos \theta_{2} \big)^{2} + \big( \sin \theta_{1} - \sin \theta_{2} \big)^{2} = 2. \big[ 1 - \cos (\theta_{1} - \theta_{2}) \big] = 2. d_{c}(\theta_{1}, \theta_{2}).
\end{align*}
Hence, $\sqrt{d_c}$ is a distance on $[-\pi,\pi)$ and we naturally look for a measurable function $m$ such that: 
\begin{align}\label{minimisation}
\mathbb{E} \big[ d_c(\Theta, m(X))  \big] =
\underset{f: \;\mathbb{R} \rightarrow [-\pi,\pi) }{ \textrm{min} } \; \mathbb{E} \big[ d_c(\Theta, f(X))  \big],
\end{align}
where the minimum is taken over $[-\pi, \pi)$-valued functions $f$ that are measurable with respect to the $\sigma$-algebra generated by $X$. It is interesting to notice that the minimization problem (\ref{minimisation}) is directly linked to the definition of the Frechet mean on the circle (see Charlier \cite{Charlier:2013}). 
Furthermore, in the literature of directional statistics, the problem of finding such a regression function $m(X)$ as defined in (\ref{minimisation}) has been already considered to solve the circular regression problem (see \cite{Marzio-Panzera-Taylor:2014} and \cite{MVila-FFernandez-Crujeiras-Panzera}).
\\
Now let us work conditionally to $X$. For $x \in \R$ let
\begin{equation}\label{def:m1m2}
m_{1}(x) := \mathbb{E} \big( \sin (\Theta) | X=x\big)\quad and \quad m_{2}(x) := \mathbb{E} \big( \cos (\Theta) | X=x\big).
\end{equation}
Moreover, write for an arbitrary function $f: \mathbb{R} \rightarrow [-\pi,\pi) $
\begin{align*}
\mathbb{E} \big[ \cos (\Theta - f(X)) | X \big] &= \cos(f(X)). m_{2}(X) + \sin(f(X)). m_{1}(X)
\\
&=  \sqrt{(m_{2}(X))^{2} + (m_{1}(X))^{2} }. \cos (f(X) - \gamma(X)) ,
\end{align*}
where $\gamma : \;\mathbb{R} \rightarrow [-\pi,\pi)$ is defined for $x\in\R$ by 
\begin{equation*}
\cos(\gamma(x)) := \dfrac{m_{2}(x)}{\sqrt{(m_{2}(x))^{2} + (m_{1}(x))^{2} }}   \quad  \textrm{ and } \quad  \sin(\gamma(x)) := \dfrac{m_{1}(x)}{\sqrt{(m_{2}(x))^{2} + (m_{1}(x))^{2} }}.
\end{equation*}
Observe that 
$$\gamma(x) = \Atan ( m_{1}(x),m_{2}(x)).$$
Thus, we have
\begin{align*} \label{eq:m(X).given.by.atan2}
\underset{f: \;\mathbb{R} \rightarrow [-\pi,\pi) }{ \textrm{min} }  \; \mathbb{E} \big[ d_c(\Theta, f(X))  \big]
&=1- \underset{f: \;\mathbb{R} \rightarrow [-\pi,\pi) }{ \textrm{max} } \E\Big[\mathbb{E} \big[ \cos (\Theta - f(X)) | X \big] \Big]\\
&=1- \underset{f: \;\mathbb{R} \rightarrow [-\pi,\pi) }{ \textrm{max} } \E\Big[\sqrt{m_1^2(X)+m_2^2(X)}\cos (f(X) - \gamma(X))\Big].  
\end{align*}
Finally the minimizer of the minimization problem \eqref{minimisation} is achieved for 
$$f(x)=\gamma(x) = \Atan \big( m_{1}(x),m_{2}(x)\big).$$

In conclusion, the circular nature of the response is taken into account by the arctangent of the ratio of the conditional expectation of sine and cosine components of $\Theta$ given $X$ 
and we tackle the problem by estimating the function \begin{equation}
m(x) = \Atan (m_1(x), m_2(x)),\quad x\in \R,
\label{formula:regress-function.m}
\end{equation} 
with $m_1$ and $m_2$ defined in \eqref{def:m1m2}. 

\begin{remark}
Observe that if $m_1(x)=m_2(x)=0$, then $m(x)$ is not defined. This occurs if and only if
$$
\phi_1(f(\cdot|x)):=\int_{-\pi}^{\pi} e^{i\theta}f(\theta|x) d\theta=0,
$$
where $f(\cdot|x)$ denotes the conditional density of $\Theta|X=x$. Note that $\phi_1(f(\cdot|x))$ plays a specific role in the literature of directional statistics. See for instance Section 3.4.2 of \cite{book:Mardia-Jupp}.
\end{remark}
In the sequel, we estimate the circular regression function $m$ as defined in \eqref{formula:regress-function.m} under the condition
\begin{equation}\label{not0}
\phi_1(f(\cdot|x))\not=0.
\end{equation}
We set $\zeta:=(\zeta_i)_{i=1,\ldots,n}$ the vector of errors so that 
\begin{equation} \label{def:regression-model}
\Theta_i = m(X_i) + \zeta_i \quad (\textrm{mod } 2\pi), \quad i = 1, \ldots, n.
\end{equation}
Our estimation methodology is based on a warping strategy.
\subsection{Warping strategy}
\label{sec:warping-strategy}
The popular Nadaraya-Watson (NW) methodology provides a natural estimator of $m$ of the form
\begin{equation*}
\widehat{m}^{NW}_{h} : x \longmapsto \dfrac{\frac{1}{n} \sum_{j=1}^{n} \Theta_{j}.K_{h}(x - X_{j}) }{ \frac{1}{n} \sum_{j=1}^{n}  K_{h}(x - X_{j}) },
\end{equation*}
with $K : \R \rightarrow \R$ such that $\int_{\mathbb{R}} K(y) dy = 1$ and $K_{h}(\cdot) := \frac{1}{h} K(\frac{\cdot}{h})$, for some bandwidth $h > 0$.
However, on the one hand, the denominator which can be small may lead to some instability. On the other hand, as adaptive estimation requires the data-driven selection of the bandwidth, the ratio form of the NW estimate indicates that we should select two bandwidths: one for the numerator and one for the denominator. 
Consequently, considering NW estimators for $m_{1}$ and $m_{2}$ involve four bandwidths. This makes the study of these estimators quite intricate. 

Recalling that $g= m \circ F_X^{<-1>}$ with $ F_X^{<-1>}$ the inverse of $F_X$, warping methods then boil down to first estimating $g$ by say $\widehat g$ and then estimating the regression function of interest $m$ by $\widehat g\, \circ F_X$. To deal with regression with random design, the warping strategy has been applied for instance by Kerkyacharian and Picard~\cite{Kerkyacharian-Picard:2004}, Pham Ngoc~\cite{PhamNgoc2009}, Chagny~\cite{Chagny2013} and Chagny et al.~\cite{Chagny-Laloe-Servien:2019}. Among the advantages of this method, let us mention that a warped kernel estimator does not involve a ratio, which strengthens its stability whatever the design distribution, even when the design is inhomogeneous.  In our framework, in order to construct an estimator for the regression function $m$, we first estimate $m_{1}$ and $m_{2}$ (see~\eqref{formula:regress-function.m}). Consequently, we introduce two auxiliary functions $g_{1} , g_{2} : (0,1) \longmapsto \R$ defined by
	\begin{equation*}
	g_{1}  :=  m_{1} \circ  F_{X}^{<-1>}, \quad \textrm{and} \quad g_{2}  :=  m_{2} \circ  F_{X}^{<-1>} ,
	\end{equation*}
	so that $m_{1} = g_{1} \circ F_{X}$ and $m_{2} = g_{2} \circ  F_{X}$; we then have for $u \in  (0,1)$
\begin{align*}
g(u)= \Atan \big( g_{1}(u),g_{2}(u) \big).
\end{align*} 
Our fully data-driven approach is based on the selection of two bandwidths that adapt automatically to the unknown smoothness of functions $g_1$ and $g_2$.

Now, we propose to adapt the strategy developed in the linear case by Chagny et al. in \cite{Chagny-Laloe-Servien:2019}. The warping device is based on the transformation $F_{X}(X_i)$ of the data $X_i$, $i = 1,\ldots,n$. 
We first define kernels considered in our framework as follows.
\begin{definition} \emph{
Let $K: \mathbb{R} \rightarrow \mathbb{R}$ be an integrable function such that $K$ is compactly supported, $K \in  \L^{\infty}(\R) \cap \L^{1}(\R) \cap \L^{2}(\R)$. We say that $K$ is a kernel if it satisfies $\int_{\mathbb{R}} K(y)dy = 1$. 
}\end{definition}
Then, for $u \in (0,1)$, we estimate $g_{1}(u)$ and $g_{2}(u)$ by
\begin{align}
\widehat{g}_{1,h_{1}}(u) :=  \dfrac{1}{n} \sum_{i=1}^{n} \sin (\Theta_{i}) . K_{h_{1}} (u -  F_{X}(X_{i})) , 
\label{formula:estimate.g1-g2.fixed-h}
\\   
\widehat{g}_{2,h_{2}}(u) :=  \dfrac{1}{n} \sum_{i=1}^{n} \cos (\Theta_{i}) . K_{h_{2}} (u -  F_{X}(X_{i}))  \nonumber
\end{align}
respectively,  where $h_{1}, h_{2} > 0$ are bandwidths of kernels $K_{h_{1}}(\cdot)$ and $K_{h_{2}}(\cdot)$ respectively.\\
Thus, we estimate $g$ by
\begin{equation}
\widehat{g}_{h}(u) := \Atan \big( \widehat{g}_{1,h_{1}}(u), \widehat{g}_{2,h_{2}}(u) \big) , \quad  u \in  (0,1),
\end{equation}
where we denote $h := (h_{1},h_{2})$. Moreover, as a consequence, for $x \in \R$, the estimators for $m_{1}$ and $m_{2}$ are
\begin{equation} \label{formula:estimator.m_1h}
\widehat{m}_{1,h_{1}}(x)  :=  \widehat{g}_{1,h_{1}} \big( F_{X}(x) \big)  = \dfrac{1}{n } \sum_{i=1}^{n} \sin (\Theta_{i}). K_{h_{1}} \big(  F_{X}(x) -  F_{X}(X_{i}) \big)  ,
\end{equation}
and
\begin{equation} \label{formula:estimator.m_2h}
\widehat{m}_{2,h_{2}}(x)  :=  \widehat{g}_{2,h_{2}} \big( F_{X}(x) \big) = \dfrac{1}{n } \sum_{i=1}^{n} \cos (\Theta_{i}). K_{h_{2}} \big(  F_{X}(x) -  F_{X}(X_{i}) \big)  .
\end{equation}
Using $m = g \circ  F_{X}= \Atan \big(m_1,m_2\big),$ we then obtain an estimator of $m(x)$ at $x \in \R$ by setting $$\widehat{m}_{h}(x) := \Atan \big( \widehat{m}_{1,h_{1}}(x), \widehat{m}_{2,h_{2}}(x) \big)=\widehat{g}_{h}\big(  F_{X}(x)\big).$$
\subsection{Bandwidth selection}
\label{sec:regression.circular:data-driven.estimator}
We study the pointwise risk of the estimator $ \widehat{m}_{h}(x)$ associated to the divergence $d_c$.  The expression of the risk is then
\begin{equation*}
\E \Big[  d_c( \widehat{m}_{h}(x),m(x)) \Big] = \mathbb{E} \Big[ d_c\big( \widehat{g}_{h}(F_X(x)), g(F_X(x)) \big) \Big].
\end{equation*} 
We first focus on the estimator $\widehat{g}_{h}$ of $g$ by studying the adaptive choice of bandwidths belonging to a convenient grid $\mathcal{H}_{n}$. To define the latter, we assume that the kernel $K$ satisfies $\supp(K) \subseteq [-A,A]$ for some $A>0$ and we take $h_{\max}$ a constant such that $F_X(x) - A. h_{\max} > 0$ and $F_X(x) + A. h_{\max} < 1$. Then, we set
\begin{equation}
\label{def:bandwidth.colletion.H_n}
\mathcal{H}_{n} := \left\{ h = k^{-1} :\  k \in{\mathbb N}^{*}, \ h \leq h_{\max} , \  n.h >  \max \Big( \dfrac{\left\|K\right\|_{\L^{2}(\R)}^{2} }{\left\|K\right\|_{\infty}^{2} } ; 1 \Big)  . \log(n)\right\}.
\end{equation}
\begin{remark}
Observe that the condition
$$F_X(x) - A. h_{\max} > 0,\quad F_X(x) + A. h_{\max} < 1$$
is satisfied for $n$ large enough if $h_{\max}$ depends on $n$ and goes to 0 (even slowly) when $n\to+\infty$.\end{remark}
We have $  \card\big( \mathcal{H}_{n} \big)\lesssim n/\log n$. 
In the sequel, we apply the method proposed by Goldenshluger and Lepski in~\cite{Goldenshluger-Lepski:2011} to select an optimal value for bandwidths $h_{1}$ and $h_{2}$ automatically. Let $j\in\{1,2\}$.
For $h_{j} \in \mathcal{H}_{n}$ and $v \in (0,1)$ we set
\begin{equation} \label{adaptive:def:A1(h1,v)}
A_{j}(h_{j} , v)  :=  \underset{h_{j}' \in \mathcal{H}_{n}}{ \sup }  \Big\{ \big| \widehat{g}_{j,h_{j},h_{j}'}(v) - \widehat{g}_{j,h_{j}'}(v) \big|  -  \sqrt{\widetilde{V}_{j}(n,h_{j}')} \Big\}_{+}  ,
\end{equation}
with $\widetilde{V}_{j}(n,h_{j}') := c_{0, j}.\dfrac{\log(n) . \left\|K\right\|_{\L^2(\R)}^2}{n . h_{j}'}$, $c_{0, j} > 0$ a tuning parameter and \begin{equation*}
\widehat{g}_{j,h_{j},h_{j}'}(v) := \big( K_{h_{j}'}*\widehat{g}_{j,h_{j}} \big) (v),
\end{equation*}
so that $\widehat{g}_{j,h_{j},h_{j}'}(v) = \widehat{g}_{j,h_{j}',h_{j}}(v)$.
Then, a data-driven choice of bandwidth $h_{j}$ is performed as follows:
\begin{equation} \label{adaptive:def:hat.h1:criterion}
\widehat{h}_{j} = \underset{h_{j} \in \mathcal{H}_{n}}{ \textrm{argmin} } \Big\{ A_{j}(h_{j} , v)  +   \sqrt{\widetilde{V}_{j}(n,h_{j})} \Big\}.
\end{equation}
Observe that our bandwidth  selection rule depends on $x$. 
The criterion~\eqref{adaptive:def:hat.h1:criterion} is inspired from \cite{Goldenshluger-Lepski:2011}, in order to mimic the optimal "bias-variance" trade-off in the pointwise quadratic decomposition:  
\begin{align*}
\E[|\widehat{g}_{j,h_{j}}( v )-g_j(v)|^2] &= |\E[\widehat{g}_{j,h_{j}}( v )]-g_j(v)|^2+\E[|\widehat{g}_{j,h_{1}}( v )-\E[\widehat{g}_{j,h_{j}}( v )]|^2]
\\
&=: b^2(h_j,v)+V(h_j,v).
\end{align*}
It is common to use $\widetilde{V}_{j}(n,h_{j}')$ to provide an upper bound for the variance term $V(h_j,v)$ (see Section~\ref{sec:lem:point-wise:mean-var-cov.g_{2}:proof}), whereas the more involved task of the Goldenshluger-Lepski method is to provide an estimate for the bias term by comparing pair-by-pair several estimators. In our framework, the bias term corresponds to 
$$b(h_j,v)=|\E[\widehat{g}_{j,h_{j}}( v )]-g_j(v)|=\big| \big( K_{h_{j}} * g_{j} \big) (v) - g_{j}(v) \big|,$$ (see~\eqref{form:compute.expetation.g_{1,h1}}), so it is natural to estimate it by an estimator of the form $\big| \big(  K_{h_{j}}*\widehat{g}_{j,h_{j}'} \big) (v) - \widehat{g}_{j,h_{j}'}(v) \big|$. 
Thus, the estimator of the bias term is $A_{j}(h_{j},v)$, defined in~\eqref{adaptive:def:A1(h1,v)}, where the second term $\sqrt{\widetilde{V}_{j}(n,h_{j}')}$ controls the fluctuations of the first term. 
Now, we define the kernel estimator of $g(v)$ with data-driven bandwidths as follows:
\begin{equation} \label{def:adaptive-estimator:g_h}
\widehat{g}_{\widehat{h}}(v)  :=  \Atan \big( \widehat{g}_{1,\widehat{h}_1}(v) , \widehat{g}_{2,\widehat{h}_2}(v) \big)  ,
\end{equation}
where we denote $\widehat{h} := (\widehat{h}_{1}, \widehat{h}_{2})$. 
We finally define the adaptive estimator for $m(x)$ by \begin{equation} \label{adaptive:def:m_h.adaptive}
\widehat{m}_{\widehat{h}}(x) := \Atan \big( \widehat{m}_{1,\widehat{h}_1}(x), \widehat{m}_{2,\widehat{h}_2}(x) \big).
\end{equation}  
\section{Theoretical results}\label{sec:theoretical}
\subsection{Minimax rates of convergence}
\label{sec:adaptive:oracle-inequalities.rate-convergence}
The minimax approach is a framework that shows the optimality of an estimate among all
possible estimates.  The minimax pointwise quadratic risk for the estimator $\widehat{g}_{\widehat{h}} =  \Atan \big( \widehat{g}_{1,\widehat{h}_1}, \widehat{g}_{2,\widehat{h}_2} \big)$ will be derived from the following control of the pointwise quadratic risks  of $\widehat{g}_{1,\widehat{h}_1}$ and $\widehat{g}_{2,\widehat{h}_2}$.
\begin{proposition}\label{adaptive:prop:upper-bound.high-Proba:g2-g1}
Consider the collection of bandwidths $\mathcal{H}_{n}$ defined in~\eqref{def:bandwidth.colletion.H_n}. Let $j\in\{1,2\}$ and $q \geq  1$ and assume that $\min\left\{c_{0, 1} ; c_{0, 2}\right\} \geq 16 \big(2 + q \big)^{2}. \big( 1 + \left\|K\right\|_{\L^{1}(\R)} \big)^{2}$. Then,  with probability larger than $1 - 4.n^{-q}$,
\begin{align*}
\big|  \widehat{g}_{j,\widehat{h}_j}( F_X(x) )  - g_{j}(F_X(x))  \big|   \leq  \inf_{h_{j} \in \mathcal{H}_{n}} \Big\{ & \big( 1 + 2. \left\|K\right\|_{\L^{1}(\R)} \big). \left\| g_{j} - K_{h_{j}}*g_{j} \right\|_{\infty} \\ 
&+ 3. \sqrt{\widetilde{V}_{j}(n,h_{j})} \Big\} .
\end{align*}
\end{proposition}
\noindent
The proof of Proposition~\ref{adaptive:prop:upper-bound.high-Proba:g2-g1} is given in Section~\ref{sec:adaptive:prop:upper-bound.high-Proba:g2-g1:proof}. Roughly speaking, in view of results of Section~\ref{sec:regression.circular:proofs:preliminaries}, the right hand side of the inequality stated in Proposition~\ref{adaptive:prop:upper-bound.high-Proba:g2-g1} may be viewed as the bias-variance decomposition of the pointwise quadratic-risk of the best warped-kernel estimate, up to a logarithmic term. 
\begin{remark}
Examining the proof of Proposition~\ref{adaptive:prop:upper-bound.high-Proba:g2-g1}, the "uniform bias" comes from Inequality \eqref{uniformbias}. This control can be refined and the term $\left\| g_{j} - K_{h_{j}}*g_{j} \right\|_{\infty} $ can be replaced by
$$\sup_{t \in V(u_x)}\big| g_{j}(t) - \big(K_{h_{j}}*g_{j}\big)(t)\big|,$$
with $V(u_x):=\{t: \ |t-u_x|\leq Ah_{\max}\}.$ Observe that the size of this neighborhood of $u_x$ goes to 0 if $h_{\max}\to 0$.
\end{remark}

Since the function $\Atan (w_{1},w_{2})$ is undefined when $w_{1} = w_{2} = 0$, it is reasonable to consider the following assumption:
\begin{assumption}\label{assumption:c_{low}} \emph{We have 
\begin{equation*}
m_1(x)\not= 0 \quad \textrm{ or } \quad m_2(x)\not= 0.
\end{equation*}
Then, we define $\delta>0$ such that
\begin{equation}\label{deltawx}
\delta=\left\{
\begin{array}{ccc}
\min\big(|m_1(x)|, |m_2(x)|\big)  &\mbox{ if }   & m_1(x)\not= 0 \mbox{ and }  m_2(x)\not= 0, \\
|m_1(x)|  &\mbox{ if }   & m_2(x)= 0,  \\
|m_2(x)|  &\mbox{ if }   & m_1(x)= 0.
\end{array}
\right.
\end{equation}
}  \end{assumption}
In the minimax setting, we need some assumptions on the regularity of $g_1$ and $g_2$. Thus, we introduce the following H\"older classes that are adapted to local estimation.
\begin{definition}\label{definition:holder.class} 
\emph{Let $\beta > 0$ and $L > 0$. The H\"older class $\mathcal{H}(\beta,L)$ is the set of functions $f : (0,1) \longmapsto \mathbb{R}$, such that $f$ admits derivatives up to the order $\lfloor \beta \rfloor$, and for any $(y,\widetilde{y})\in(0,1)^2$,
\begin{align*}
\left| \dfrac{ d^{\lfloor \beta \rfloor} f }{ (d y )^{\lfloor \beta \rfloor} } (\widetilde{y}) - \dfrac{ d^{\lfloor \beta \rfloor} f }{ (d y)^{\lfloor \beta \rfloor} } (y) \right|  \leq  L. \big|\widetilde{y}-y\big|^{\beta - \lfloor \beta \rfloor}.
\end{align*}
} \end{definition}
\noindent

We also consider the following assumption on the kernel $K$:
\begin{assumption}\label{assumption:kernel.K} \emph{
The kernel $K$ is of order $\mathcal{L} \in \mathbb{R}_{+}$, i.e.
\begin{enumerate}
\item[(i)] $C_{K, \mathcal{L}} :=  \int_{\mathbb{R}}(1+ |y|)^{\mathcal{L}} . |K(y)| dy < \infty $ ;
\item[(ii)]  $\forall k \in \left\{ 1,...,\lfloor  \mathcal{L} \rfloor  \right\}$, $\int_{\mathbb{R}} y^{k}. K(y) dy = 0 $.
\end{enumerate}
} \end{assumption}
Now, we obtain an upper bound for the pointwise risk of our final estimator $\widehat{m}_{\widehat{h}}$ at $x$ defined in~\eqref{adaptive:def:m_h.adaptive}:
\begin{thm} \label{adaptive:thm:pointwise-risk.upper-bound:g}
Let $\beta_{1} , \beta_{2} , L_{1}, L_{2} > 0$. Suppose that $g_{1}$ belongs to $\mathcal{H}(\beta_{1}, L_{1})$, $g_{2}$ belongs to $\mathcal{H}(\beta_{2}, L_{2})$, the kernel K satisfies Assumption~\ref{assumption:kernel.K} with an index $\mathcal{L} \in \mathbb{R}_{+}$ such that $\mathcal{L} \geq  \max( \beta_{1}  , \beta_{2}  )$. Let $q \geq 1$, and suppose that $\min\left\{c_{0, 1} ; c_{0, 2}\right\} \geq 16 \big( 2 + q \big)^{2}. \big( 1 + \left\|K\right\|_{\L^{1}(\R)} \big)^{2}$. Then, by taking $h_{\max}=(\log n)^{-1}$,  under Assumption~\ref{assumption:c_{low}}, for $n$ sufficiently large, 
\begin{equation*}
\E \Big[ d_c( \widehat{m}_{\widehat{h}}(x), m(x) )\Big]
\leq \frac{C}{\delta^2}  . \max \left\{  \psi_{n}^2(\beta_{1}) ,  \psi_{n}^2(\beta_{2})\right\} , 
\end{equation*}
where 
$$\psi_{n}(\beta_{1}) = \big( \log(n)/n \big)^{\frac{\beta_{1}}{2\beta_{1} + 1}},\quad \psi_{n}(\beta_{2}) = \big( \log(n)/n \big)^{\frac{\beta_{2}}{2\beta_{2} + 1}},$$ $\delta$ is defined in~\eqref{deltawx} 	and $C$ is a constant depending on $\beta_{1}, \beta_{2},L_{1}, L_{2},c_{0,1}, c_{0,2}$ and~$K$.  
\end{thm}
A proof of Theorem~\ref{adaptive:thm:pointwise-risk.upper-bound:g} is given in Section~\ref{sec:adaptive:thm:pointwise-risk.upper-bound:g:proof}. Observe that if $\beta_1=\beta_2=\beta$, then we obtain the rate $\psi_{n}(\beta)=(\log n/n)^{\beta/(2\beta+1)}$, which is the optimal rate for adaptive univariate regression function estimation and pointwise risk (see e.g. Section 2 in~\cite{Bochkina-Sapatinas.2009}). Note that the logarithmic term appearing in the rate of convergence is expected since we deal with pointwise adaptive estimation. For further details, we refer the reader to   Lepski \cite{Lepski1990} and Lepski and Spokoiny \cite{Lepski-Spokoiny:1997} who have highlighted and discussed this fact for the Gaussian white noise model.  

\begin{remark} 
Eventually, to obtain the fully computable estimator, we replace the c.d.f. $F_{X}$ by its natural estimate $\widehat{F}_{n}(y) :=n^{-1} \sum_{i=1}^{n} \ind_{X_i \leq y}$. Following arguments of \cite{Chagny2015}, this replacement should not change the final rate of convergence of our nonparametric estimator. \end{remark}

\subsection{Minimax lower bounds}\label{sec:lowerbound}
To establish minimax lower bounds, we assume that the $\zeta_i$'s are centered i.i.d. random angles, independent of the $X_i$'s. We also assume that  Model~\eqref{def:regression-model} satisfies the following assumption.
\begin{assumption}\label{assumption:model} 
	The design points $X_i$'s are i.i.d. random variables with density $f_X(.)$ on $[0, 1]$ such that there exists $\mu_0 < \infty$  and $f_X(t)\leq  \mu_0$ $\forall t \in [0, 1]$ and the errors $\zeta_i$ have common density $p_\zeta(.)$ on $\mathbb{S}^1$ with respect to the Lebesgue measure on $\mathbb{S}^1$, verifying
	\begin{equation}\label{condition-erreur}
	\exists\, p_*>0,\ \exists\, \theta_0>0 : \int p_{\zeta}(t) \log \frac{p_\zeta(t)}{p_{\zeta}(t+\theta)} dt \leq p_* \theta^2, \;\forall |\theta|\leq \theta_0. 
	\end{equation}
\end{assumption}
The subsequent minimax lower bound is based on a reduction scheme based on some well-chosen probability distributions.  The closeness between the associated models is measured by using the Kullback-Leibler divergence and is controlled by using Assumption~\ref{assumption:model}.  
In the sequel, the function $m$ belongs to the class $\tilde \Sigma(\beta,L)$ defined as the set of functions $f : [0,1] \longmapsto \mathbb{S}^1$ such that the derivative $f^{(l)}$, $l=\lfloor \beta \rfloor$ exists and verifies
$$
\sqrt{d_c( f^{(l)}(t), f^{(l)}(t'))} \leq L |t-t'|^{\beta-l}, \quad \forall t, t' \in [0,1].
$$  
\begin{remark}\label{rem:support} 
	For two classes of functions $\mathcal{D}$ and $\mathcal{D}'$ such that  $\mathcal{D} \subset  \mathcal{D}' $, a lower bound for the minimax rate of convergence for $\mathcal{D}$ will also be a lower bound for the minimax rate for $\mathcal{D}'$. Hence, this justifies the restriction of the study of the lower bound to circular functions $m$ defined on $[0,1]$.
\end{remark} 
\begin{remark} 
The classical von Mises distribution with location parameter $\mu \in  [-\pi,\pi) $ and  concentration parameter $\kappa >0$ whose density $f_{vM(\mu, \kappa)}$ is defined for any $\theta \in [-\pi,\pi)$ by
\begin{equation}\label{vmdensity}
f_{vM(\mu, \kappa)}(\theta)=c(\kappa). \exp \big({\kappa}. \cos( \theta - {\mu} ) \big),
\end{equation}
with $c(\kappa)$ the normalizing constant, satisfies condition (\ref{condition-erreur}). This is proved in Lemma \ref{preuve-condition-erreur}. Note that apart from the most popular von Mises distribution, two other classical circular distributions namely the cardioid and the wrapped Cauchy distributions, respectively defined by (see \cite{{book:Ley-Verdebout}})
$$\theta\in[-\pi,\pi)\longmapsto \frac{1}{2\pi} \big(1+ 2\rho \cos(\theta-\mu)\big), \quad \rho \in [0, \frac 1 2), \ \mu \in  [-\pi,\pi)
$$
and
$$\theta\in[-\pi,\pi)\longmapsto \frac{1}{2\pi} \frac{1-\ell^2}{1+ \ell^2 -2 \ell \cos(\theta-\mu)}, \quad \ell \in [0, 1 ), \ \mu \in  [-\pi,\pi)
$$
also satisfy (\ref{condition-erreur}). Proofs are very similar to the von Mises case.  
\end{remark}
We obtain the following lower bound:
\begin{thm}\label{lower-bound}
	Let $\beta >0$ and $L>0$. Under Assumptions \ref{assumption:model},  we have
	$$
	\liminf_{n \rightarrow \infty} \inf_{T_n} \sup_{m\in\tilde \Sigma(\beta,L)} \E\big[n^{\frac{2\beta}{2\beta+1}}d_c(T_n(x),m(x))\big] \geq \tilde c,
	$$   
	where $\tilde c$ depends only on $\beta, L, p_*$ and $\mu_0$ and the infimum is taken over all possible estimates based on observations $(\Theta_i,X_i)_{i=1,\ldots,n}$.
\end{thm}
According to Remark~\ref{rem:support}, Theorem \ref{lower-bound} entails that the lower bound for the minimax risk for functions 
$m : \R \longmapsto \mathbb{S}^1$ such that $m \in \tilde \Sigma(\beta,L)$ is $n^{-\frac{2\beta}{2\beta+1}}$. 
Now let us connect this result  to the upper bound obtained in  Theorem \ref{adaptive:thm:pointwise-risk.upper-bound:g}. As the function $\Atan$ is infinitely differentiable on $\R^* \times \R^*$, and if $F_X$ is smoother than $g_1$ and $g_2$, then if one writes $m(x) =  \Atan (g_1(F_X(x)),g_2(F_X(x))$, the smoothness $\beta$ of $m$ will be the minimum of the smoothness of $g_1$ and the smoothness of $g_2$. Hence, the result of Theorem \ref{adaptive:thm:pointwise-risk.upper-bound:g} guarantees the near
optimal  rate of our adaptive estimator provided that $F_X$ is known.   
\section{Numerical simulations}
\label{sec:regression.circular:numerical.simulation}
In this section, we implement some simulations to study the numerical performances of our procedure. 
We consider three different regression models:
\begin{align}
\textrm{ M1}. \quad \Theta  &=  \Atan \big( 2 X - 1, X^{2} + 2 \big)  +  \zeta  \quad  \textrm{(mod $2\pi$)}, \label{regression.circular:numerical.simu:modelM1}
\\
\textrm{ M2}. \quad \Theta  &=  \Atan \big( - 2 X + 1, X^{2} - 1 \big)  +   \zeta  \quad  \textrm{(mod $2\pi$)}, \label{regression.circular:numerical.simu:modelM2}
\\
\textrm{ M3}. \quad \Theta  &=  \arccos \big( X^{5} - 1 \big)  +  3. \arcsin \big( X^{3} - X + 1 \big)  +  \zeta  \quad  \textrm{(mod $2\pi$)},
\label{regression.circular:numerical.simu:modelM3}
\end{align}
where the circular error, $\zeta$, is distributed according to a von Mises distribution $f_{v M(0,10)}$ (see \eqref{vmdensity})  and is independent from $X$. 

In the sequel, for models M1 and M2, we consider two cases: 
$X \sim U([-5,5])$ and $X \sim \mathcal{N}(0,1.5)$. For model M3, we consider $X \sim U([0,1])$.
Then, for different values of $n$, we draw a sample $\big(\Theta_i , X_i\big)_{i=1,\ldots,n}$ with the same distribution as $(\Theta,X)$. To implement the Goldenshluger-Lepski methodology, we shall consider either  the Gaussian kernel defined by $y\longmapsto K(y) = \dfrac{1}{\sqrt{2 \pi}}. e^{-\frac{y^2}{2}},$ or the Epanechnikov kernel $K$ defined by
$ y\longmapsto K(y) = \dfrac{3}{4}. (1 - y^2). \ind_{|y| \leq 1}.$ Moreover, we consider the following collection of bandwidths $\mathcal{H}_{n}$ defined as
$$\mathcal{H}_{n} := \left\{ k^{-1} : k \in \N , 1 \leq k \leq \dfrac{n}{\log(n)} \right\}.$$
Finally, we make simulations for the general case of unknown design distribution, i.e. the final estimators are computed using $y\in\R\longmapsto \widehat{F}_{n}(y) := \dfrac{1}{n} \sum_{j=1}^{n} \ind_{X_{j} \leq y}$ instead of $F_{X}$.
\subsection{Practical calibration of tuning parameters}
In the bandwidth selection procedure described in Section~\ref{sec:regression.circular:data-driven.estimator}, we need to tune two parameters $c_{0,1}$ and $c_{0,2}$ in order to find an optimal value of the pointwise risk
\begin{equation}\label{def:risk}
\mathcal{R} := 1 - \cos \big( \widehat{m}_{\widehat{h}}(x) , m(x) \big) , 
\end{equation}  
with $\widehat{m}_{\widehat{h}}(x) = \Atan \big( \widehat{g}_{1,\widehat{h}_{1}} (\widehat{F}_{n}(x)) , \widehat{g}_{2,\widehat{h}_{2}} (\widehat{F}_{n}(x)) \big)$.
To do this, we implement preliminary simulations to calibrate $c_{0,1}$ and $c_{0,2}$ by only considering model M1. Figure~\ref{fig:simulation:model1.ex4:sample_test1} displays an illustration of our setting.
\vspace{-0.4cm}
\begin{figure}[h!]
	\centering
	\begin{minipage}{.45\linewidth}
		\hspace{-0.8cm}
		\includegraphics[scale=0.22]{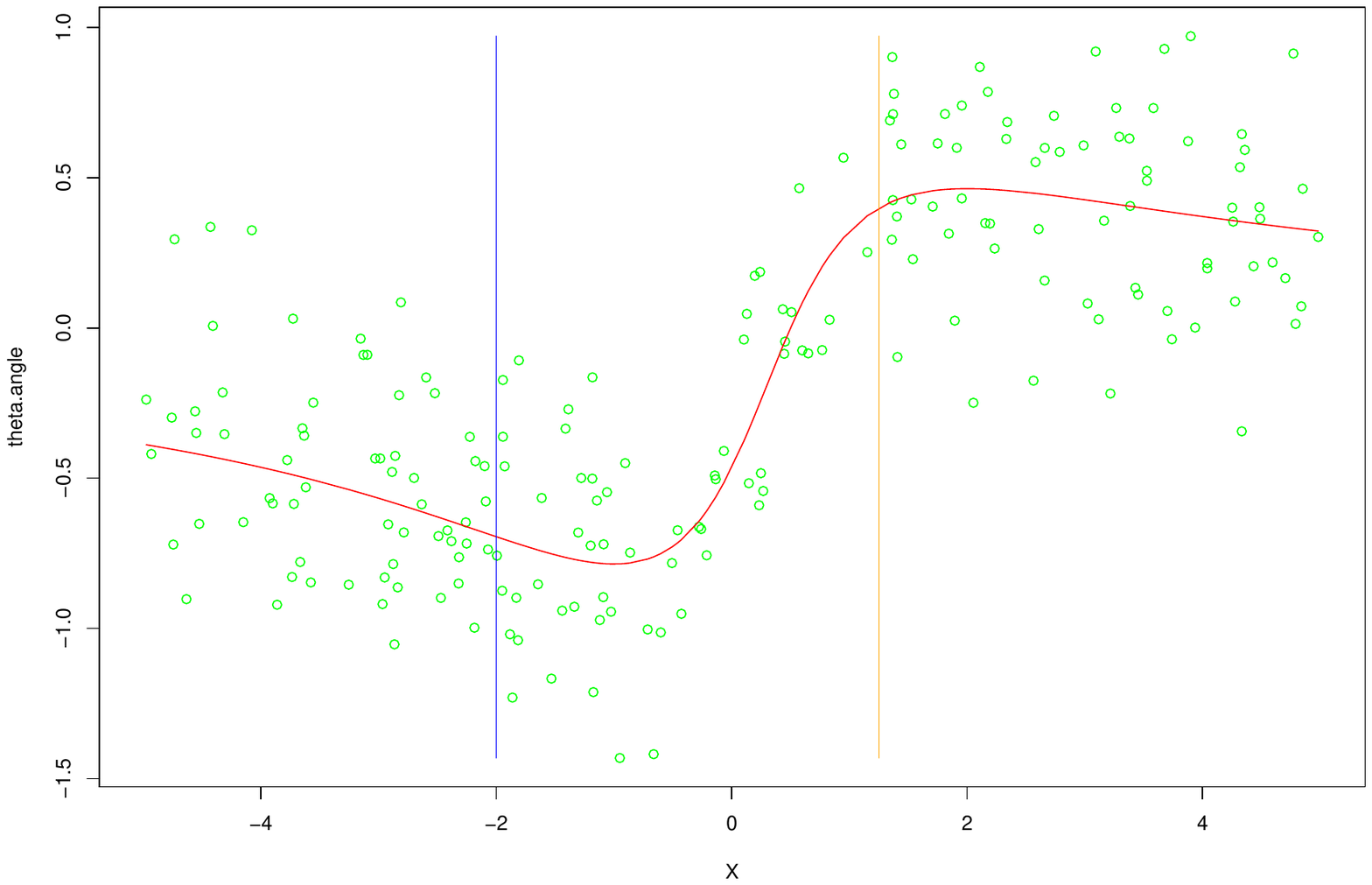}
		\subcaption{\scriptsize{ Model M1 : $X_i \sim U([-5,5])$ }}	
	\end{minipage}
	\hspace{0.1cm}
	\begin{minipage}{.45\linewidth}
		\centering
		\includegraphics[scale=0.22]{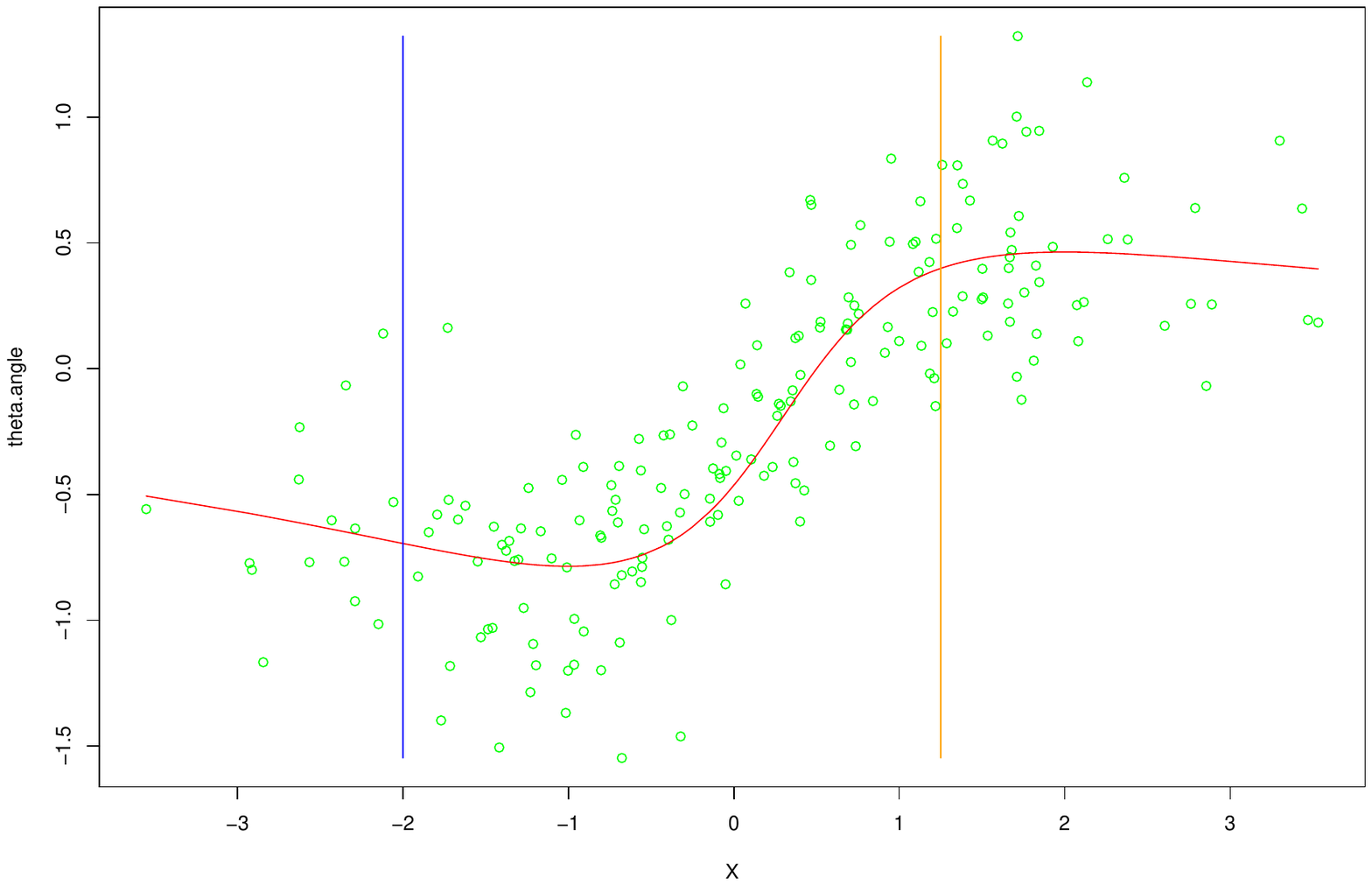}	
		\subcaption{\scriptsize{ Model M1 : $X_i \sim \mathcal{N} (0 , 1.5)$ }}
	\end{minipage}
	\vspace{-0.4cm}   
	\caption{Illustration of model M1 with $n=200$ for two different density functions of the design. Simulated data $(\Theta_i )_{i=1}^{n}$ are displayed in green points. The red curve represents the regression function $m$, while the blue vertical line displays the point $x = -2$ and the orange vertical line displays the point $x = 1.25$ where we aim at estimating $m(x)$.}\label{fig:simulation:model1.ex4:sample_test1}
\end{figure}
\subsubsection{The case $c_{0,1} = c_{0,2}$}
\vspace{-0.2cm}
To select $\widehat{h}_{1}$ and $\widehat{h}_{2}$, we  first consider the case $c_{0,1} = c_{0,2} = c_{0}$. For different sample sizes $n \in \left\{ 100 ; 200; 500; 1000  \right\}$, we compute the risk $\mathcal{R}$ defined in \eqref{def:risk} as a function of $c_{0}$ on the following discretization grid
\begin{equation*}
G_{c_0} := \left\{ 0.001 ; \hspace{0.05cm} 0.0025; \hspace{0.05cm} 0.005; \hspace{0.05cm} 0.0075; \hspace{0.05cm} 0.01; \hspace{0.05cm} 0.025; \hspace{0.05cm} 0.05; \hspace{0.05cm} 0.075 ; \hspace{0.05cm} 0.1 ; \hspace{0.05cm} 0.2 ; \hspace{0.05cm} 0.3 ; \hspace{0.05cm} 0.4 \right\}.
\end{equation*}
We denote $\mathcal{R}\equiv\mathcal{R}(c_0)$.
The numerical illustrations are  displayed in Figure~\ref{fig:simulation:regression.circular:model1.ex4.crit.find.c01=c02.XGauss.XUnif.x=-2} for $x = -2$ and in Figure~\ref{fig:simulation:regression.circular:model1.ex4.crit.find.c01=c02.XGauss.XUnif:x=1.25} for $x = 1.25$, respectively.
\begin{figure}[h!]
	\begin{minipage}{.45\linewidth}
		\hspace{-0.8cm}
		\includegraphics[scale=0.21]{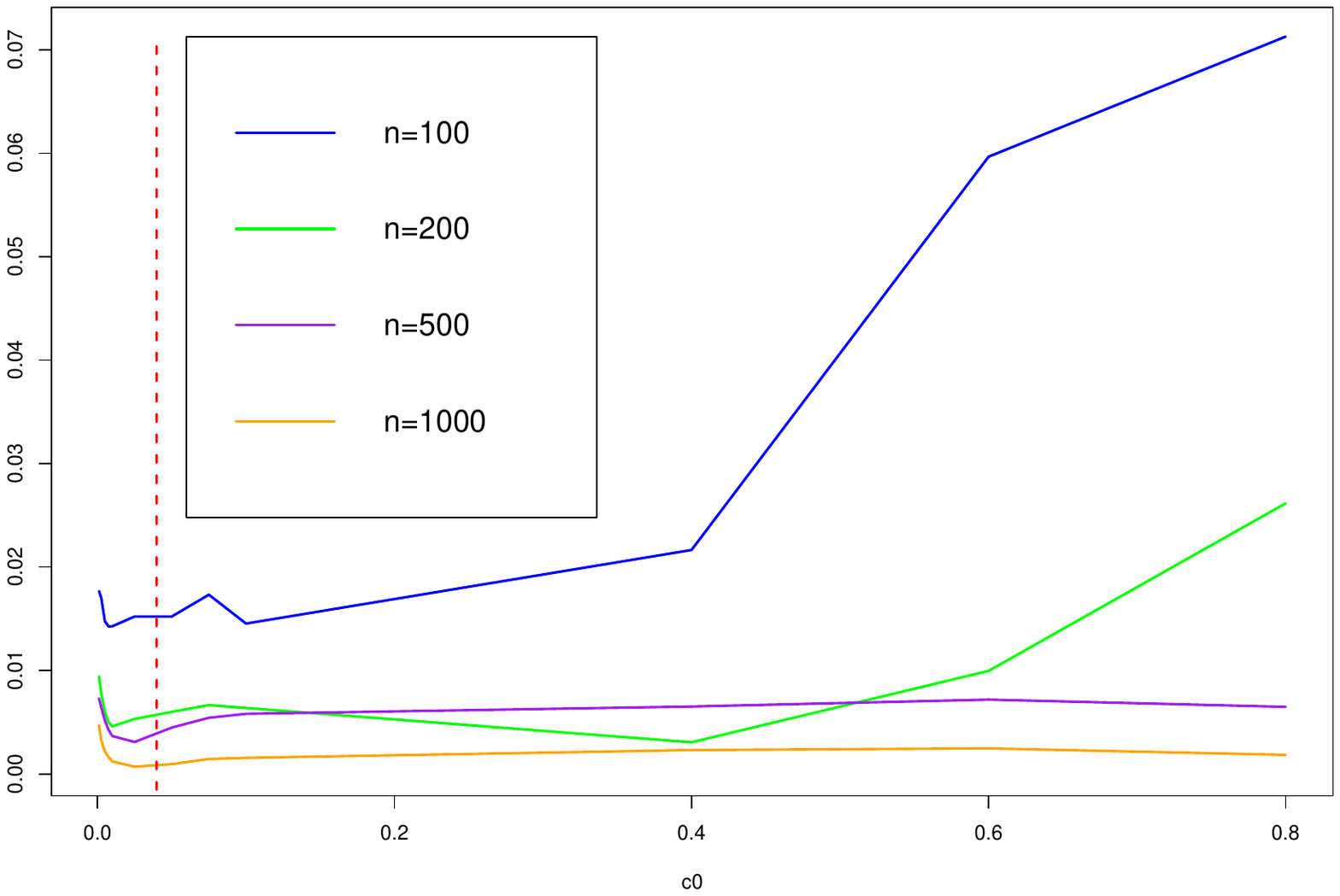}
		\subcaption{\scriptsize{$X_i \sim \mathcal{N}(0,1.5)$}}
	\end{minipage}
	\hspace{0.001cm}	
	\begin{minipage}{.45\linewidth}
		\includegraphics[scale=0.21]{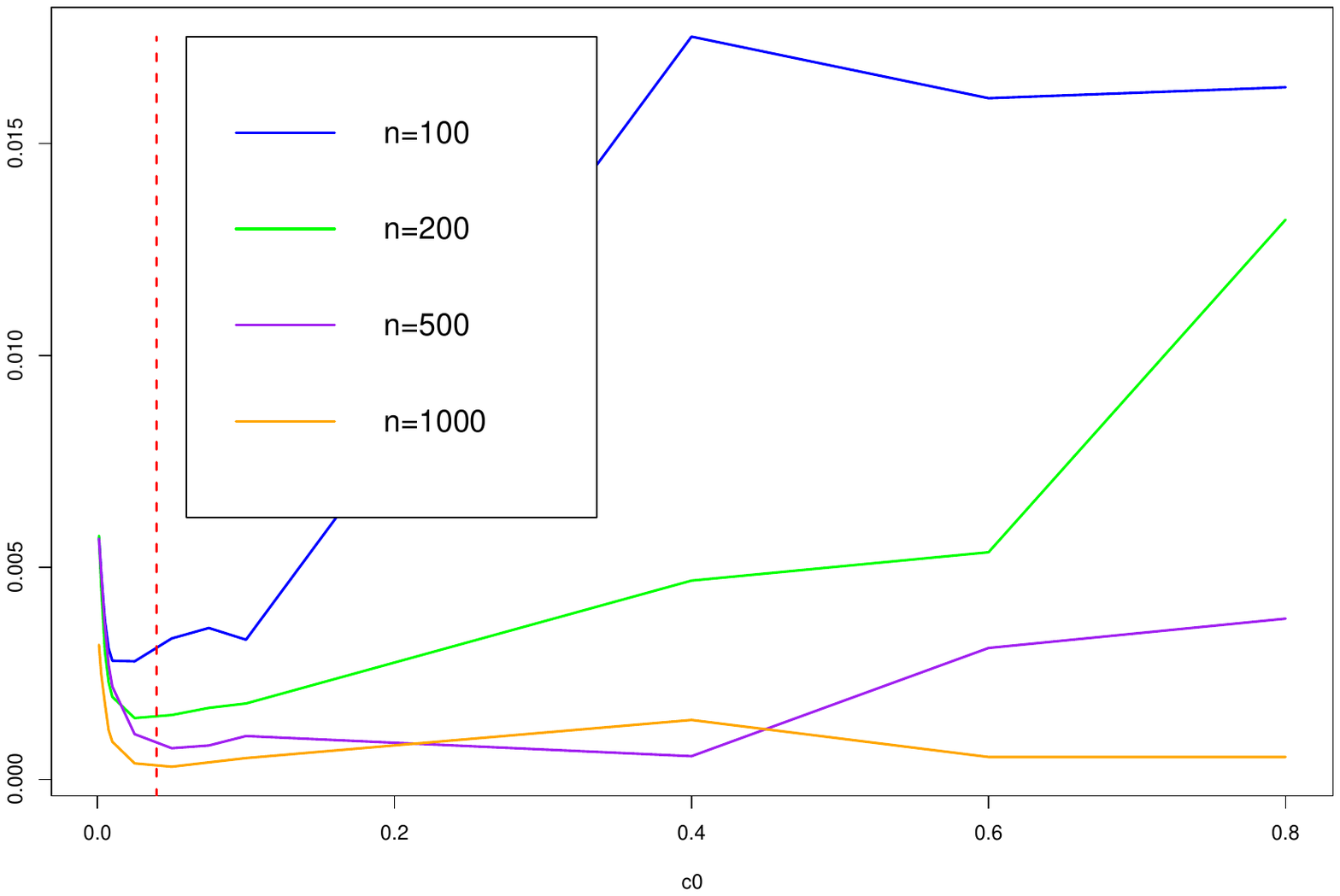}
		\subcaption{\scriptsize{$X_i \sim U([-5,5])$}}
	\end{minipage}
	\vspace{-0.4cm}  
	\caption{Model M1. Plot of the Monte Carlo estimation of the function $c_{0} \in G_{c_{0}}\longmapsto\mathcal{R}(c_0)$, based on 50 runs, for $x=-2$, the Gaussian and the uniform designs and by using the Epanechnikov kernel for $n \in \left\{ 100 ; 200; 500; 1000  \right\}$ corresponding to the line in blue, green, violet and orange, respectively. The red vertical line displays the point $c_0 =0.04$.}
	\label{fig:simulation:regression.circular:model1.ex4.crit.find.c01=c02.XGauss.XUnif.x=-2}
\end{figure}
\begin{figure}[h!]
	\begin{minipage}{.45\linewidth}
		\hspace{-0.8cm}
		\includegraphics[scale=0.21]{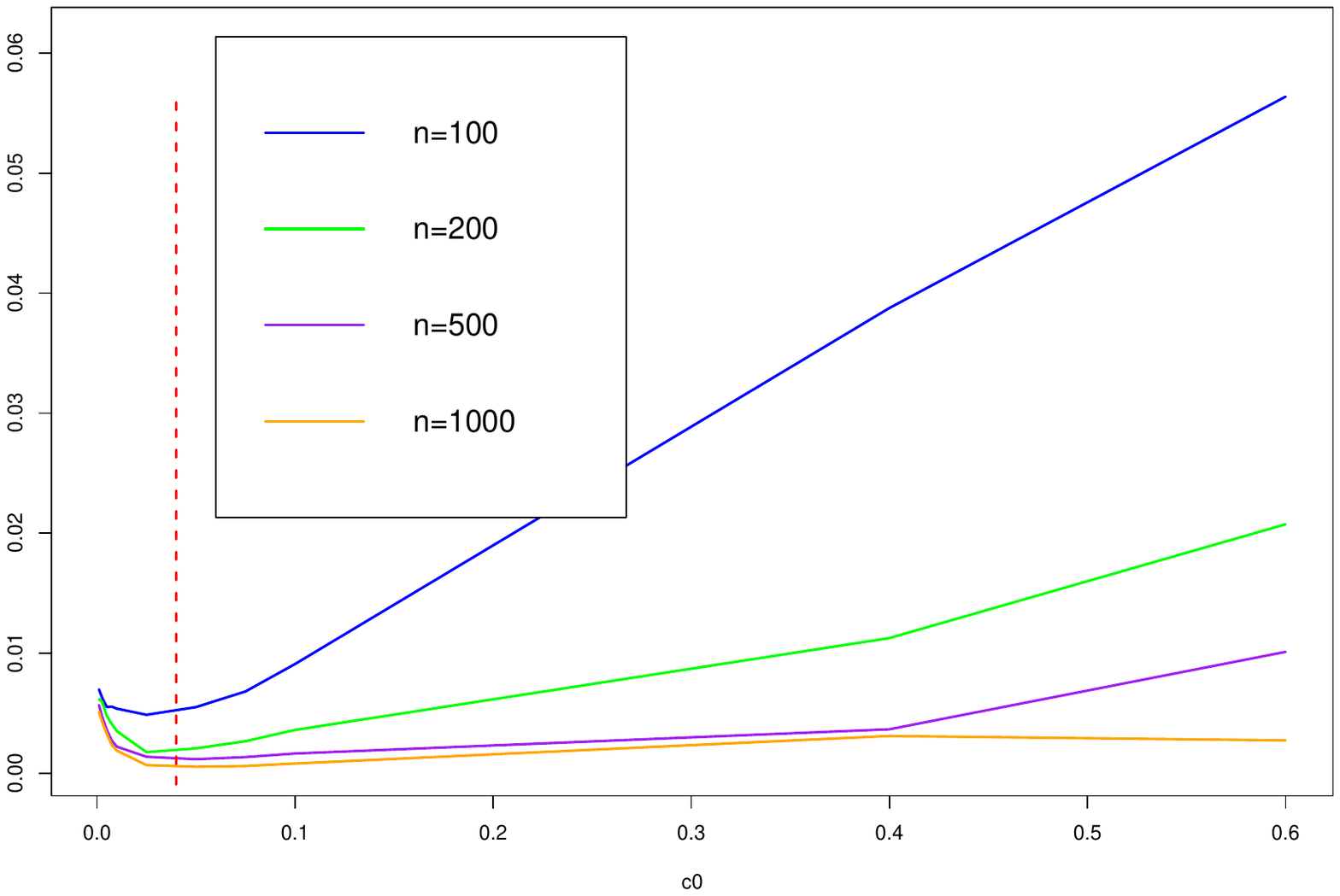}
		\subcaption{\scriptsize{$X_i \sim \mathcal{N}(0,1.5)$}}
	\end{minipage}
	\hspace{0.001cm}
	\begin{minipage}{.45\linewidth}
		\includegraphics[scale=0.21]{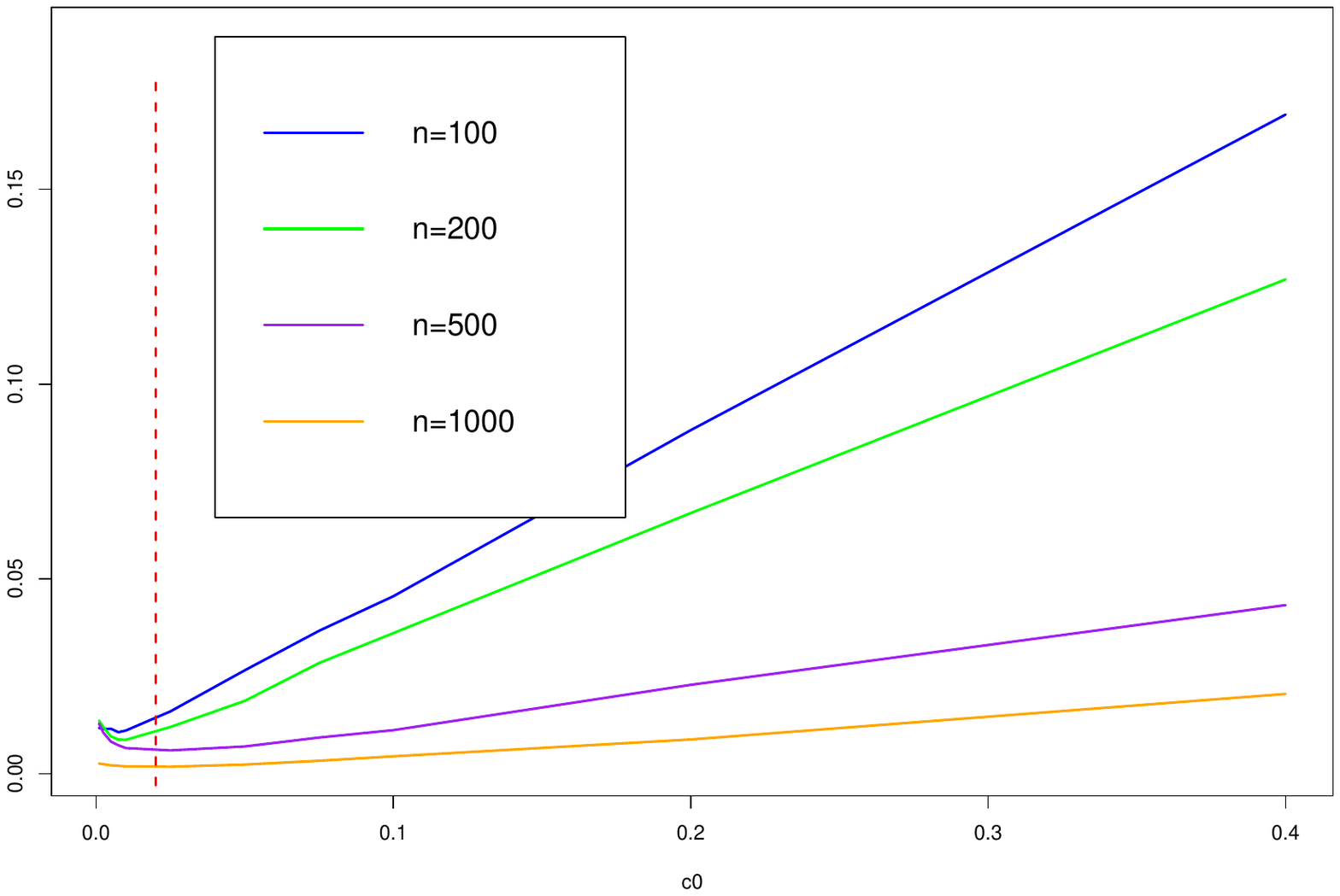}  
		\subcaption{\scriptsize{$X_i \sim U([-5,5])$}}
	\end{minipage}
	\vspace{-0.4cm}  
	\caption{Model M1. Plot of the Monte Carlo estimation of the function $c_{0} \in G_{c_{0}}\longmapsto\mathcal{R}(c_0)$, based on 50 runs, for $x=1.25$, the Gaussian and the uniform designs and by using the Epanechnikov kernel for $n \in \left\{ 100 ; 200; 500; 1000  \right\}$ corresponding to the line in blue, green, violet and orange, respectively. The red vertical line displays the point $c_0 =0.04$.}
	\label{fig:simulation:regression.circular:model1.ex4.crit.find.c01=c02.XGauss.XUnif:x=1.25}
\end{figure}
\leavevmode \\[0.2cm]  
To further study an influence of the kernel rule, we consider the Gaussian kernel. The associated numerical illustrations are provided in Figure~\ref{fig:simulation:regression.circular:GaussKernel:model1.ex4.crit.find.c01=c02.XUnif-XGauss.x=-2}  
for $X \sim U([-5,5])$ and 
for $X \sim \mathcal{N}(0,1.5)$. 
\begin{figure}[h!]
	\begin{minipage}{.45\linewidth}
		\hspace{-0.8cm}
		\includegraphics[scale=0.21]{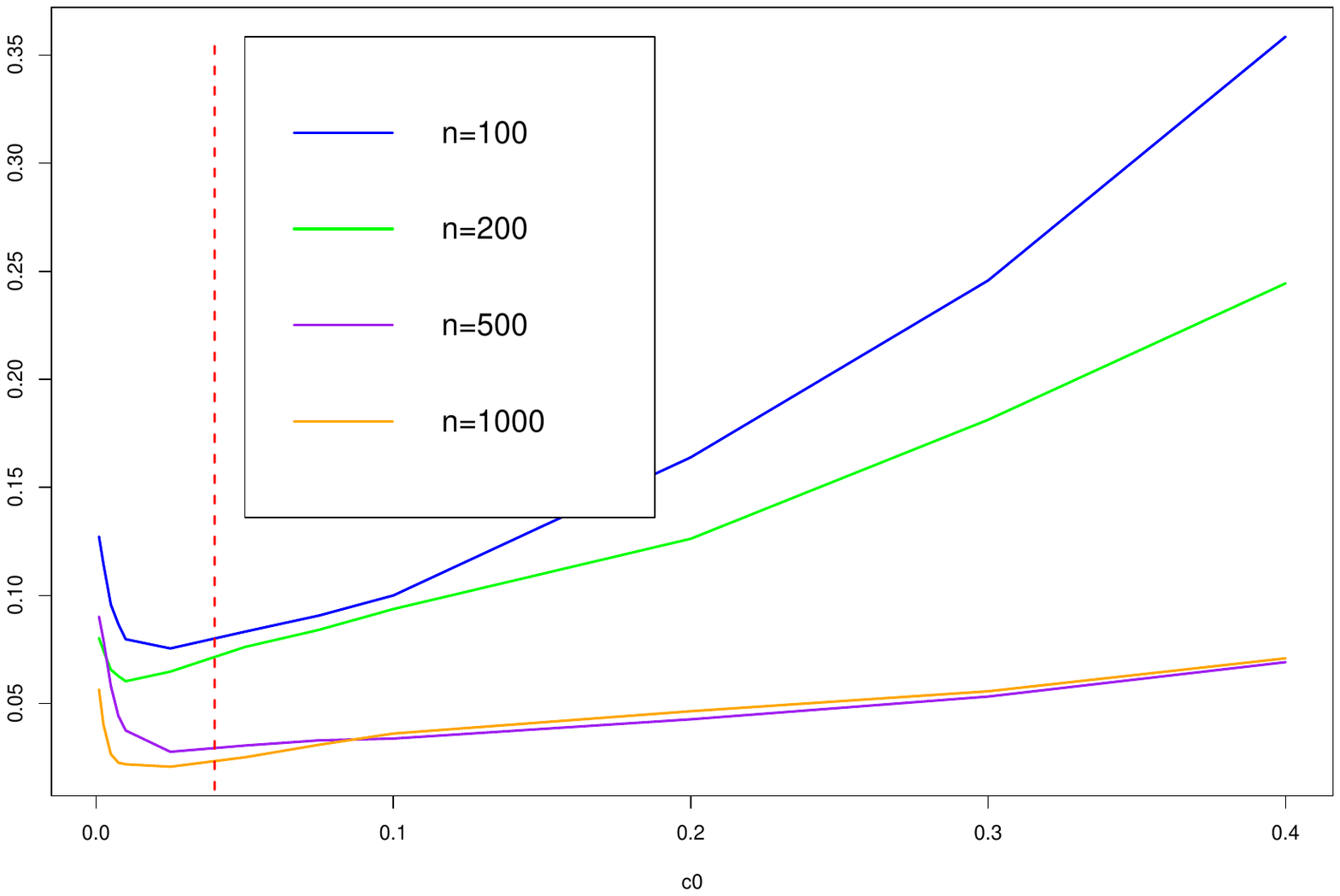}
		\subcaption{\scriptsize{$X_i \sim U([-5,5])$. }}
	\end{minipage}
	\hspace{0.01cm}
	\begin{minipage}{.45\linewidth}
		\includegraphics[scale=0.21]{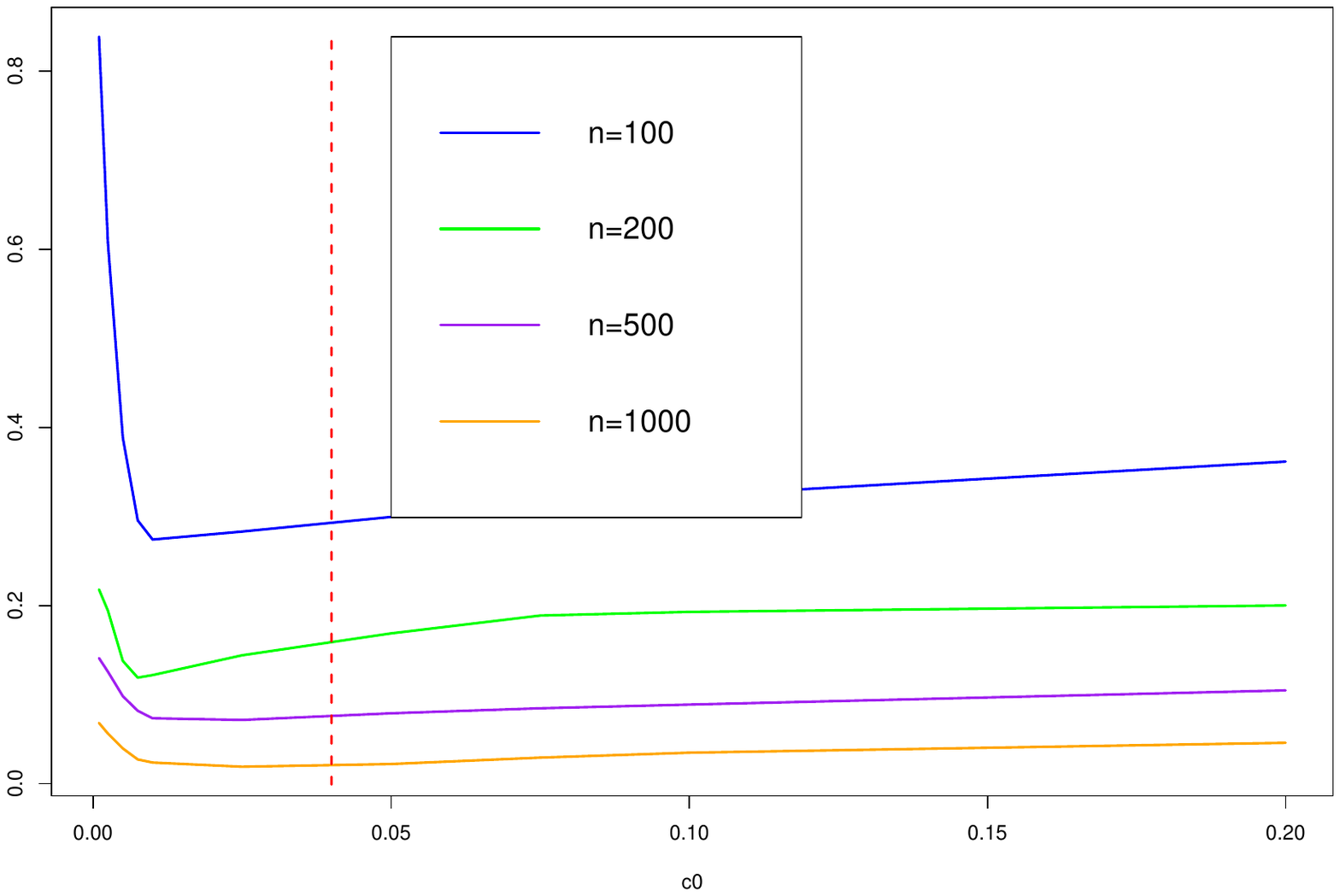}
		\subcaption{\scriptsize{$X_i \sim \mathcal{N}(0,1.5)$. }}
	\end{minipage}
	\vspace{-0.4cm}  
	\caption{Model M1. Plot of the Monte Carlo estimation of the function $c_{0} \in G_{c_{0}}\longmapsto\mathcal{R}(c_0)$, based on 50 runs, for $x=-2$, the Gaussian and the uniform designs and by using the Gaussian kernel for $n \in \left\{ 100 ; 200; 500; 1000  \right\}$ corresponding to the line in color of blue, green, violet and orange, respectively.}
	\label{fig:simulation:regression.circular:GaussKernel:model1.ex4.crit.find.c01=c02.XUnif-XGauss.x=-2}
\end{figure}
This brief numerical study shows that the choice $c_{0,1} = c_{0,2} = 0.04$ is convenient for each numerical scheme.
\subsubsection{The case $c_{0,1}\not= c_{0,2}$}
We do no longer assume that $c_{0,1}= c_{0,2}$.
For $n = 200$, we compute the risk $\mathcal{R}$ defined in \eqref{def:risk} as a function of $(c_{0,1} , c_{0,2})$ on the following discretization grid
\begin{equation*}
G_{c_0} := \left\{ 0.001 ; \hspace{0.1cm} 0.005;  \hspace{0.1cm} 0.01; \hspace{0.1cm} 0.025; \hspace{0.1cm} 0.05; \hspace{0.1cm} 0.075 ; \hspace{0.1cm} 0.1 ; \hspace{0.1cm} 0.2 ; \hspace{0.1cm} 0.3 ; \hspace{0.1cm} 0.4  
\right\}.
\end{equation*}
We denote $\mathcal{R}\equiv\mathcal{R} (c_{0,1} , c_{0,2})$. The associated numerical illustrations are provided in Figure~\ref{fig:calibration:c01-c02:risk:model1_ex4_XGauss.n200.MC50} and Figure~\ref{fig:calibration:c01-c02:risk:model1_ex4_XUnif.n200.MC50} for the case $X \sim \mathcal{N}(0,1.5)$ and $X\sim U([-5 , 5])$,  respectively. 
\begin{figure}[h!]
	\hspace{-1.5cm}
	\begin{minipage}{.52\linewidth}
		\includegraphics[scale=0.22]{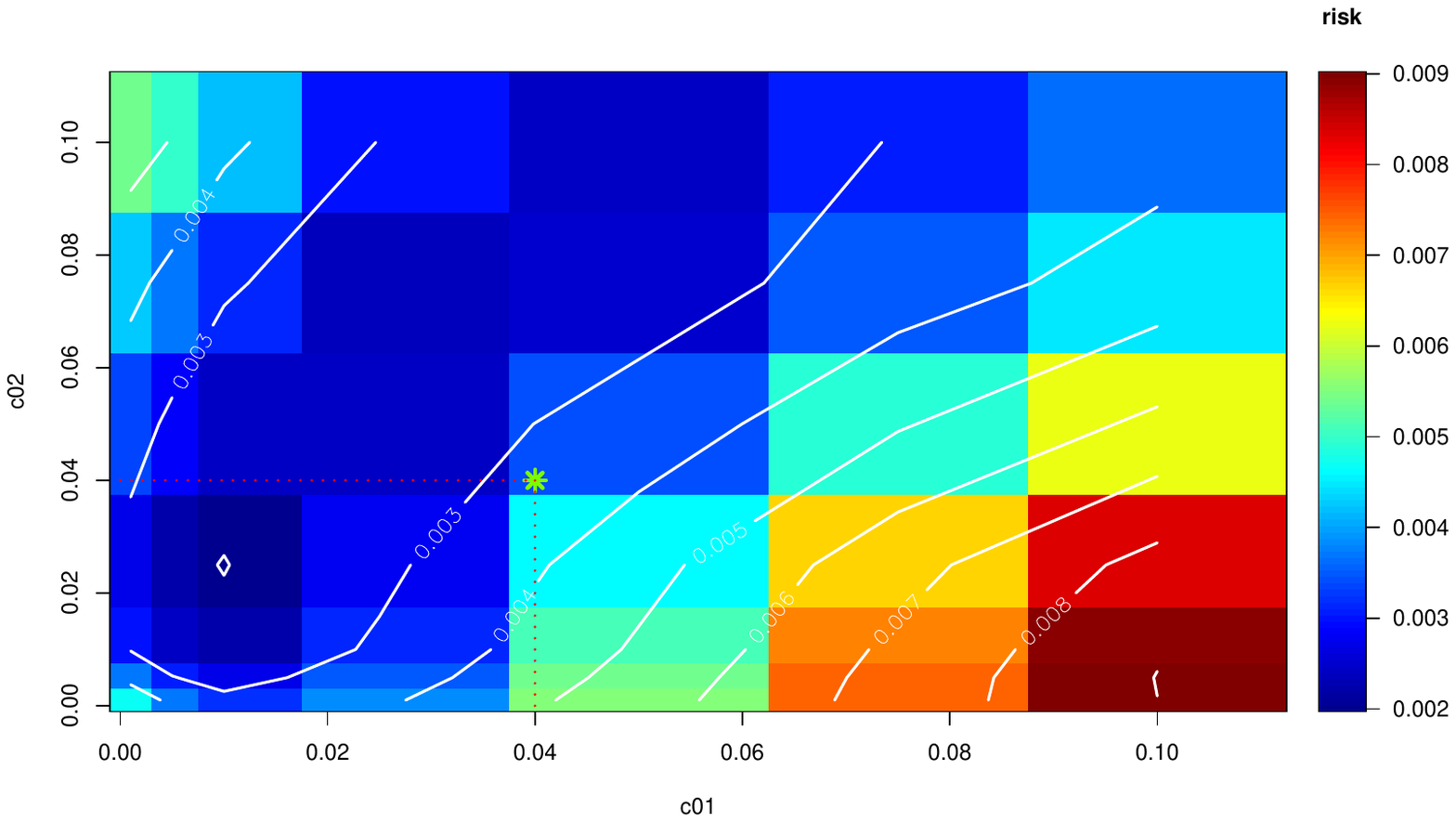} 
		\vspace{-1cm}
		\subcaption{\scriptsize{estimation at $x = -2$}}  
	\end{minipage}
	\begin{minipage}{.45\linewidth}
		\hspace{-0.2cm}
		\includegraphics[scale=0.22]{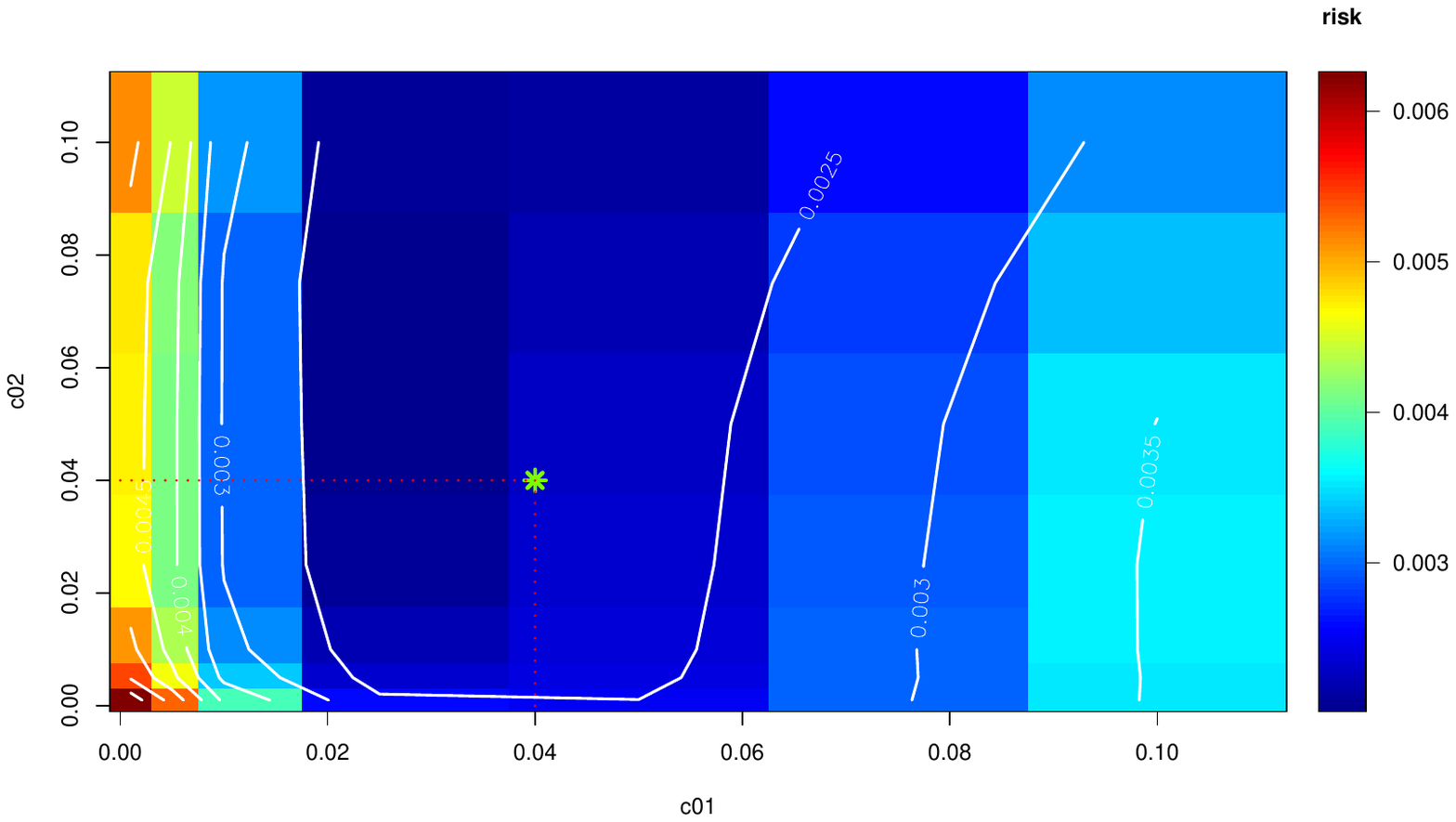}
		\vspace{-1cm}
		\subcaption{\scriptsize{estimation at $x = 1.25$}}
	\end{minipage}
	\caption{Model M1. 2D-representation of the Monte Carlo estimation of the function $(c_{0,1},c_{0,1}) \in G_{c_{0}}\times G_{c_{0}}\longmapsto\mathcal{R}(c_{0,1} , c_{0,2})$, based on 50 runs, with $X_i \sim \mathcal{N}(0,1.5)$, $n=200$ at $x = -2$ and $x = 1.25$ by using the Epanechnikov kernel. We display the specific point $c_{0,1} = c_{0,2} = 0.04$.}
	\label{fig:calibration:c01-c02:risk:model1_ex4_XGauss.n200.MC50}
\end{figure}
\begin{figure}[h!]
	\hspace{-1.5cm}
	\begin{minipage}{.52\linewidth}
		\includegraphics[scale=0.22]{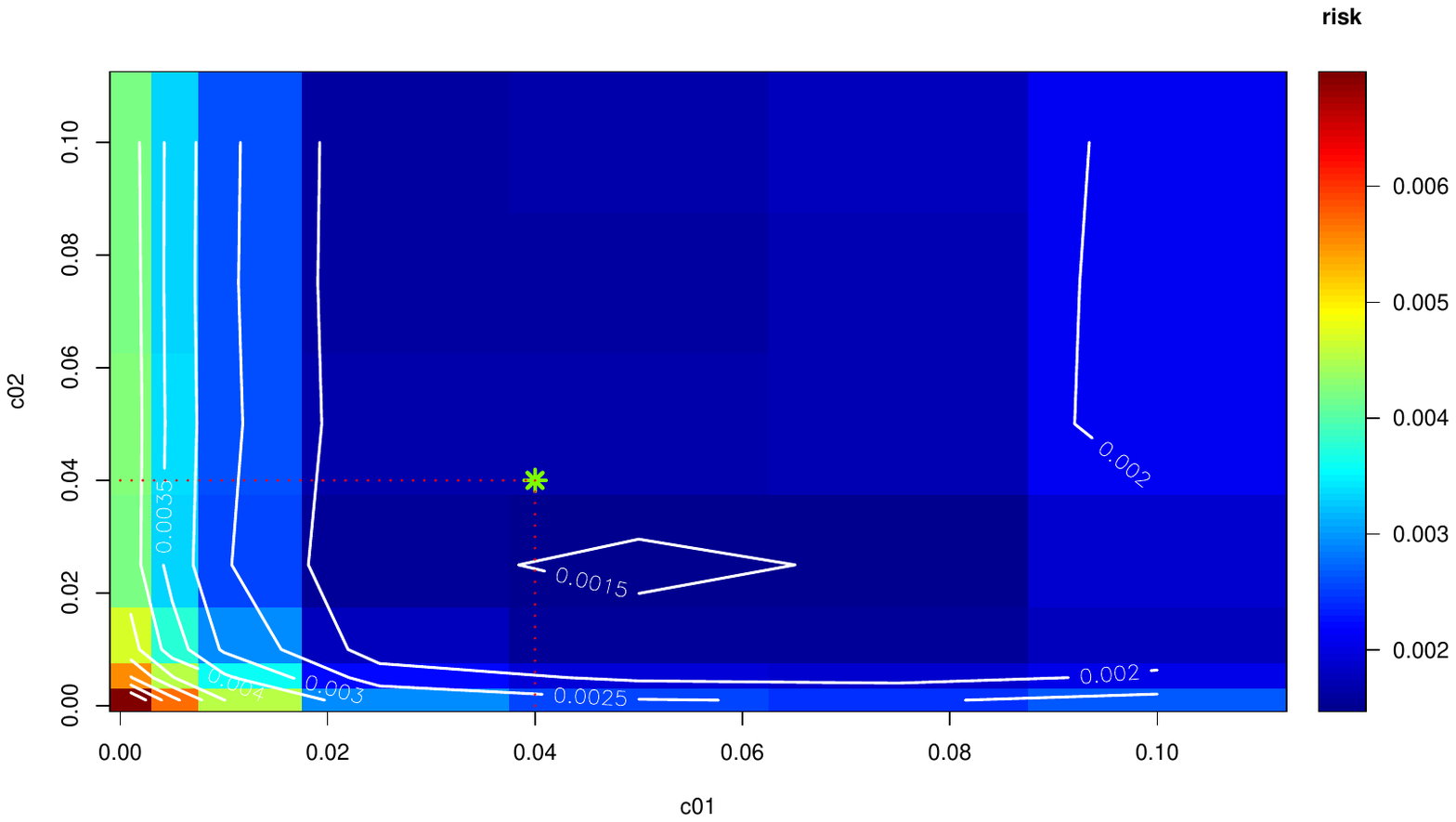}
		\vspace{-1cm}
		\subcaption{\scriptsize{estimation at $x = -2$}}
	\end{minipage}
	\begin{minipage}{.45\linewidth}
		\hspace{-0.2cm}
		\includegraphics[scale=0.22]{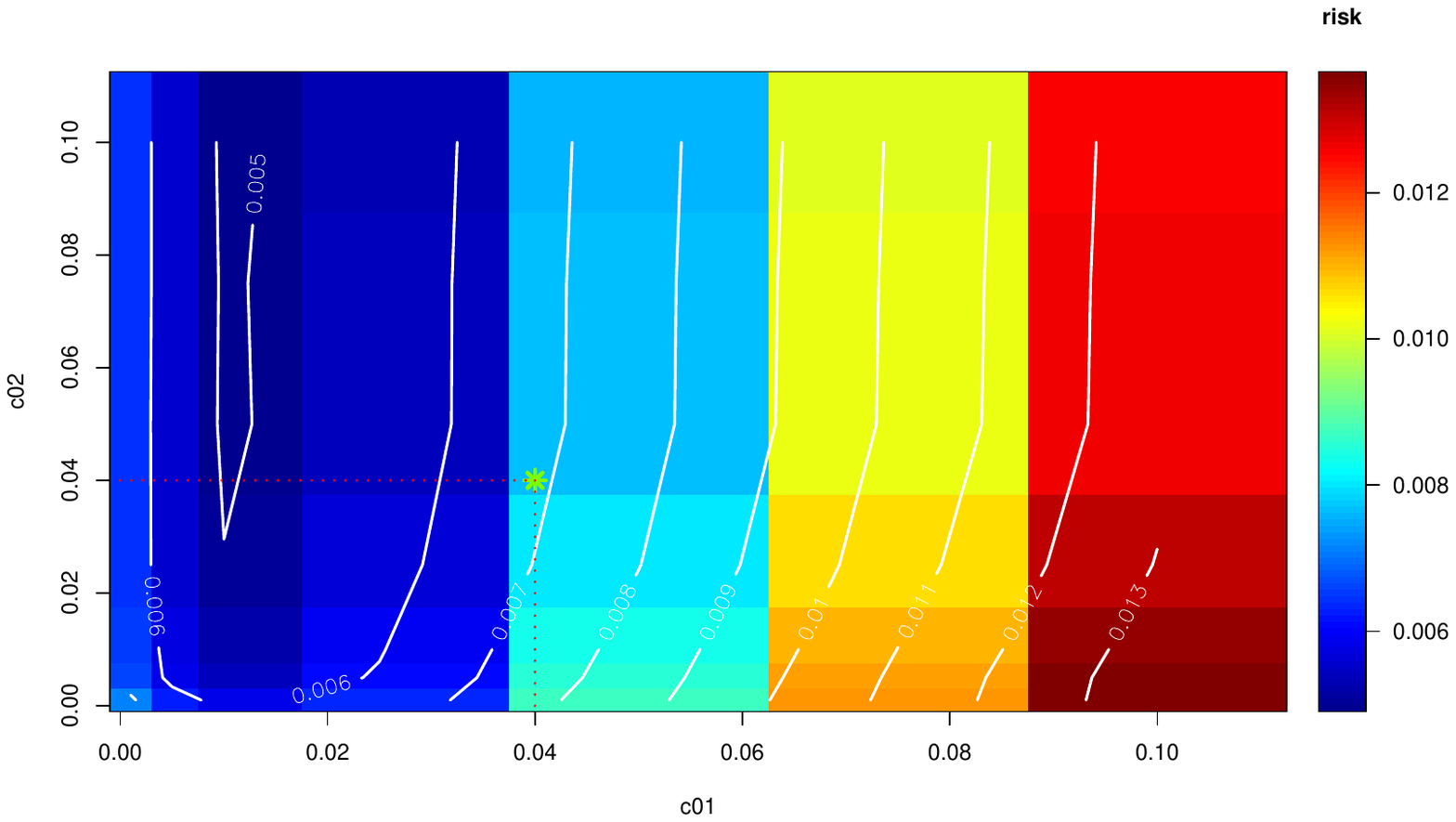}
		\vspace{-1cm}
		\subcaption{\scriptsize{estimation at $x = 1.25$}}
	\end{minipage}
	\caption{Model M1. 2D-representation of the Monte Carlo estimation of the function $(c_{0,1},c_{0,1}) \in G_{c_{0}}\times G_{c_{0}}\longmapsto\mathcal{R}(c_{0,1} , c_{0,2})$, based on 50 runs with $X_i \sim U([-5 , 5])$, $n=200$ at $x = -2$ and $x = 1.25$ by using the Epanechnikov kernel. We display the specific point $c_{0,1} = c_{0,2} = 0.04$.}
	\label{fig:calibration:c01-c02:risk:model1_ex4_XUnif.n200.MC50}
\end{figure}
\leavevmode \\[0.1cm]
Even if it is not the best one, the choice of $c_{0,1} = c_{0,2} = 0.04$ is reasonable. For sake of simplicity, we fix $c_{0,1} = c_{0,2} = 0.04$ for subsequent numerical simulations.
\subsection{Numerical results}
We now illustrate the numerical performances obtained by our methodology for models M1, M2 and M3 by using the Epanechnikov kernel. They are also compared to other approaches. A similar scheme is conducted by using the Gaussian kernel. Remember that in the following numerical experiments, our estimate is tuned with $c_{0,1} = c_{0,2} = 0.04$. 
We first display several graphs to illustrate numerical performances obtained by our methodology, denoted GL, by using the Epanechnikov kernel.
More precisely, we display boxplots in Figures~\ref{fig:simulation:regression.circular:EpaKernel.model1.ex4.boxplot.n200.XUnif.vsNW.vsLL} and~\ref{fig:simulation:regression.circular:EpaKernel.model1.ex4.boxplot.n200.XGauss.vsNW.vsLL} summarizing our numerical results for model M1 in the case $X \sim U([-5,5])$ and in the case $X \sim \mathcal{N}(0,1.5)$, respectively. In both cases, for model M1, we estimate $m(x)$ at $x = -2$ and at $x = 1.25$.
Figure~\ref{fig:simulation:EpaKernel.model1.ex5:boxplot:GL.vsNW.vsLL} shows simulations for model M2 with $X \sim \mathcal{N}(0,1.5)$ 
and we estimate $m(x)$ at $x = 1.05$.
Figure~\ref{fig:simulation:EpaKernel.model2.ex1:boxplot:GL.vsNW.vsLL} shows simulations for model M3 with $X \sim U([0,1])$ and we estimate $m(x)$ at $x = 0.95$.
\begin{figure}[h!]
	\begin{minipage}{.45\linewidth}
		\hspace{-0.8cm}
		\includegraphics[scale=0.21]{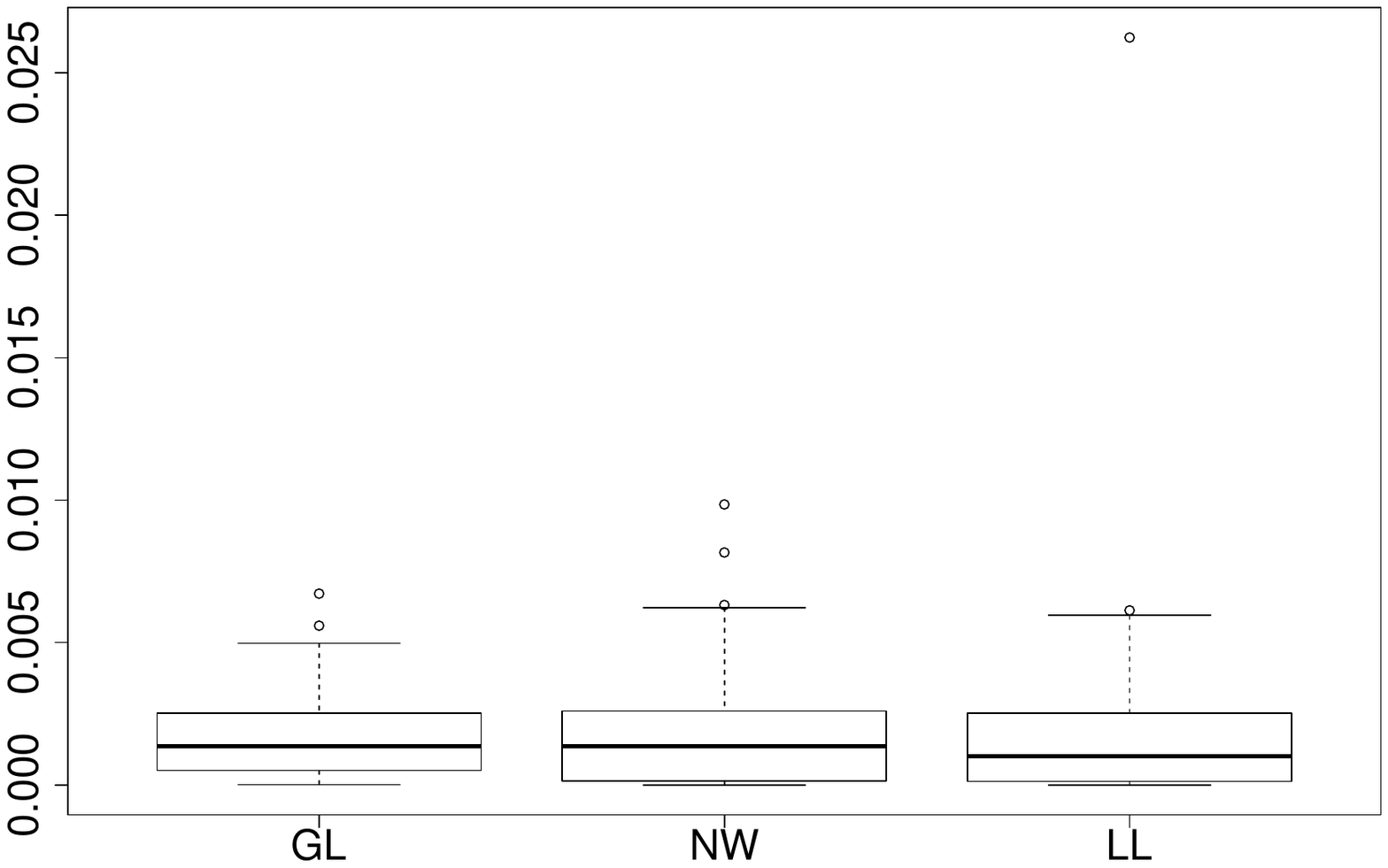}
		\vspace{-0.8cm}
		\subcaption{\scriptsize{estimation at $x = -2$}}
	\end{minipage}
	\hspace{0.1cm}
	\begin{minipage}{.45\linewidth}
		\includegraphics[scale=0.21]{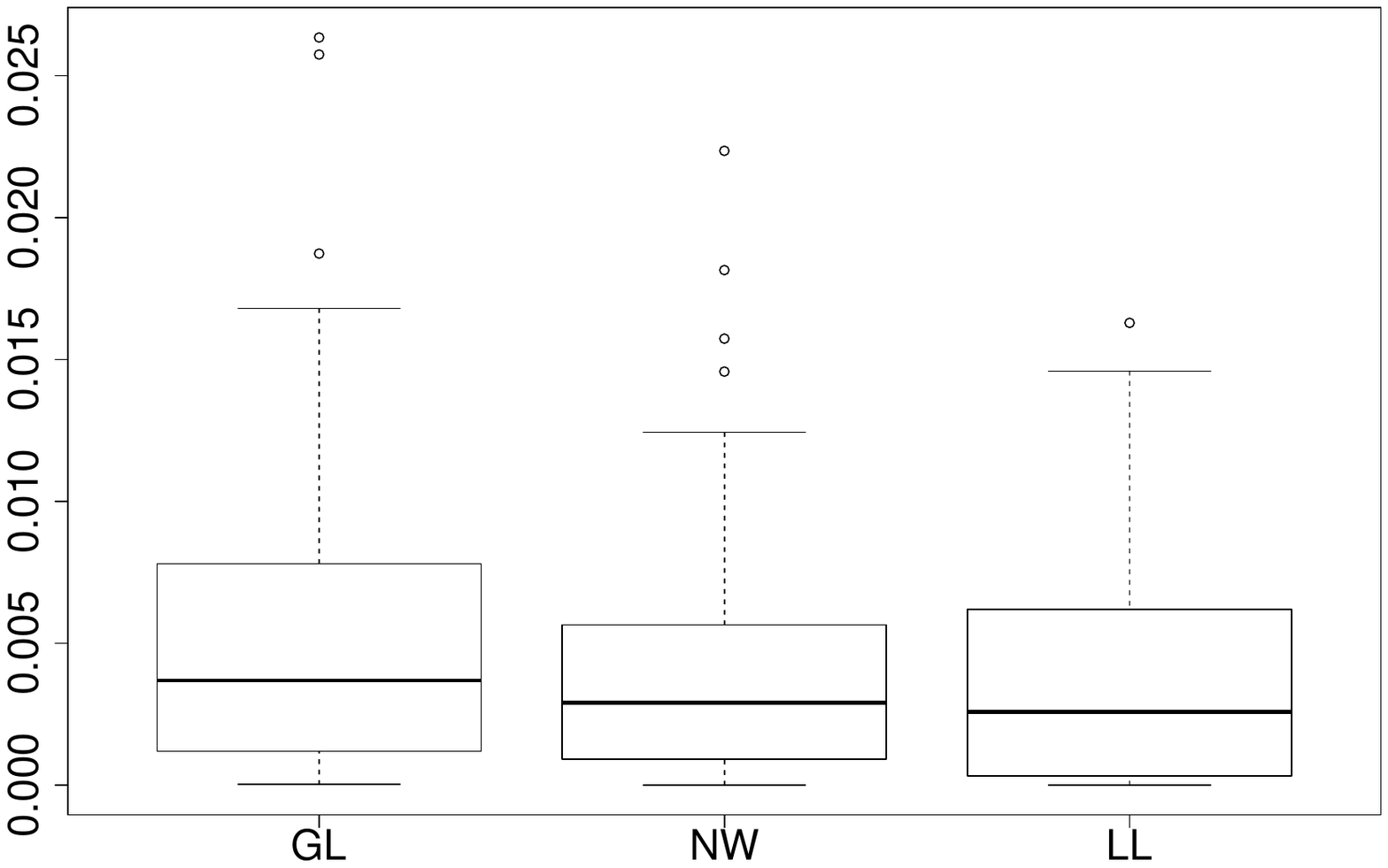}
		\vspace{-0.8cm}
		\subcaption{\scriptsize{estimation at $x = 1.25$}}
	\end{minipage}
\vspace{-0.2cm}
	\caption{Model M1. Boxplots of the estimated risk with $50$ runs for the GL, NW and LL methodologies (from left-hand side to right-hand side on each plot (a) and (b) for $n=200$ and $X \sim U([-5,5])$ by using the Epanechnikov kernel.}
	\label{fig:simulation:regression.circular:EpaKernel.model1.ex4.boxplot.n200.XUnif.vsNW.vsLL}
\end{figure}
\begin{figure}[h!]
	\vspace{-0.5cm}
	\begin{minipage}{.45\linewidth}
		\hspace{-0.8cm}
		\includegraphics[scale=0.21]{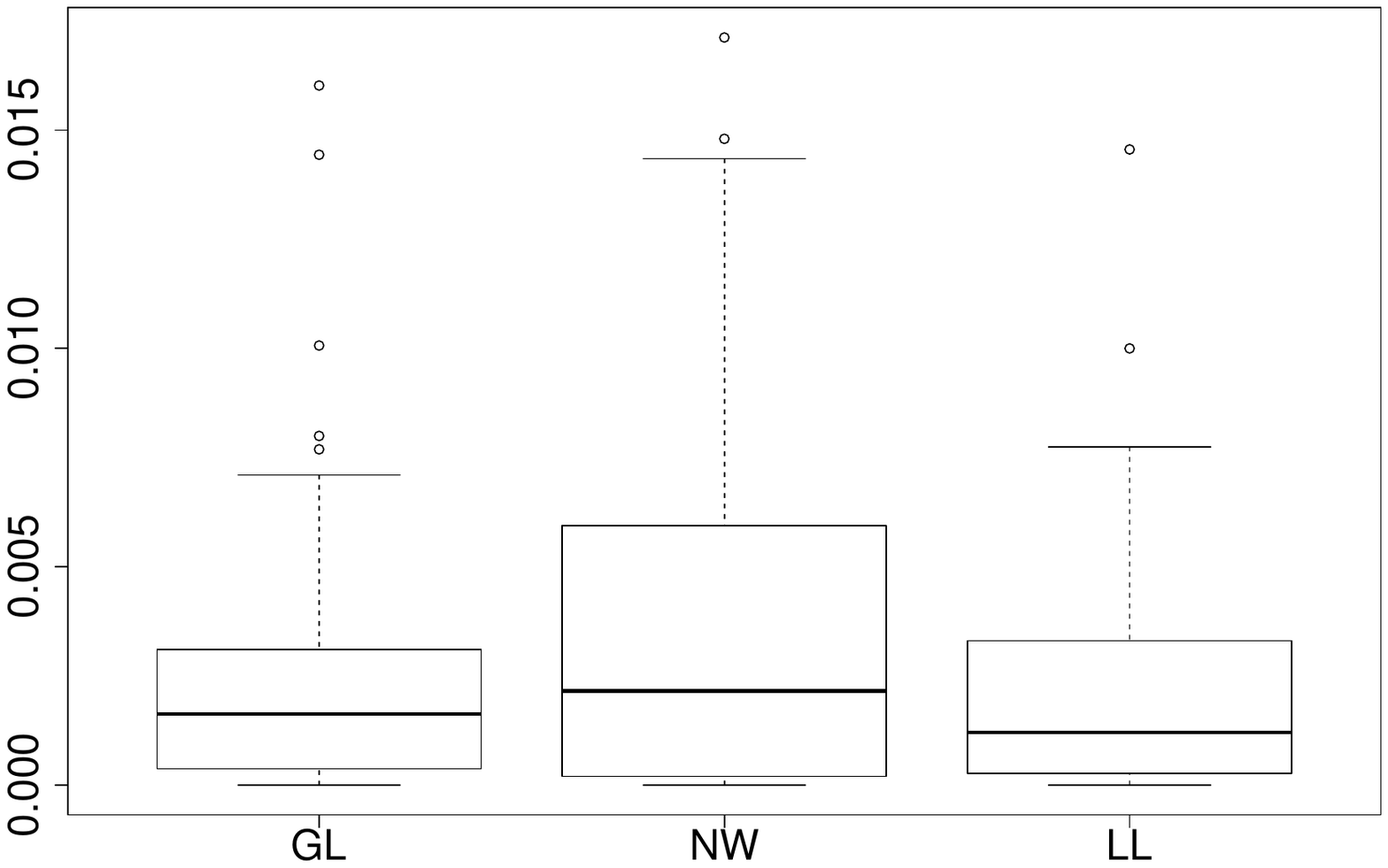}
		\vspace{-0.8cm}
		\subcaption{\scriptsize{estimation at $x = -2$}}
	\end{minipage}
	\hspace{0.1cm}
	\begin{minipage}{.45\linewidth}
		\includegraphics[scale=0.21]{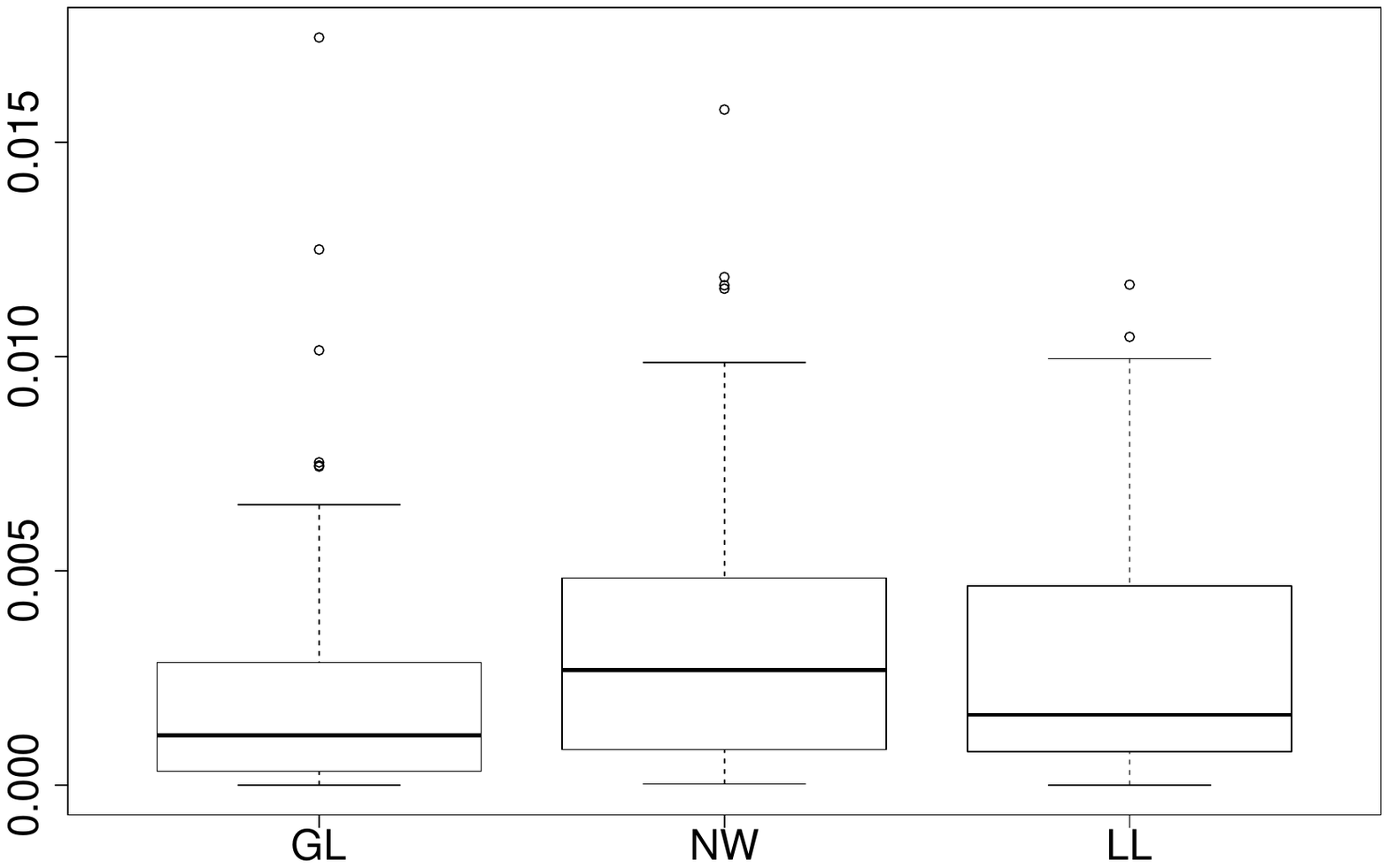}
		\vspace{-0.8cm}
		\subcaption{\scriptsize{estimation at $x = 1.25$}}
	\end{minipage}
\vspace{-0.2cm}
	\caption{Model M1. Boxplots of the estimated risk with $50$ runs for the GL, NW and LL methodologies (from left-hand side to right-hand side on each plot (a) and (b) for $n=200$ and $X \sim \mathcal{N}(0,1.5)$ by using the Epanechnikov kernel.}  
	\label{fig:simulation:regression.circular:EpaKernel.model1.ex4.boxplot.n200.XGauss.vsNW.vsLL}
\end{figure}
\begin{figure}[h!]
	\hspace{-0.5cm}
	\begin{minipage}{.48\linewidth}
		\centering
		\includegraphics[scale=0.21]{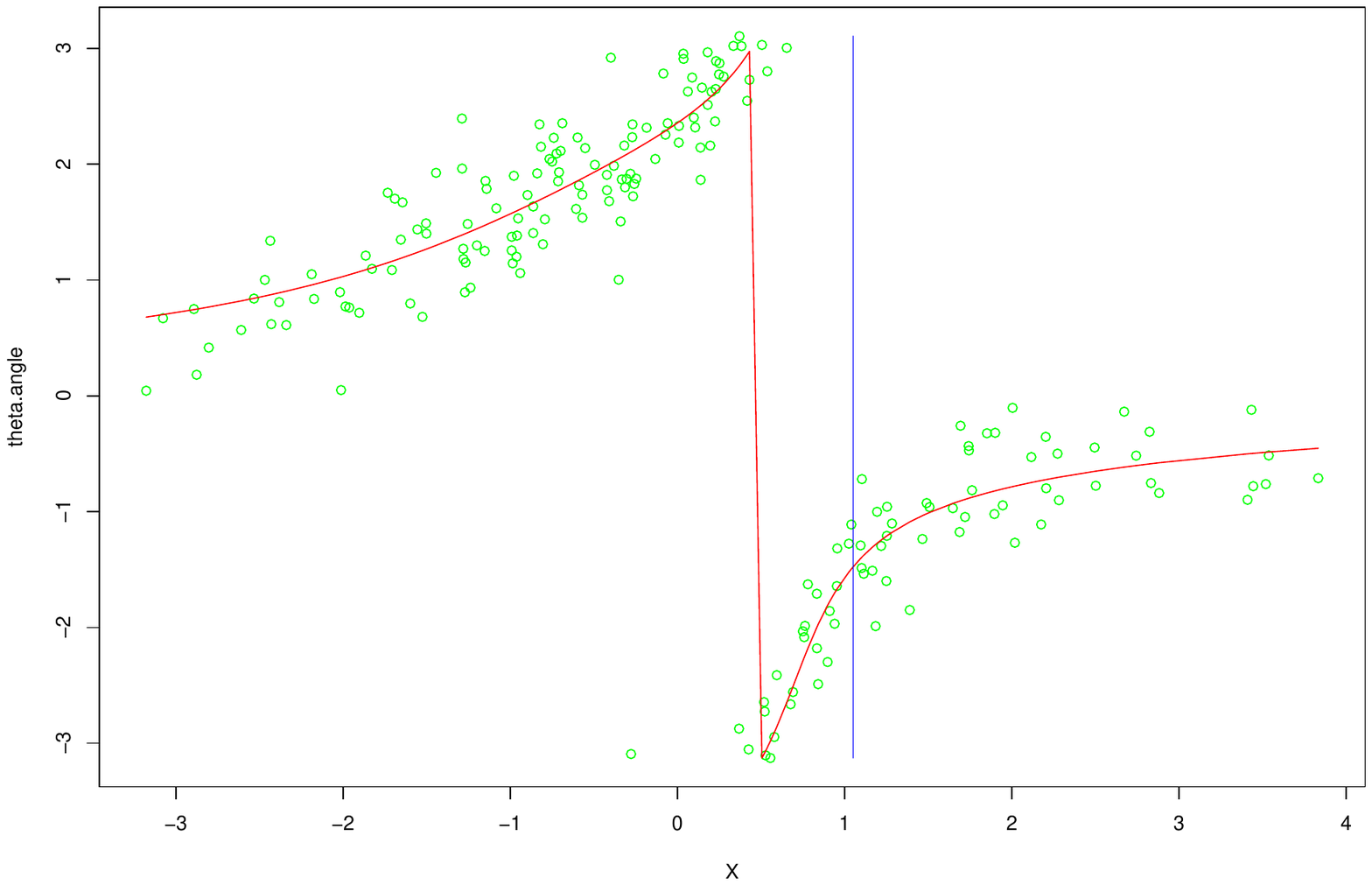}
		\vspace{-0.6cm}
		\subcaption{\scriptsize{$X_i \sim \mathcal{N}(0,1.5)$}}	
		\label{fig:simulation:EpaKernel.model1.ex5:fig.a}
	\end{minipage}
	\hspace{0.2cm}
	\begin{minipage}{.5\linewidth}
		\centering
		\includegraphics[scale=0.21]{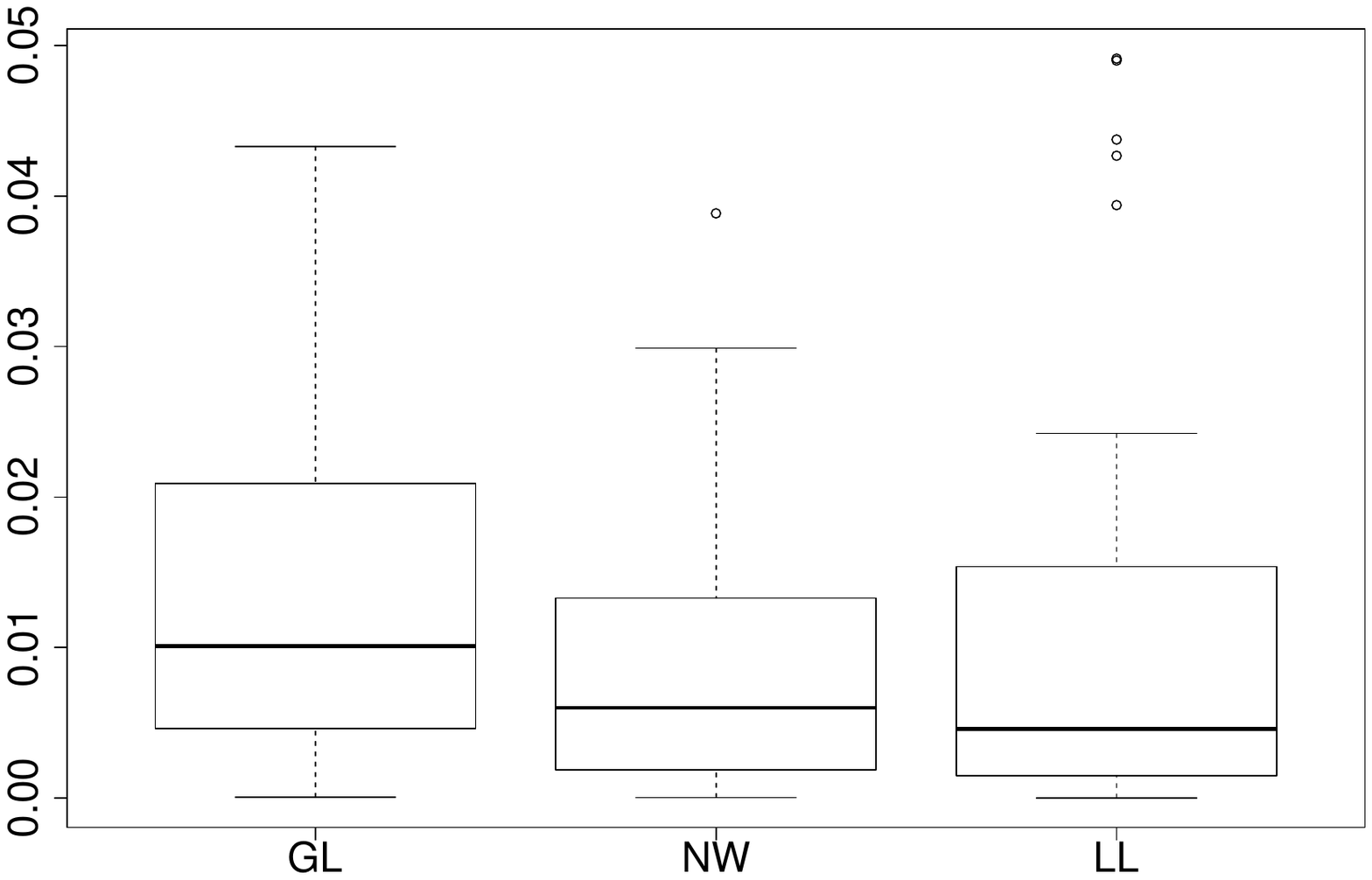}
		\vspace{-0.6cm}
		\subcaption{\scriptsize{Boxplots for GL, NW and LL methodologies}}
		\label{fig:simulation:EpaKernel.model1.ex5:fig.b}
	\end{minipage}
\vspace{-0.2cm}  
	\caption{(a): Model M2. Simulated data $(\Theta_i)_{i=1}^{n}$ of model M2 (green points) with $n=200$ and $X_i \sim \mathcal{N}(0,1.5)$. The red curve represents the true regression function $m$, while the blue vertical line displays the point $x = 1.05$; (b): Boxplots of the estimated risk with $50$ runs for the GL, NW and LL methodologies (from left-hand side to right-hand side by using the Epanechnikov kernel.}	\label{fig:simulation:EpaKernel.model1.ex5:boxplot:GL.vsNW.vsLL}
\end{figure}
\begin{figure}[h!]
	\vspace{-0.5cm}
	\hspace{-0.5cm}
	\begin{minipage}{.48\linewidth}
		\centering
		\includegraphics[scale=0.21]{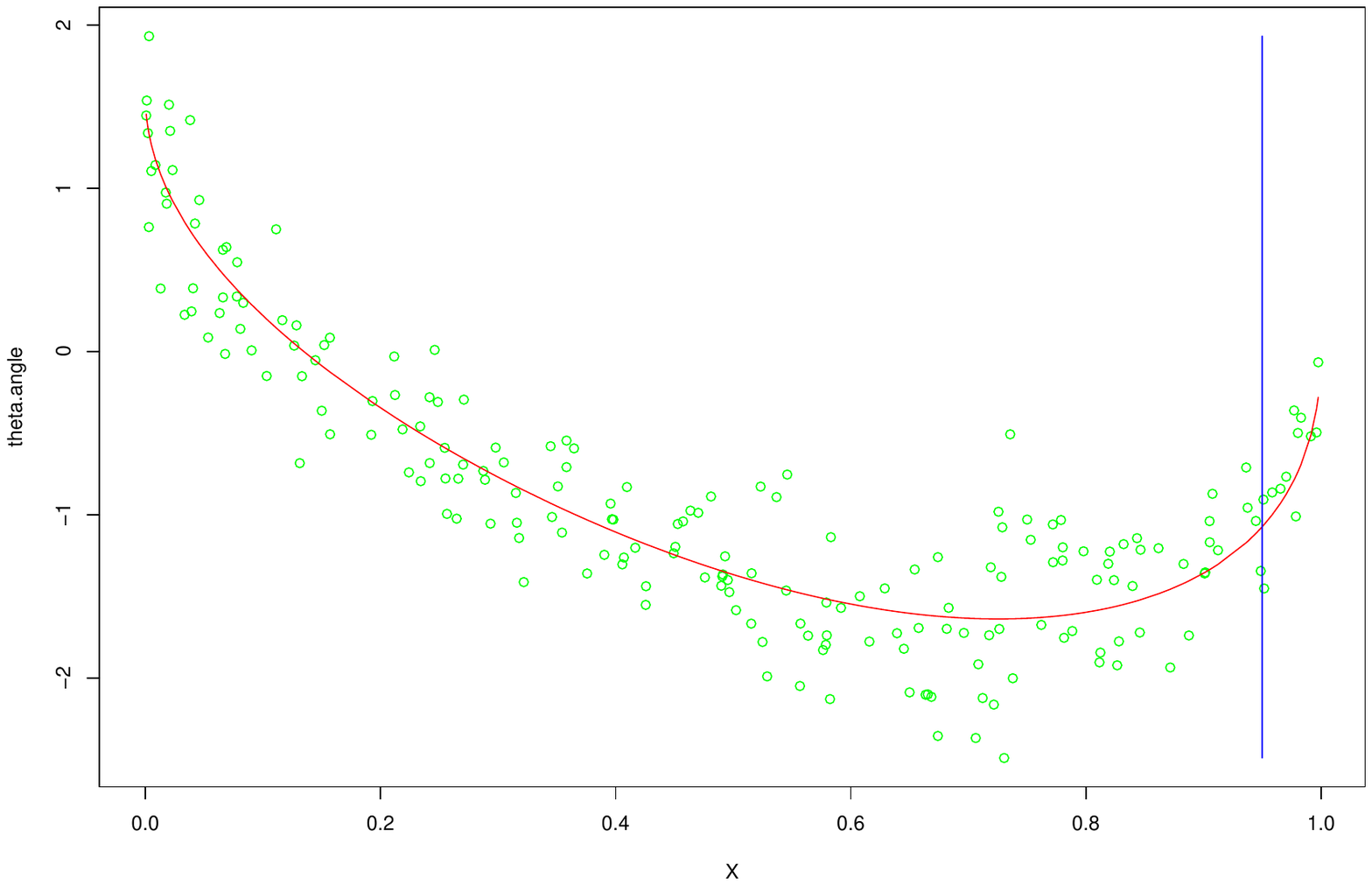}
		\vspace{-0.6cm}
		\subcaption{\scriptsize{$X_i \sim U([0,1])$}}	
		\label{fig:simulation:EpaKernel.model2.ex1:fig.a}
	\end{minipage}
	\hspace{0.1cm}
	\begin{minipage}{.5\linewidth}
		\centering
		\includegraphics[scale=0.21]{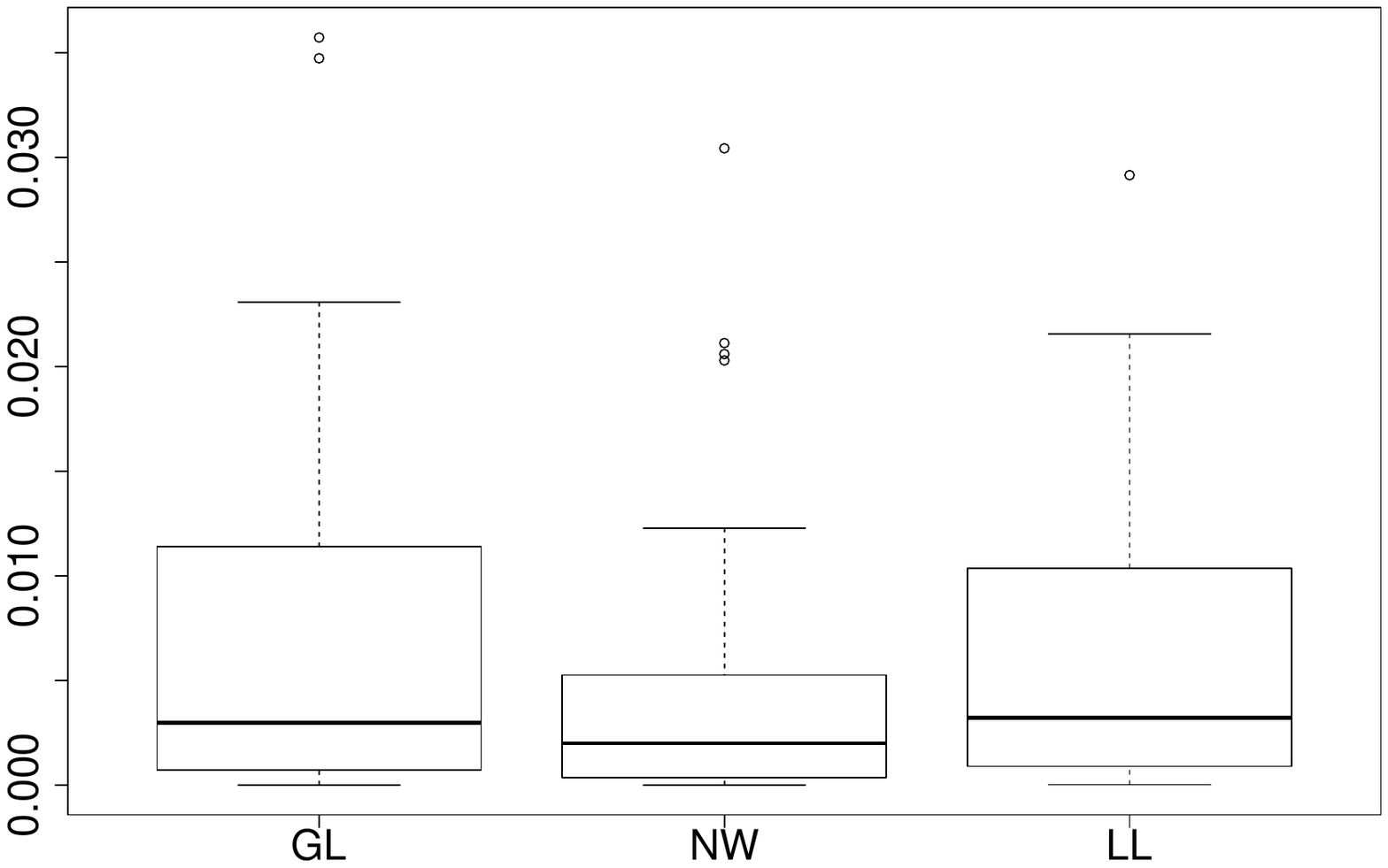}
		\vspace{-0.6cm}
		\subcaption{\scriptsize{Boxplots for GL, NW and LL methodologies}}
		\label{fig:simulation:EpaKernel.model2.ex1:fig.b}
	\end{minipage}
\vspace{-0.2cm}
	\caption{(a): Model M3. Simulated data $(\Theta_i)_{i=1}^{n}$ of Model M3 (green points) with $n=200$ and $X_i\sim U([0,1])$. The red curve represents the true regression function $m$, while the blue vertical line displays the point $x = 0.95$; (b): Boxplots of the estimated risk with $50$ runs for the GL, NW and LL methodologies (from left-hand side to right-hand side) by using the Epanechnikov kernel. }	\label{fig:simulation:EpaKernel.model2.ex1:boxplot:GL.vsNW.vsLL}
\end{figure}
Moreover, to make a comparison with our adaptive estimator, as proposed in~\cite{Marzio-Panzera-Taylor}, we also compute the Nadaraya-Watson (NW) estimator $\widehat{m}_{h}^{NW}$ and the version of the local linear (LL) estimator proposed by \cite[Section 4.2]{Marzio-Panzera-Taylor}) denoted by $\widehat{m}_{h}^{LL}$. Cross-Validation is used to select the bandwidth parameter for $\widehat{m}_{h}^{NW}$ and $\widehat{m}_{h}^{LL}$. 
Boxplots in Figures~\ref{fig:simulation:regression.circular:EpaKernel.model1.ex4.boxplot.n200.XUnif.vsNW.vsLL}, ~\ref{fig:simulation:regression.circular:EpaKernel.model1.ex4.boxplot.n200.XGauss.vsNW.vsLL}, ~\ref{fig:simulation:EpaKernel.model1.ex5:boxplot:GL.vsNW.vsLL} and ~\ref{fig:simulation:EpaKernel.model2.ex1:boxplot:GL.vsNW.vsLL} show that the performances of our estimator are quite satisfying. 

We finally repeat the previous numerical experiments but with the use of the Gaussian kernel: Figures~\ref{fig:simulation:regression.circular:Gausskernel:model1.ex4.boxplot.n200.XUnif.vsNW.vsLL} 
and~\ref{fig:simulation:regression.circular:GaussKernel:model1.ex4.boxplot.n200.XGauss.vsNW.vsLL} are the analogs of Figures~\ref{fig:simulation:regression.circular:EpaKernel.model1.ex4.boxplot.n200.XUnif.vsNW.vsLL} and~\ref{fig:simulation:regression.circular:EpaKernel.model1.ex4.boxplot.n200.XGauss.vsNW.vsLL}. Figure~\ref{fig:simulation:GaussKernel:model1.ex5:boxplot:GL.vsNW.vsLL} 
shows the numerical simulation for model M2 with $X \sim \mathcal{N}(0,1.5)$ 
and we estimate $m(x)$ at $x = 1.05$. Figure~\ref{fig:simulation:GaussKernel:model2.ex1:boxplot:GL.vsNW.vsLL}  
shows the numerical simulation for model M3 with $X \sim U([0,1])$ and we estimate $m(x)$ at $x = 0.95$.
These graphs show that the performances of our adaptive estimator associated with the Gaussian kernel are quite satisfying as well.
\begin{figure}[h!]
	\begin{minipage}{.42\linewidth}
		\hspace{-1cm}
		\includegraphics[scale=0.21]{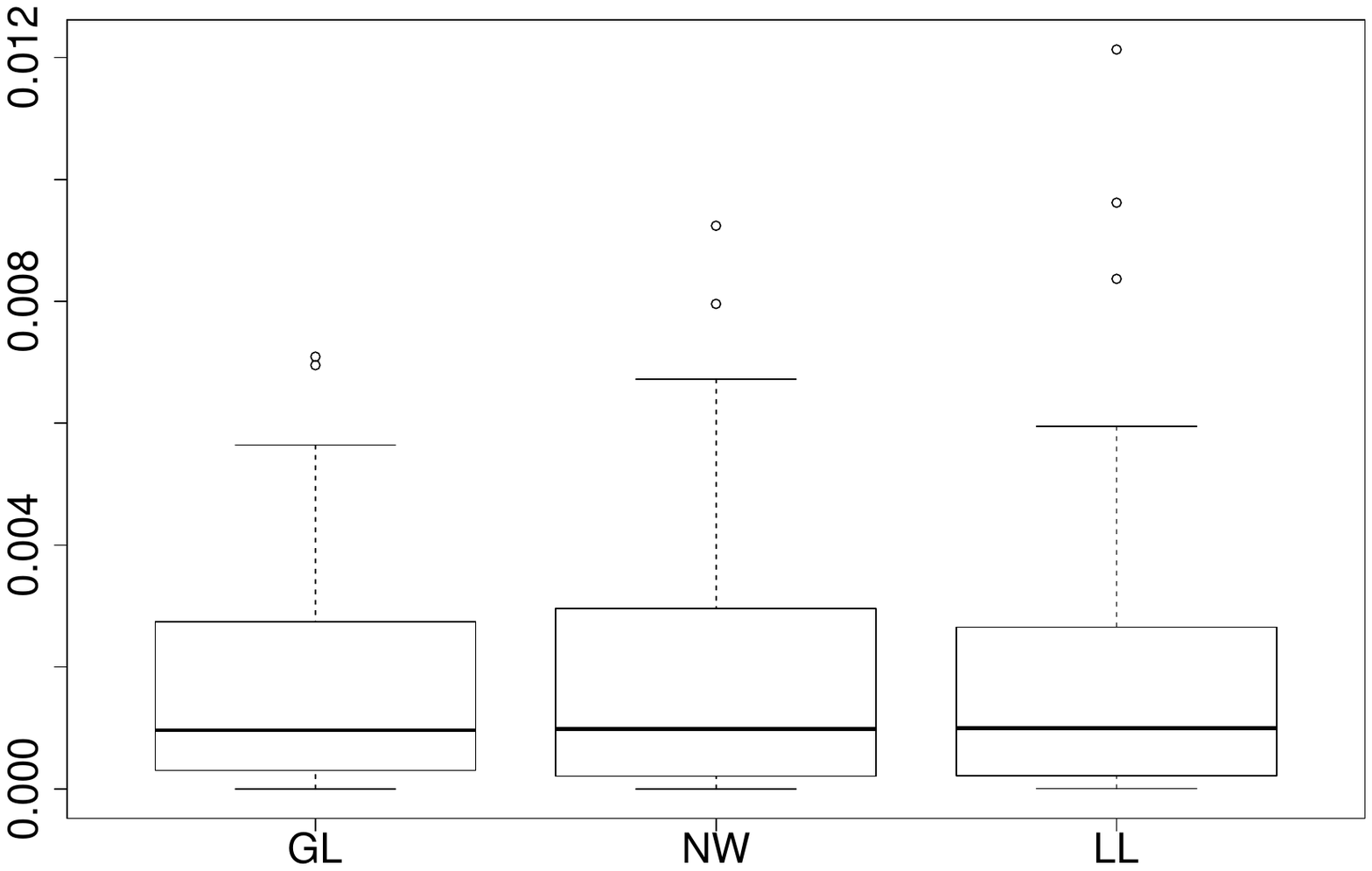}
		\vspace{-0.8cm}
		\subcaption{\scriptsize{estimation at $x = -2$}}
	\end{minipage}
	\hspace{0.01cm}
	\begin{minipage}{.45\linewidth}
		\includegraphics[scale=0.21]{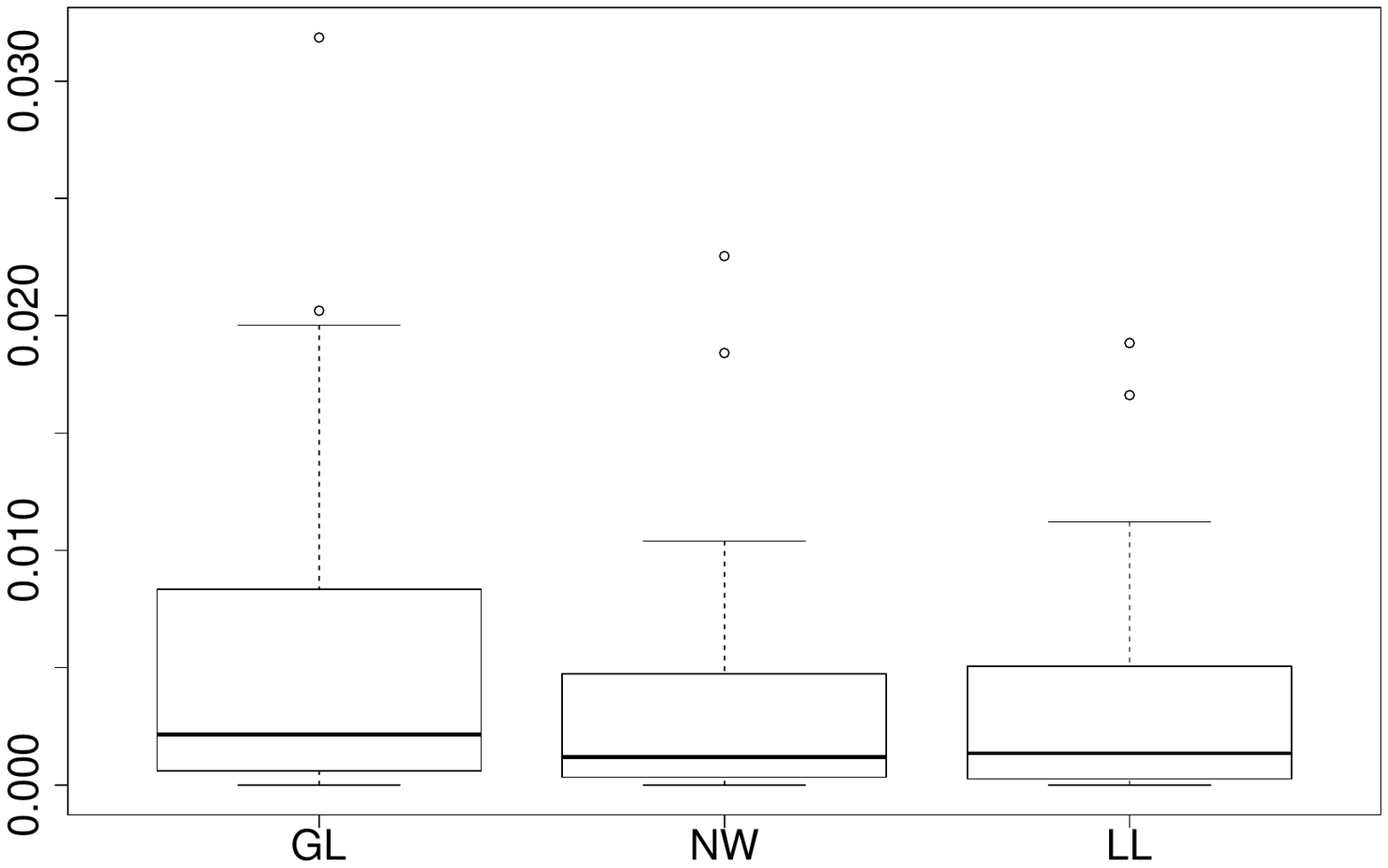}
		\vspace{-0.8cm}
		\subcaption{\scriptsize{estimation at $x = 1.25$}}
	\end{minipage}
	\vspace{-0.2cm}
	\caption{Model M1. Boxplots of the estimated risk with $50$ runs for the GL, NW and LL methodologies (from left-hand side to right-hand side on each plot (a) and (b)) for $n=200$ and $X \sim U([-5,5])$ by using the Gaussian kernel.}
	\label{fig:simulation:regression.circular:Gausskernel:model1.ex4.boxplot.n200.XUnif.vsNW.vsLL}
\end{figure}
\begin{figure}[h!]
	\vspace{-0.5cm}
	\begin{minipage}{.42\linewidth}
		\hspace{-1cm}
		\includegraphics[scale=0.21]{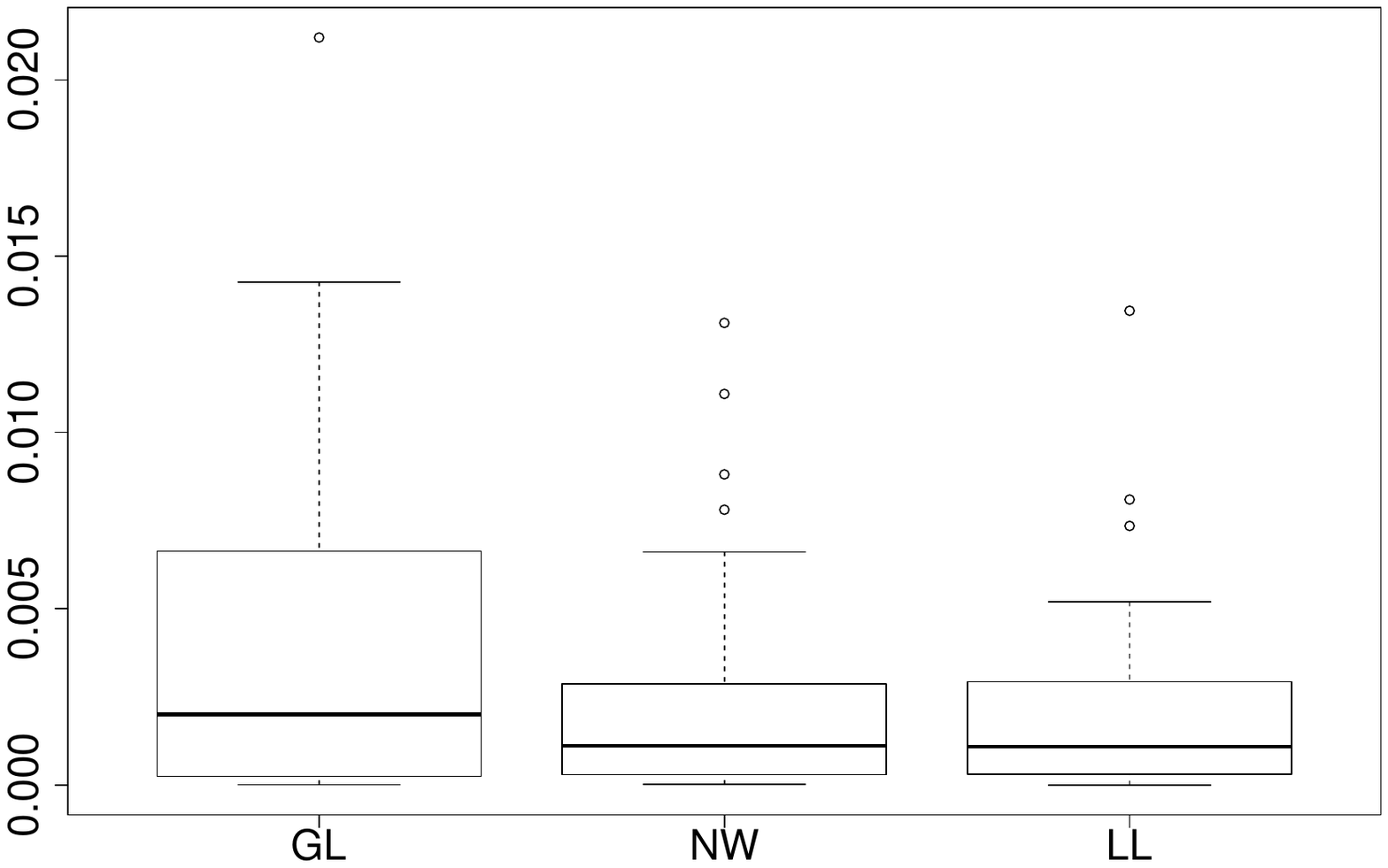}
		\vspace{-0.8cm}
		\subcaption{\scriptsize{estimation at $x = -2$}}
	\end{minipage}
	\hspace{0.01cm}
	\begin{minipage}{.45\linewidth}
		\includegraphics[scale=0.21]{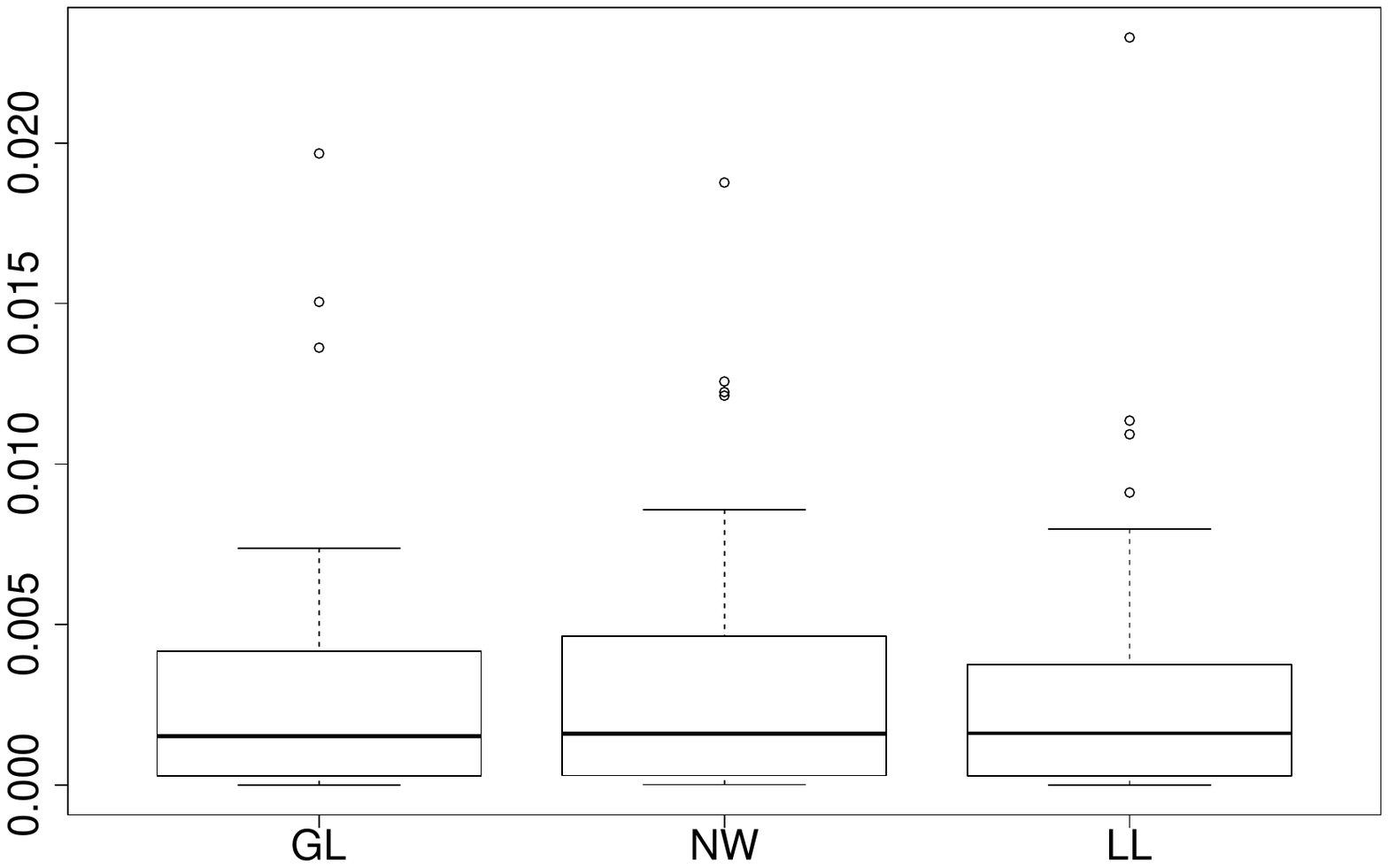} 
		\vspace{-0.8cm}
		\subcaption{\scriptsize{estimation at $x = 1.25$}}
	\end{minipage}
\vspace{-0.2cm}
	\caption{Model M1. Boxplots of the estimated risk with $50$ runs for the GL, NW and LL methodologies (from left-hand side to right-hand side on each plot (a) and (b)) for $n=200$ and $X \sim \mathcal{N}(0,1.5)$ by using the Gaussian kernel.}
	\label{fig:simulation:regression.circular:GaussKernel:model1.ex4.boxplot.n200.XGauss.vsNW.vsLL}
\end{figure}
\begin{figure}[h!]
	\begin{minipage}{.42\linewidth}
		\hspace{-1cm}
		\includegraphics[scale=0.21]{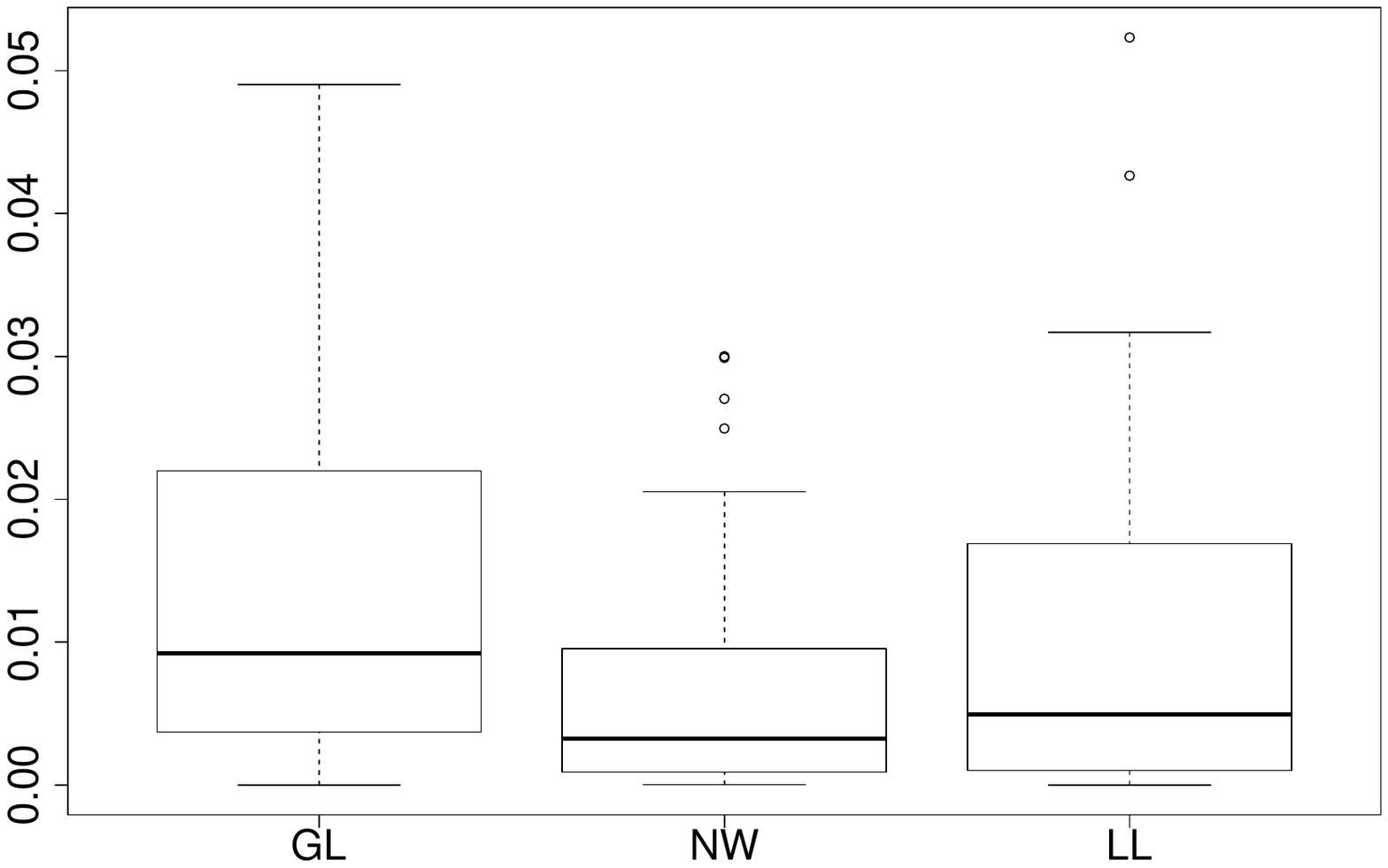}	
		\vspace{-0.8cm}
		\subcaption{\scriptsize{Model M2 with $X \sim \mathcal{N}(0,1.5)$: estimation at $x=1.05$}}
		\label{fig:simulation:GaussKernel:model1.ex5:boxplot:GL.vsNW.vsLL}
	\end{minipage}
	\hspace{0.01cm}
	\begin{minipage}{.45\linewidth}
		\includegraphics[scale=0.21]{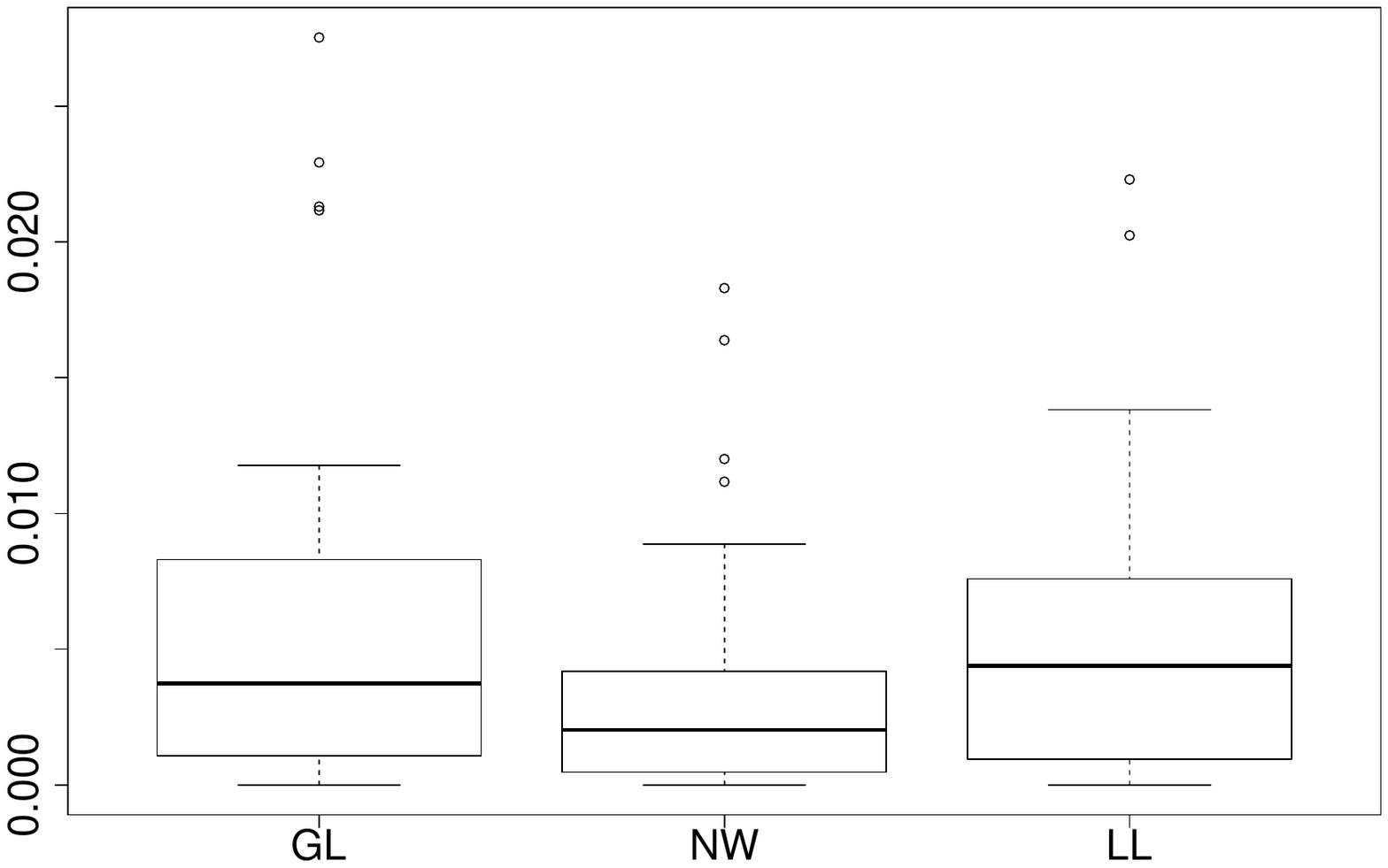} 
		\vspace{-0.8cm}
		\subcaption{\scriptsize{Model M3 with $X \sim U([0,1])$: estimation at $x=0.95$}}  
		\label{fig:simulation:GaussKernel:model2.ex1:boxplot:GL.vsNW.vsLL}
	\end{minipage}
\vspace{-0.2cm}
	\caption{Model M2 (left), Model M3 (right). Boxplots of the estimated risk with $50$ runs for the GL, NW and LL methodologies (from left-hand side to right-hand side on each plot (a) and (b)) for $n=200$  by using the Gaussian kernel.}	\label{fig:simulation:GaussKernel:model1.ex5.model2.ex1}
\end{figure}

\section{Proofs}
\label{sec:proofs}
Along this section, we fix $x$ in $\R$ and we set $u_{x} := F_X(x)$. 
\subsection{Preliminary results}
\label{sec:regression.circular:proofs:preliminaries}
In this section, we study several preliminary results for $\widehat{g}_{1,h_{1}}$ and $\widehat{g}_{2,h_{2}}$ defined in~\eqref{formula:estimate.g1-g2.fixed-h}.  First of all, via the warping method, we observe

\begin{align*} 
\mathbb{E}\big(\widehat{g}_{1,h_{1}}(u_{x})\big) &= \mathbb{E} \Bigg[ \dfrac{1}{n} \sum_{k=1}^{n} \sin(\Theta_{k}). K_{h_{1}} \big( u_{x} - F_{X}(X_{k}) \big) \Bigg]  
\\
&=  \mathbb{E} \Big[ \mathbb{E} \big[ \sin(\Theta) | X \big]. K_{h_{1}} \big( u_{x} - F_{X}(X) \big) \Big]  \nonumber
\\
&=  \int_{\R} \mathbb{E}[\sin(\Theta) | X = y ] . K_{h_{1}} \big( u_{x} - F_{X}(y) \big) . f_{X}(y) dy  \nonumber
\\
&=  \int_{\R} m_{1}(y). K_{h_{1}} \big( u_{x} - F_{X}(y) \big) . f_{X}(y) dy  \nonumber
\\
&=  \int_{F_{X}(\R)} g_{1}(w). K_{h_{1}} \big( u_{x} - w \big) dw  . 
\end{align*}
Then, using the choice of $h_{\max}$, since $u_{x} =  F_{X}(x) \in (0,1)=F_{X}(\R)$ is fixed and $K_{h_{1}} \big( u_{x} - w \big) = 0$ for $w \in (0,1)$ such that $\big| u_{x} - w \big| > A. h_{\max}$, we have 
\begin{align}\label{form:compute.expetation.g_{1,h1}}
\int_{F_{X}(\R)} g_{1}(w). K_{h_{1}} \big( u_{x} - w \big) dw   &=  \int_{u_{x} - A.h_{\max}}^{u_{x} + A.h_{\max}} g_{1}(w). K_{h_{1}} \big( u_{x} - w \big) dw    \nonumber
\\&=  \big( K_{h_{1}} \ast   g_{1}  \big) (u_{x}).
\end{align}
Thus, we obtain
\begin{equation} \label{form:compute.expetation.g_{1,h1}}
\mathbb{E}\big(\widehat{g}_{1,h_{1}}(u_{x})\big) =   \big( K_{h_{1}} \ast  g_{1}  \big) (u_{x})
\end{equation}
and similarly,
\begin{equation} \label{form:compute.expetation.g_{2,h2}}
\mathbb{E}\big(\widehat{g}_{2,h_{2}}(u_{x})\big) =   \big( K_{h_{2}} \ast g_{2}  \big) (u_{x}) .
\end{equation}
We obtain upper bounds for the bias and variance terms.
\begin{lem}\label{lem:point-wise:mean-var-cov.g_{2}}
Let $j\in\{1,2\}$. Suppose that $g_{j}$ belongs to  $\mathcal{H}(\beta_{j}, L_{j})$, with $L_{j},\beta_{j} \in \R^{*}_{+}$. Assume that the kernel $K$ satisfies Assumption~\ref{assumption:kernel.K} with an index $\mathcal{L} \in \mathbb{R}_{+}$ such that $\mathcal{L} \geq   \beta_{j}$.  Then, for any $h_j\in{\mathcal H}_n$,
	\begin{align*}
	\Big| \mathbb{E}\big(\widehat{g}_{j,h_{j}}(u_{x})\big) - g_{j}(u_{x}) \Big|  \leq   C_{K,\mathcal{L}} .L_{j}.  h_{j}^{\beta_{j}} , 
	\quad  &\textrm{ and }  \quad
	\Var\big( \widehat{g}_{j,h_{j}}(u_{x}) \big)  \leq  \dfrac{\left\|K\right\|_{\mathbb{L}^{2}(\R)}^{2}}{n. h_{j} },
	\end{align*}
	with $C_{K,\mathcal{L}}$ the constant defined in Assumption~\ref{assumption:kernel.K}.  	
\end{lem}
\noindent
The proof of Lemma~\ref{lem:point-wise:mean-var-cov.g_{2}} is given in Section~\ref{sec:lem:point-wise:mean-var-cov.g_{2}:proof}.
\leavevmode \\
We introduce in the sequel several events on which we will establish some concentration results for $\widehat{g}_{j,h_{j}}$, $j\in\{1,2\}$.
\begin{definition} \label{def:definition.Omega1n.Omega2.n}\emph{
		For $n \in \N^{*}$, $p \geq 1$, and $h_{j}> 0$, we define for an arbitrary $v \in F_{X}(\R)$ the following event
		\begin{equation*}
		\Omega_{j,n}(v,h_{j})  :=  \left\{ \big| \widehat{g}_{j,h_{j}}(v) - \E[\widehat{g}_{j,h_{j}}(v)] \big|  \leq  c_{1}(p). \sqrt{\widetilde{V}_{j}(n,h_{j})}  \right\}  
		\end{equation*}
		with $c_{1}(p)$ satisfying $c_{1}(p).\sqrt{\min\left\{ c_{0,1} ; c_{0,2} \right\}} \geq 4p$.
		\\
		Furthermore, for $L_{j},\beta_{j} > 0$, we also introduce
		\begin{equation*}
		E_{j,n}(v, h_{j})  :=  \left\{  \big| \widehat{g}_{j,h_{j}}(v) - g_{j}(v) \big| \leq \Phi_{j}(n,h_{j})  \right\} 
		\end{equation*}
		where
		\begin{equation*}
		\Phi_{j}(n,h_{j})  :=  c_{1}(p). \sqrt{ \widetilde{V}_{j}(n,h_{j}) } +  C_{K,\mathcal{L}} . L_{j}. h_{j}^{\beta_{j}}. 
		\end{equation*}
}
\end{definition}
Then, the following proposition gives a concentration inequality for $\widehat{g}_{j,h_{j}}(u_{x})$.
\begin{proposition}\label{prop:apply.Bernstein.inequa}
For $p \geq 1$ and $h_{j}\in \mathcal{H}_{n}$ (defined in~\eqref{def:bandwidth.colletion.H_n}), we have:
	\begin{equation*}
	\mathbb{P} \left( \Big( \Omega_{j,n}(u_{x}, h_{j}) \Big)^{c} \right) 	\leq 2. n^{-p}.
	\end{equation*}
	Consequently, suppose further that $g_{j}$ belongs to $\mathcal{H}(\beta_{j}, L_{j})$ with $L_{j}, \beta_{j} \in \R^{*}_{+}$ and the kernel K satisfies  Assumption~\ref{assumption:kernel.K} with an index $\mathcal{L} \in \mathbb{R}_{+}$ such that $\mathcal{L} \geq  \beta_{j}$, then we get:
\begin{equation*}
\mathbb{P} \left( \Big( E_{j,n}(u_{x}, h_{j}) \Big)^{c} \right) 	\leq 2. n^{-p}.
\end{equation*}
\end{proposition}
\noindent  The proof of Proposition~\ref{prop:apply.Bernstein.inequa} is given in Section~\ref{sec:prop:apply.Bernstein.inequa:proof}.
\subsection{Proof of Lemma~\ref{lem:point-wise:mean-var-cov.g_{2}}}
\label{sec:lem:point-wise:mean-var-cov.g_{2}:proof}
\begin{proof}
First, for the bias of $\widehat{g}_{j,h}(u_{x})$ at $u_{x} = F_{X}(x)$, using~\eqref{form:compute.expetation.g_{1,h1}}, we can write 	\begin{align} \label{lem:point-wise:mean-var-cov.g_{2}:proof:bias:eq1}
	\mathbb{E}\big(\widehat{g}_{j,h_{j}}(u_{x})\big) - g_{j}(u_{x})  &=  \dfrac{1}{h_{j}} . \int_{u_{x} - A. h_{\max}}^{u_{x} + A. h_{\max}} K \Big( \dfrac{u_{x} - z}{h_{j}} \Big). g_{j}(z) dz  - g_{j}(u_{x})    \nonumber
	\\
	&=   \int_{-A}^{A} K ( w ). \big( g_{j}(u_{x} - h_{j}. w) - g_{j}(u_{x}) \big) dw .
	\end{align}
	Since $g_{j}$ belongs to  $\mathcal{H}(\beta_{j}, L_{j})$, using a Taylor expansion for $g_{j}$, we get for $w \in [-A,A]$, $0 \leq \tau \leq 1$,
	\begin{align*}
	g_{j}(u_{x} - h_{j}. w) = g_{j}(u_{x})  &+  g_{j}' (u_{x}). (-h_{j}).w   
	\\  
	&+  ...  + \dfrac{ ( - h_{j}. w)^{ \lfloor\beta_{j} \rfloor } }{(\lfloor\beta_{j} \rfloor) !} .  g_{j}^{(\lfloor\beta_{j} \rfloor)} (u_{x} - \tau. h_{j}. w) .
	\end{align*}
	Then, under Assumption~\ref{assumption:kernel.K} with an index $\mathcal{L} \in \mathbb{R}_{+}$ satisfying $\mathcal{L} \geq   \beta_{j} $, from~\eqref{lem:point-wise:mean-var-cov.g_{2}:proof:bias:eq1} one gets
	\begin{align*}
	&\hspace{-1cm} \int_{-A}^A K ( w ). \big( g_{j}(u_{x} - h_{j}. w) - g_{j}(u_{x}) \big) dw
	\\
	& =  \int_{-A}^A K ( w ).   \dfrac{ (- h_{j}. w)^{ \lfloor\beta_{j} \rfloor } }{(\lfloor\beta_{j} \rfloor) !} .  g_{j}^{(\lfloor\beta_{j} \rfloor)} (u_{x} - \tau. h_{j}. w)  dw
	\\
	& =  \int_{-A}^A K ( w ).   \dfrac{ (- h_{j}. w)^{\lfloor\beta_{j} \rfloor} }{(\lfloor\beta_{j} \rfloor) !} .  \Big( g_{j}^{(\lfloor\beta_{j} \rfloor)} (u_{x} - \tau. h_{j}. w)  - g_{j}^{(\lfloor\beta_{j} \rfloor)}(u_{x}) \Big) dw.
	\end{align*}
	This implies that with $0 \leq \tau \leq 1$, with $C_{K,\mathcal{L}}$ the constant defined in  Assumption~\ref{assumption:kernel.K}, since  $g_{j}  \in  \mathcal{H}(\beta_{j}, L_{j})$,
	\begin{align*}
	& \hspace{-1cm} \Big| \mathbb{E}\big(\widehat{g}_{j,h_{j}}(u_{x})\big) - g_{j}(u_{x}) \Big|  
	\\
	&\leq  \int_{-A}^A  |K(w)|.  \dfrac{ | h_{j}. w |^{ \lfloor\beta_{j} \rfloor } }{(\lfloor\beta_{j} \rfloor) !} .  \Big| g_{j}^{(\lfloor\beta_{j} \rfloor)} (u_{x} - \tau. h_{j}. w)  - g_{j}^{(\lfloor\beta_{j} \rfloor)}(u_{x}) \Big|  dw
	\\
	&\leq  \int_{-A}^A  |K(w)|.  \dfrac{ | h_{j}. w |^{ \lfloor\beta_{j} \rfloor } }{(\lfloor\beta_{j} \rfloor) !} .  L_{j}. \big| \tau. h_{j}. w \big|^{\beta_{j} - \lfloor\beta_{j} \rfloor }  dw  
	\\
	&\leq  L_{j}. h_{j}^{\beta_{j}}. \int_{-A}^A  |K(w)|. | w |^{\beta_{j}}  dw
	\\
	&\leq  L_{j}. h_{j}^{\beta_{j}}. \int_{-A}^A  |K(w)|. (1+ | w | )^{\beta_{j}}  dw   \leq  L_{j}. h_{j}^{\beta_{j}}. C_{K,\mathcal{L}}  . 	\end{align*}
For the variance of $\widehat{g}_{j,h_{j}}(u_{x})$,  one gets, with $a_1(\Theta_{k})= \sin(\Theta_{k})$  and $a_2(\Theta_{k})= \cos(\Theta_{k})$,
	\begin{align*}
	\Var\big( \widehat{g}_{j,h_{j}}(u_{x}) \big) &= \mathbb{E} \Big[ \Big( \widehat{g}_{j,h_{j}}(u_{x}) - \mathbb{E}\big(\widehat{g}_{j,h_{j}}(u_{x})\big) \Big)^{2} \Big] 
	\\
	&= \Var \left( \dfrac{1}{ n }. \sum_{k=1}^{n} a_j(\Theta_{k}). K_{h_{j}} \big( u_{x} - F_{X}(X_{k}) \big) \right)
	\\
	&\leq  \dfrac{1}{n }. \mathbb{E} \left[ \big(a_j(\Theta) \big)^{2}. \big[ K_{h_{j}} \big( u_{x} - F_{X}(X) \big) \big]^{2} \right]
	\\
	&\leq  \dfrac{1}{n }. \mathbb{E} \left( \big[ K_{h_{j}} \big( u_{x} - F_{X}(X) \big) \big]^{2} \right)  \leq  \dfrac{\left\|K\right\|_{\L^{2}(\R)}^{2}}{n. h_{j}} .
	\end{align*}
	This concludes the proof of Lemma \ref{lem:point-wise:mean-var-cov.g_{2}}.
\end{proof}
\subsection{Proof of Proposition~\ref{prop:apply.Bernstein.inequa}}
\label{sec:prop:apply.Bernstein.inequa:proof}
We shall use the following version of Bernstein inequality (see \cite[Lemma 2]{comtelacour}).
\begin{lem}[Bernstein inequality] \label{lem:Berstein-inequality.cite}
	Let $T_{1},\ldots,T_{n}$ be i.i.d. random variables and $S_{n} = \sum_{j=1}^{n} \big[ T_{j} - \E (T_{j}) \big]$. Then, for any $\eta > 0$,
	\begin{align*}
	\mathbb{P} \Big( \big| S_{n} \big| \geq n. \eta \Big)  \leq  2. \max \left( \exp \Big( -\dfrac{n. \eta^{2}}{4.V} \Big) , \exp \Big( -\dfrac{n. \eta}{4.b} \Big) \right)  ,
	\end{align*}
	with $\Var(T_{1}) \leq V$ and $|T_{1}| \leq b$, where $V$ and $b$ are two positive deterministic constants.
\end{lem}
Now, we can start to prove Proposition~\ref{prop:apply.Bernstein.inequa}.
\begin{proof}[Proof of Proposition~\ref{prop:apply.Bernstein.inequa}]
	We follow the procedure proposed in \cite[Proposition $6$]{comtelacour}.
	First of all, we define random variables $Z_{k}(u_{x}) := a_j(\Theta_{k}). K_{h_{j}}\big( u_{x} - F_{X}(X_{k}) \big)$, for $1 \leq k \leq n$,  with $a_1(\Theta_{k})= \sin(\Theta_{k})$ and $a_2(\Theta_{k})= \cos(\Theta_{k})$. Hence, $\widehat{g}_{j,h_{j}}(u_{x}) = \dfrac{1}{n}\sum_{k=1}^{n} Z_{k}( u_{x} )$.
	Notice that $\mathbb{E}(Z_{k}(u_{x})) =  \big( K_{h_j} * g_{j} \big)  ( u_{x} )  $ (see \eqref{form:compute.expetation.g_{1,h1}} and \eqref{form:compute.expetation.g_{2,h2}}).
	Since $\left\|K\right\|_{\infty} < +\infty$, we then have for any $v \in F_{X}(\R)$:
	\begin{equation}  \label{formula:b(h)}
	\left| Z_{k}(v) \right| = \left| a_j(\Theta_{k}). K_{h_{j}}\big( v - F_{X}(X_{k}) \big) \right|   \leq  \dfrac{\left\|K\right\|_{\infty}}{h_{j}} =:  b(h_{j}),
	\end{equation}
	and $\Var \big( Z_{k}(v) \big)  =  n. \Var\big(\widehat{g}_{j,h_{j}}(v) \big)  \leq  n. \dfrac{\left\|K\right\|_{\L^{2}(\R)}^{2}}{n. h_{j}} =: n.V_{0}(n,h_{j}) $.
	Now, applying  Lemma~\ref{lem:Berstein-inequality.cite} to the  $Z_{k}(v)$'s, we obtain for $\eta(h_{j})>0$,
	\begin{align*}
	&\hspace{-1.5cm} \mathbb{P} \left( \Big| \widehat{g}_{j,h_{j}}(v) - \mathbb{E} \big( \widehat{g}_{j,h_{j}}(v) \big) \Big| \geq \eta(h_{j}) \right) 
	\\
	&=
	\mathbb{P} \left( \Big| \sum_{k=1}^{n} Z_{k}(v) - \mathbb{E}\Big(Z_{k}(v)\Big) \Big|  \geq  n. \eta(h_{j}) \right)
	\\
	&\leq  2. \max \left\{ \exp \left( -\dfrac{n.\big(\eta(h_{j}) \big)^2}{4 n.V_{0}(n,h_{j})} \right)  ;  \exp \left( -\dfrac{n.\eta(h_{j})}{4 b(h_{j})} \right)  \right\}   .
	\end{align*}
	For $p \geq 1$, choose $\eta(h_{j}) = c_{1}(p). \sqrt{\widetilde{V}_{j}(n,h_{j})}$, with 
	\begin{equation}\label{def:tildeV}
	\widetilde{V}_{j}(n,h_{j}) = c_{0,j}.\log(n). V_{0}(n,h_{j}).
	\end{equation}
	Then,
	\begin{align}\label{bernstein.inqua:byBerstein}
	&\hspace{-0.5cm}\mathbb{P} \left( \Big| \widehat{g}_{j,h_{j}}(v) - \mathbb{E} \big( \widehat{g}_{j,h_{j}}(v) \big) \Big| \geq c_{1}(p). \sqrt{\widetilde{V}_{j}(n,h_{j})} \right) \nonumber
	\\
	&\leq  2. \max \left\{ \exp \left( -\dfrac{n.c_{1}(p)^{2}. \widetilde{V}_{j}(n,h_{j}) }{4 n.V_{0}(n,h_{j})} \right)  ;  \exp \left( -\dfrac{n.c_{1}(p). \sqrt{\widetilde{V}_{j}(n,h_{j})}}{4 b(h_{j})} \right)  \right\}  .
	\end{align}
	We then choose $c_{1}(p)$.
	\\
	$+$ First, $c_{1}(p)$ is chosen such that
	\begin{equation} \label{bernstein.inqua:maximize.1}
	\dfrac{n.c_{1}(p)^{2}. \widetilde{V}_{j}(n,h_{j}) }{4 n.V_{0}(n,h_{j})}  =  \dfrac{n.c_{1}(p)^{2}. c_{0,j}.\log(n).V_{0}(n,h_{j}) }{4 n.V_{0}(n,h_{j})} \geq  p. \log(n) ,
	\end{equation}
	that is $c_{1}(p)$ satisfies $c_{1}(p)^{2}. c_{0,j} \geq 4p$.
	\\
	$+$ Secondly, we can write
	\begin{equation*}
	\dfrac{n.c_{1}(p). \sqrt{\widetilde{V}_{j}(n,h_{j})}}{4 b(h_{j})} = \dfrac{c_{1}(p).\sqrt{c_{0,j}} }{4}. \sqrt{\log(n)}. \dfrac{n.\sqrt{V_{0}(n,h_{j})}}{b(h_{j})}
	\end{equation*}
	and for $h_{j} \in \mathcal{H}_{n}$,
	\begin{equation*} 
	n. \dfrac{\sqrt{V_{0}(n,h_{j})}}{b(h_{j})} = n. \dfrac{\left\|K\right\|_{\L^{2}(\R)}}{\sqrt{n. h_{j}}}. \dfrac{h_{j}}{\left\|K\right\|_{\infty}} = \sqrt{n.h_{j}}. \dfrac{\left\|K\right\|_{\L^{2}(\R)}}{\left\|K\right\|_{\infty}}  >  \sqrt{\log(n)}   ,
	\end{equation*}
	then, we have
	\begin{align}
	\label{bernstein.inqua:maximize.2}
	\dfrac{n.c_{1}(p). \sqrt{\widetilde{V}_{j}(n,h_{j})}}{4 b(h_{j})} = \dfrac{c_{1}(p).\sqrt{c_{0,j}} }{4}. \sqrt{\log(n)}. \dfrac{n.\sqrt{V_{0}(n,h_{j})}}{b(h_{j})}  &\geq  \dfrac{c_{1}(p).\sqrt{c_{0,j}} }{4}. \log(n) 
	\\ 
	&\geq  p. \log(n) \nonumber ,
	\end{align}
	provided that  $c_{1}(p).\sqrt{c_{0,j}} \geq 4p$. Note that this condition also ensures the constraint $c_{1}(p)^{2}. c_{0,j} \geq 4p$.

	Now, combining \eqref{bernstein.inqua:maximize.1} and \eqref{bernstein.inqua:maximize.2}, we get from \eqref{bernstein.inqua:byBerstein} for any $p \geq 1$:
	\begin{align*}
	\mathbb{P} \left(\Big( \Omega_{j,n}(u_{x}, h_{j}) \Big)^{c} \right)  &=  \mathbb{P} \left( \Big| \widehat{g}_{j,h_{j}}(u_{x}) - \mathbb{E} \big( \widehat{g}_{j,h_{j}}(u_{x}) \big) \Big| > c_{1}(p). \sqrt{\widetilde{V}_{j}(n,h_{j})} \right)  
	\\
	&\leq 2. n^{-p}.
	\end{align*}
	This implies that with probability larger than $1 - 2.n^{-p}$, we have:
	\begin{align*}
	\big| \widehat{g}_{j,h_{j}}(u_{x}) - \mathbb{E} \big( \widehat{g}_{j,h_{j}}(u_{x}) \big) \big|  \leq  c_{1}(p). \sqrt{\widetilde{V}_{j}(n,h_{j})}.
	\end{align*}
	Then, with probability larger than $1 - 2.n^{-p}$, we obtain
	\begin{align*}
	\big| \widehat{g}_{j,h_{j}}(u_{x}) - g_{j}(u_{x})  \big|  &\leq  \big| \widehat{g}_{j,h_{j}}(u_{x}) - \mathbb{E} \big( \widehat{g}_{j,h_{j}}(u_{x}) \big) \big| + \big| \mathbb{E} \big( \widehat{g}_{j,h_{j}}(u_{x}) \big) - g_{j}(u_{x})  \big|
	\\
	&\leq  c_{1}(p). \sqrt{\widetilde{V}_{j}(n,h_{j})} \quad + \quad  L_{j}. h_{j}^{\beta_{j}} . C_{K,\mathcal{L}} .  
		\end{align*}
	Recall that $\Phi_{j}(n,h_{j}) =  c_{1}(p). \sqrt{\widetilde{V}_{j}(n,h_{j})}  +  L_{j}. h_{j}^{\beta_{j}} . C_{K,\mathcal{L}}  
		$, therefore, we finally obtain for $p \geq 1$ that
	\begin{equation}
	\mathbb{P} \left(\Big( E_{j,n}( u_{x} , h_{j} ) \Big)^{c} \right) =
	\mathbb{P} \left( \big| \widehat{g}_{j,h_{j}}(u_{x}) - g_{j}(u_{x}) \big| > \Phi_{j}(n,h_{j}) \right)  \leq 2. n^{-p}  .
	\end{equation}
 This concludes the proof of Proposition~\ref{prop:apply.Bernstein.inequa}.
	\end{proof}
\subsection{Proofs of main results}
\label{sec:regression.circular:proofs:proofs.main.results}
\subsubsection{Proof of Proposition~\ref{adaptive:prop:upper-bound.high-Proba:g2-g1} }
\label{sec:adaptive:prop:upper-bound.high-Proba:g2-g1:proof}
We first have the following concentration result.
\begin{corollary}
	\label{corollary:prop:apply.Bernstein.inequa:study.adaptation}
Let $j\in\{1,2\}$.	Under the Assumptions of Proposition~\ref{adaptive:prop:upper-bound.high-Proba:g2-g1}, for all $h_{j}, h_{j}' \in \mathcal{H}_{n}$, for all $p \geq 1$, for $u_{x} = F_{X}(x)$,
	\begin{equation*}
	\mathbb{P} \left( \big| \widehat{g}_{j, h_{j} ,h_{j}'}(u_{x}) - \E \big[ \widehat{g}_{j, h_{j} ,h_{j}'}(u_{x}) \big] \big|  >  c_{1}(p). \left\|K\right\|_{\L^1(\R)}.  \sqrt{\widetilde{V}_{j}(n,h_{j}')} \right)  \leq  2. n^{-p},
	\end{equation*}
	with $c_{1}(p)$ satisfying $c_{1}(p).\sqrt{\min\left\{ c_{0,1} ; c_{0,2} \right\}} \geq 4p$.
\end{corollary}
\begin{proof}[Proof of Corollary~\ref{corollary:prop:apply.Bernstein.inequa:study.adaptation}]
We define random variables $\widetilde{Z}_{k}(u_{x}) := a_j(\Theta_{k}).  \big( K_{h_{j}'} \ast K_{h_{j}}  \big)  \big( u_{x} - F_{X}(X_{k}) \big)$, for $1 \leq k \leq n$,  with $a_1(\Theta_{k})= \sin(\Theta_{k})$ and $a_2(\Theta_{k})= \cos(\Theta_{k})$. Hence, $\widehat{g}_{j,h_{j}, h_{j}'}(u_{x}) = \dfrac{1}{n}\sum_{k=1}^{n} \widetilde{Z}_{k}( u_{x} )$.
	Since $\left\|K\right\|_{\infty} < +\infty$, we then have: 	\begin{equation*} 
	\left| \widetilde{Z}_{k}(u_{x}) \right| = \left| a_j (\Theta_{k}). \big( K_{h_{j}'} \ast K_{h_{j}}  \big) \big( u_{x} - F_{X}(X_{k}) \big) \right|   \leq  \dfrac{\left\|K\right\|_{\infty} . \left\|K\right\|_{\L^{1}(\R)} }{h_{j}} =:  \widetilde{b} (h_{j}),
	\end{equation*}
	and $\Var \big( \widetilde{Z}_{k}(u_{x}) \big)  =  n. \Var\big(\widehat{g}_{j,h_{j},h_{j}'}(u_{x}) \big)  \leq  n. \dfrac{\left\|K\right\|_{\L^{1}(\R)}^{2} . \left\|K\right\|_{\L^{2}(\R)}^{2}}{n. h_{j}} =: n.\widetilde{V}_{0}(n,h_{j}) $.\\
Using similar arguments of the proof of Proposition~\ref{prop:apply.Bernstein.inequa}, we obtain with a probability greater than $1 - 2.n^{-p}$ that
	\begin{align*}
	\big| \widehat{g}_{j, h_{j} ,h_{j}'}(u_{x}) - \E \big[ \widehat{g}_{j, h_{j} ,h_{j}'}(u_{x}) \big] \big|  \leq  \left\|K\right\|_{\L^{1}(\R)}.  c_{1}(p).  \sqrt{\widetilde{V}_{j}(n,h_{j})} .
	\end{align*}
This concludes the proof of Corollary~\ref{corollary:prop:apply.Bernstein.inequa:study.adaptation}.
\end{proof}
Now, we can start to prove Proposition~\ref{adaptive:prop:upper-bound.high-Proba:g2-g1}.
\begin{proof}[Proof of Proposition~\ref{adaptive:prop:upper-bound.high-Proba:g2-g1}]
We follow the strategy proposed in \cite[Theorem 2]{comtelacour}. The target is to find an upper bound for $\big| \widehat{g}_{j,\widehat{h}_{j}}(u_{x})  -  g_{j}(u_{x}) \big|$. Let $h_{j} \in \mathcal{H}_{n}$ be fixed. We consider the following decomposition:
	\begin{align*}
	\big| \widehat{g}_{j,\widehat{h}_{j}}(u_{x})  -  g_{j}(u_{x}) \big|   \leq   \underbrace{ \big| \widehat{g}_{j,\widehat{h}_{j}}(u_{x})  -  \widehat{g}_{j, h_{j} ,\widehat{h}_{j}}(u_{x}) \big| }_{=: I_{g_{j},1}}   &+   \underbrace{\big| \widehat{g}_{j, h_{j} ,\widehat{h}_{j}}(u_{x})  -  \widehat{g}_{j, h_{j}}(u_{x}) \big|}_{=: I_{g_{j},2}}   
	\\
	&+   \big|  \widehat{g}_{j, h_{j}}(u_{x}) - g_{j}(u_{x}) \big| .
	\end{align*}
	By the definition of $A_{2}(h_{j},u_{x})$, we have
	\begin{align*}
	I_{g_{j},1}  &=  \big| \widehat{g}_{j,\widehat{h}_{j}}(u_{x})  -  \widehat{g}_{j, h_{j} ,\widehat{h}_{j}}(u_{x}) \big|  
	\\
	&=  \big| \widehat{g}_{j,\widehat{h}_{j} }(u_{x})  -  \widehat{g}_{j, h_{j} ,\widehat{h}_{j}}(u_{x}) \big|  - \sqrt{\widetilde{V}_{j}(n,\widehat{h}_{j})}  +   \sqrt{\widetilde{V}_{j}(n,\widehat{h}_{j})}
	\\
	&\leq  \sup_{h_{j}' \in \mathcal{H}_{n}} \Big\{  \big| \widehat{g}_{j,h_{j}'}(u_{x})  -  \widehat{g}_{j, h_{j} ,h_{j}'}(u_{x}) \big|  -  \sqrt{\widetilde{V}_{j}(n,h_{j}')}  \Big\}_{+}  +   \sqrt{\widetilde{V}_{j}(n,\widehat{h}_{j})}
	\\
	&=  A_{2}(h_{j},u_{x})  +    \sqrt{\widetilde{V}_{j}(n,\widehat{h}_{j})}.
	\end{align*}
	And similarly, by the definition of $A_{2}(\widehat{h}_{j},u_{x})$,
	\begin{align*}
	I_{g_{j},2}  &=  \big| \widehat{g}_{j, h_{j} ,\widehat{h}_{j}}(u_{x})  -  \widehat{g}_{j, h_{j}}(u_{x}) \big|  
	\\
	&\leq  \sup_{h_{j}' \in \mathcal{H}_{n}} \Big\{  \big| \widehat{g}_{j, h_{j}' ,\widehat{h}_{j}}(u_{x})  -  \widehat{g}_{j, h_{j}'}(u_{x}) \big|  -  \sqrt{\widetilde{V}_{j}(n,h_{j}')}  \Big\}_{+}  +   \sqrt{\widetilde{V}_{j}(n,h_{j})}
	\\
	&=  A_{2}(\widehat{h}_{j},u_{x})  +    \sqrt{\widetilde{V}_{j}(n,h_{j})}.
	\end{align*}
	Therefore, by using the definition of $\widehat{h}_{j}$, we get
	\begin{align*}
	I_{g_{j},1}  +  I_{g_{j},2}  &\leq  A_{2}(h_{j},u_{x})  +    \sqrt{\widetilde{V}_{j}(n,\widehat{h}_{j})}  +  A_{2}(\widehat{h}_{j},u_{x})  +   \sqrt{\widetilde{V}_{j}(n,h_{j})}  
	\\
	&\leq  2. \Big[ A_{2}(h_{j},u_{x}) +   \sqrt{\widetilde{V}_{j}(n,h_{j})}   \Big]  .
	\end{align*}
	Hence, we obtain
	\begin{equation} \label{adaptive:thm:upper-bound.high-Proba:g2:proof:inequality1}
	\big| \widehat{g}_{j,\widehat{h}_{j}}(u_{x})  -  g_{j}(u_{x}) \big|   \leq   2.  A_{2}(h_{j},u_{x})  +   2. \sqrt{\widetilde{V}_{j}(n,h_{j})}  +   \big|  \widehat{g}_{j, h_{j}}(u_{x}) - g_{j}(u_{x}) \big| .
	\end{equation}
	Now, to study $A_{2}(h_{j},u_{x})$, we can write:
	\begin{align*}
	\widehat{g}_{j,h_{j}'}(u_{x})  -  \widehat{g}_{j, h_{j} ,h_{j}'}(u_{x})   &=   \widehat{g}_{j,h_{j}'}(u_{x})  -  \E \big[ \widehat{g}_{j,h_{j}'}(u_{x}) \big]  - \big( \widehat{g}_{j, h_{j} ,h_{j}'}(u_{x}) 
	\\
	&\hspace{0.5cm} - \E \big[ \widehat{g}_{j, h_{j} ,h_{j}'}(u_{x}) \big] \big)   +   \E \big[ \widehat{g}_{j,h_{j}'}(u_{x}) \big]  -  \E \big[ \widehat{g}_{j, h_{j} ,h_{j}'}(u_{x}) \big] ,
	\end{align*}
	and, we have $\E \big[ \widehat{g}_{j,h_{j}'}(u_{x}) \big] = \big( K_{h_{j}'}*g_{j}  \big)(u_{x})$ as well as $\E \big[ \widehat{g}_{j, h_{j} ,h_{j}'}(u_{x}) \big] = \E \big[ K_{h_{j}'} *  \widehat{g}_{j,h_{j}}(u_{x}) \big] = \big( K_{h_{j}'}*K_{h_{j}}*g_{j} \big)(u_{x})$.
	\\
	Thus,
	\begin{align*}
	\big| \widehat{g}_{j,h_{j}'}(u_{x})  -  \widehat{g}_{j, h_{j} ,h_{j}'}(u_{x}) \big| -   \sqrt{\widetilde{V}_{j}(n,h_{j}')}  &\leq  \big| \widehat{g}_{j,h_{j}'}(u_{x})  -  \E \big[ \widehat{g}_{j,h_{j}'}(u_{x}) \big]  \big|  
	\\
	&\hspace{0.2cm}-   \dfrac{\sqrt{\widetilde{V}_{j}(n,h_{j}')}}{ (1 + \left\|K\right\|_{\L^1(\R)} ) }
	\\
	& \hspace{0.2cm}  +  \big| \widehat{g}_{j, h_{j} ,h_{j}'}(u_{x}) - \E \big[ \widehat{g}_{j, h_{j} ,h_{j}'}(u_{x}) \big] \big|  
	\\
	& \hspace{0.2cm} -  \left\|K\right\|_{\L^1(\R)}.  \dfrac{\sqrt{\widetilde{V}_{j}(n,h_{j}')}}{ (1 + \left\|K\right\|_{\L^1(\R)} ) }
	\\
	& \hspace{0.2cm} +  \big| \E \big[ \widehat{g}_{j,h_{j}'}(u_{x}) \big]  -  \E \big[ \widehat{g}_{j, h_{j} ,h_{j}'}(u_{x}) \big] \big|.
	\end{align*}
	However, for any $h_{j}' \in \mathcal{H}_{n}$,
	\begin{align}\label{uniformbias}
	\big| \E \big[ \widehat{g}_{j,h_{j}'}(u_{x}) \big]  -  \E \big[ \widehat{g}_{j, h_{j} ,h_{j}'}(u_{x}) \big] \big|  &=  \big| K_{h_{j}'}* \big( g_{j} - K_{h_{j}}*g_{j} \big)(u_{x}) \big|\nonumber
	\\  
	&\leq  \left\| K \right\|_{\L^{1}(\R)} . \left\|g_{j} - K_{h_{j}}*g_{j}\right\|_{\infty}.
	\end{align}
	Hence, 
	incorporating this bound in the definition of $A_{2}(h_{j},u_{x})$, we obtain
	\begin{align}  \label{adaptive:thm:upper-bound.high-Proba:g2:proof:inequality2}
	&\hspace{-0.5cm} A_{2}(h_{j},u_{x}) 
	\\
	&= \sup_{h_{j}' \in \mathcal{H}_{n}}  \Big\{ \big| \widehat{g}_{j,h_{j}'}(u_{x})  -  \widehat{g}_{j, h_{j} ,h_{j}'}(u_{x}) \big| -   \sqrt{\widetilde{V}_{j}(n,h_{j}')} \Big\}_{+}   \nonumber
	\\
	&\leq  \sup_{h_{j}' \in \mathcal{H}_{n}}  \left\{\big| \widehat{g}_{j,h_{j}'}(u_{x})  -  \E \big[ \widehat{g}_{j,h_{j}'}(u_{x}) \big]  \big|  -   \dfrac{\sqrt{\widetilde{V}_{j}(n,h_{j}')}}{ (1 + \left\|K\right\|_{\L^1(\R)} ) }  \right\}_{+}
	\\
	&\hspace{0.4cm}  +  \sup_{h_{j}' \in \mathcal{H}_{n}}  \left\{\big| \widehat{g}_{j, h_{j} ,h_{j}'}(u_{x}) - \E \big[ \widehat{g}_{j, h_{j} ,h_{j}'}(u_{x}) \big] \big|  -   \left\|K\right\|_{\L^1(\R)}.  \dfrac{\sqrt{\widetilde{V}_{j}(n,h_{j}')}}{ (1 + \left\|K\right\|_{\L^1(\R)} ) }  \right\}_{+}
	\nonumber
	\\
	& \hspace{6cm} +   \left\|K\right\|_{\L^1(\R)}. \left\|g_{j} - K_{h_{j}}*g_{j}\right\|_{\infty}.  \nonumber
	\end{align}
	From Corollary~\ref{corollary:prop:apply.Bernstein.inequa:study.adaptation}, for $h_{j}, h_{j}' \in \mathcal{H}_{n}$,
	\begin{align*}
	\mathbb{P} \left( \big| \widehat{g}_{j, h_{j} ,h_{j}'}(u_{x}) - \E \big[ \widehat{g}_{j, h_{j} ,h_{j}'}(u_{x}) \big] \big|  >  c_{1}(p). \left\|K\right\|_{\L^1(\R)}. \sqrt{\widetilde{V}_{j}(n,h_{j}')} \right)  \leq  2. n^{-p}.
	\end{align*}
	It implies that if we take $c_{1}(p) = \dfrac{1}{1 + \left\|K\right\|_{\L^1(\R)}}$ and if $c_{0,j} \geq 16 p^{2}. \big( 1 + \left\|K\right\|_{\L^1(\R)} \big)^{2}$, then
	\begin{align*}
	\mathbb{P} \left(  \sup_{h_{j}' \in \mathcal{H}_{n}}  \left\{ \big| \widehat{g}_{j,h_{j}'}(u_{x})  -  \E \big[ \widehat{g}_{j,h_{j}'}(u_{x}) \big]  \big|  -   \dfrac{\sqrt{\widetilde{V}_{j}(n,h_{j}')}}{ (1 + \left\|K\right\|_{\L^1(\R)} ) }  \right\}_{+}  >  0  \right)  &\leq  2. \sum_{h_{j} \in \mathcal{H}_{n}} n^{-p}  
	\\
	&\leq   2. n^{1 - p}  ,
	\end{align*}
	as $\card(\mathcal{H}_{n})  \leq  n$.
		Consequently, the following set
	\begin{align*}
	\widetilde{A}_{2}  
	&:=  \Big\{  \sup_{h_{j}' \in \mathcal{H}_{n}}  \Big\{ \big| \widehat{g}_{j,h_{j}'}(u_{x})  -  \E \big[ \widehat{g}_{j,h_{j}'}(u_{x}) \big]  \big|  -    \dfrac{\sqrt{\widetilde{V}_{j}(n,h_{j}')}}{ (1 + \left\|K\right\|_{\L^1(\R)} ) }  \Big\}_{+}  = 0  \Big\}
	\\
	&\hspace{0.5cm}  \cap  \Big\{ \forall \, h_j \in \mathcal{H}_n,  \sup_{h_{j}' \in \mathcal{H}_{n}}  \Big\{ \big| \widehat{g}_{j, h_{j} ,h_{j}'}(u_{x}) - \E \big[ \widehat{g}_{j, h_{j} ,h_{j}'}(u_{x}) \big] \big|  
	\\
	&\hspace{5cm} -   \left\|K\right\|_{\L^1(\R)}.  \dfrac{\sqrt{\widetilde{V}_{j}(n,h_{j}')}}{ (1 + \left\|K\right\|_{\L^1(\R)} ) }  \Big\}_{+} = 0  \Big\}
	\end{align*}
	has probability greater than $(1 - 4.n^{2 - p})$. Now, choose $p = 2 + q$ and then $c_{0,j} \geq 16 \big(2 + q \big)^{2}. \big( 1 + \left\|K\right\|_{\L^1(\R)} \big)^{2}$. Thus, we obtain that $\mathbb{P} \big( \widetilde{A}_{2} \big) > 1 - 4.n^{-q} $.
	\leavevmode \\
	Combining inequalities \eqref{adaptive:thm:upper-bound.high-Proba:g2:proof:inequality1} and \eqref{adaptive:thm:upper-bound.high-Proba:g2:proof:inequality2}, we have on $\widetilde{A}_{2}$:
	\begin{align*}
	\big| \widehat{g}_{j,\widehat{h}_{j}}(u_{x})  -  g_{j}(u_{x}) \big|   &\leq   2.  A_{2}(h_{j},u_{x})  +   2.  \sqrt{\widetilde{V}_{j}(n,h_{j})}  +   \big|  \widehat{g}_{j, h_{j}}(u_{x}) - g_{j}(u_{x}) \big|
	\\
	&\leq  2. \left\|K\right\|_{\L^1(\R)}. \left\|g_{j} - K_{h_{j}}*g_{j}\right\|_{\infty}  +  2. \sqrt{\widetilde{V}_{j}(n,h_{j})}  
	\\
	& \hspace{4.8cm} +   \big|  \widehat{g}_{j, h_{j}}(u_{x}) - g_{j}(u_{x}) \big| ,
	\end{align*}
	but still on $\widetilde{A}_{2}$, one gets $\big| \widehat{g}_{j,h_{j}}(u_{x}) - \E \big[\widehat{g}_{j,h_{j}}(u_{x})\big] \big|  -   \dfrac{\sqrt{\widetilde{V}_{j}(n,h_{j})}}{ (1 + \left\|K\right\|_{\L^1(\R)}) } \leq 0$, so
	\begin{align*}
	\big| \widehat{g}_{j,h_{j}}(u_{x})  -  g_{j}(u_{x}) \big|   &\leq  \big| \E \big[\widehat{g}_{j,h_{j}}(u_{x})\big] - g_{j}(u_{x}) \big|  +  \big| \widehat{g}_{j,h_{j}}(u_{x}) - \E \big[\widehat{g}_{j,h_{j}}(u_{x})\big] \big|  
	\\ & \hspace{2.5cm}
	-  \dfrac{\sqrt{\widetilde{V}_{j}(n,h_{j})}}{ (1 + \left\|K\right\|_{\L^1(\R)} ) } 
	+  \dfrac{\sqrt{\widetilde{V}_{j}(n,h_{j})}}{ (1 + \left\|K\right\|_{\L^1(\R)} ) }
	\\
	&\leq  \left\|g_{j} - K_{h_{j}}*g_{j}\right\|_{\infty}  +   \sqrt{\widetilde{V}_{j}(n,h_{j})}.
	\end{align*}
	Therefore, on $\widetilde{A}_{2}$, we finally obtain
	\begin{align*}
	\big| \widehat{g}_{j,\widehat{h}_{j}}(u_{x})  -  g_{j}(u_{x}) \big|   &\leq   (1 + 2 \left\|K\right\|_{\L^1(\R)}). \left\|g_{j} - K_{h_{j}}*g_{j}\right\|_{\infty}  +  3.  \sqrt{\widetilde{V}_{j}(n,h_{j})} .
	\end{align*}
This concludes the proof of Proposition~\ref{adaptive:prop:upper-bound.high-Proba:g2-g1}.
\end{proof}
\subsubsection{Proof of Theorem~\ref{adaptive:thm:pointwise-risk.upper-bound:g} }
\label{sec:adaptive:thm:pointwise-risk.upper-bound:g:proof}
First, we establish a concentration result for $\widehat{g}_{1,\widehat{h}_1}(u_{x})$ and $\widehat{g}_{2,\widehat{h}_2}(u_{x})$ as follows. In the sequel, we consider $j\in\{1,2\}$ and we set
$$\psi_{n}(\beta_{j}) = \Bigg( \frac{\log(n)}{n} \Bigg)^{\frac{\beta_{j}}{2\beta_{j} + 1}}.$$
\begin{corollary}\label{adaptive:corollary:concentration:g2-g1}
Suppose that $g_{j}$ belongs to $\mathcal{H}(\beta_{j}, L_{j})$ for $\beta_{j}, L_{j}> 0$. Then, under the assumptions of Proposition~\ref{adaptive:prop:upper-bound.high-Proba:g2-g1}, for $q \geq 1$, there exists a constant $C_{j}$ depending on $\beta_{j}, L_{j}, c_{0,j}$ and $K$ such that, with
$$\widetilde{E}_{j,n}(u_{x}, \widehat{h}_{j}) \big):=\left\{\big| \widehat{g}_{j,\widehat{h}_j}(u_{x}) - g_{j}(u_{x}) \big|  \leq C_{j}. \psi_{n}(\beta_{j})\right\},$$
we have, for $n$ large enough,
	\begin{align*} 
\mathbb{P} \Big( \big( \widetilde{E}_{j,n}(u_{x}, \widehat{h}_{j}) \big)^{c} \Big)  &\leq  4.n^{-q},
\end{align*}
\end{corollary}
\begin{proof}[Proof of Corollary~\ref{adaptive:corollary:concentration:g2-g1}]
Since $g_{j}$ belongs to $\mathcal{H}(\beta_{j}, L_{j})$, from Lemma~\ref{lem:point-wise:mean-var-cov.g_{2}},  we have
$$
\left\| g_{j} - K_{h_{j}}*g_{j} \right\|_{\infty}  \leq L_{j}. h_{j}^{\beta_{j}} . C_{K, \mathcal{L}} . 
$$
	From Proposition~\ref{adaptive:prop:upper-bound.high-Proba:g2-g1}, this implies that with a probability greater than $1 - 4.n^{-q}$, one gets for any $h_{j} \in \mathcal{H}_{n}$:
	\begin{align} \label{adaptive:corollary:concentration:g2-g1:proof.bound.g1}
	\big| \widehat{g}_{j,\widehat{h}_j}(u_{x}) - g_{j}(u_{x}) \big|  \leq  \big( 1 + 2 \left\|K\right\|_{\L^{1}(\R)} \big). L_{j}. h_{j}^{\beta_{j}} . C_{K, \mathcal{L}} 
	+  3. \sqrt{\widetilde{V}_{j}(n,h_{j})} .
	\end{align}
	In \eqref{adaptive:corollary:concentration:g2-g1:proof.bound.g1}, we take $h_j$ so that $h_j^{-1}$ is an integer and $h_{j}$ is of order $ \Big( \dfrac{\log(n)}{n} \Big)^{\frac{1}{2 \beta_{j} + 1}}  $. 	Since $h_{\max}=(\log n)^{-1}$ and $1/(2\beta_j+1)< 1$, $h_j\in  \mathcal{H}_{n},$ for $n$ large enough. As a result, we obtain with probability greater than $1 - 4.n^{-q}$, that
	\begin{align*}
	\big| \widehat{g}_{j,\widehat{h}_j}(u_{x}) - g_{j}(u_{x}) \big|  \leq  C_{j}. \psi_{n}(\beta_{j})  ,
	\end{align*}
	with a constant $C_{j}$ (depending on $\beta_{j}, L_{j}, c_{0,j}$ and $K$). This concludes the proof of Corollary~\ref{adaptive:corollary:concentration:g2-g1}.
\end{proof}
Now, we start to prove  Theorem~\ref{adaptive:thm:pointwise-risk.upper-bound:g}.
\begin{proof}[Proof of Theorem~\ref{adaptive:thm:pointwise-risk.upper-bound:g}]
We have $$\E \Big[d_c \big(\widehat{m}_{\widehat{h}}(x), m(x) \big) \Big] = \E \Big[ d_c \big ( \widehat{g}_{\widehat{h}}(u_{x}), g(u_{x}) \big) \Big].$$	
We study $$R_n:=\E \left [  d_c ( \widehat{g}_{\widehat{h}}(u_{x}), g(u_{x})). \ind_{\widetilde{E}_{2,n}(u_{x}, \widehat{h}_{2}) \cap \widetilde{E}_{1,n}(u_{x}, \widehat{h}_{1})} \right ].$$
We have
\begin{align*}
R_n&=\E \Big[ \left ( 1- \cos \big(\Atan( \widehat{g}_{1,\widehat{h}_1}(u_{x}) , \widehat{g}_{2,\widehat{h}_2}(u_{x}))  -  \Atan( g_{1}(u_{x}), g_{2}(u_{x}) \big)\big) \right )   . \ind_{\widetilde{E}_{2,n}(u_{x}, \widehat{h}_{2}) \cap \widetilde{E}_{1,n}(u_{x}, \widehat{h}_{1})}  \Big]\\
&=2 \E \left [ \sin^2 \left ( \frac 1 2  \left (  \Atan( \widehat{g}_{1,\widehat{h}_1}(u_{x}) , \widehat{g}_{2,\widehat{h}_2}(u_{x}))  -  \Atan ( g_{1}(u_{x}), g_{2}(u_{x}))  \right ) \right). \ind_{\widetilde{E}_{2,n}(u_{x}, \widehat{h}_{2}) \cap \widetilde{E}_{1,n}(u_{x}, \widehat{h}_{1}) } \right].
\end{align*}
We now distinguish 3 cases.\\
\underline{Case 1: $|g_1(u_x)|>0$ and $|g_2(u_x)|>0$.}\\ We denote
$$\delta_1=|g_1(u_x)|\quad\mbox{and}\quad \delta_2=|g_2(u_x)|,$$
meaning that $\delta=\min( \delta_1,\delta_2).$
First, on the event $\widetilde{E}_{2,n}(u_{x}, \widehat{h}_{2}) \cap \widetilde{E}_{1,n}(u_{x}, \widehat{h}_{1})$, for $n$ large enough satisfying $$C_{2}. \psi_{n}(\beta_{2})  <  \delta_{2}/2\quad\mbox{and}\quad C_{1}. \psi_{n}(\beta_{1})  < \delta_{1}/2.$$ we have $$\big| \widehat{g}_{2,\widehat{h}_2}(u_{x}) - g_{2}(u_{x}) \big| \leq C_{2}. \psi_{n}(\beta_{2}) < \dfrac{\big|g_{2}(u_{x})\big|}{2}$$ and $$\big| \widehat{g}_{1,\widehat{h}_1}(u_{x}) - g_{1}(u_{x}) \big| \leq C_{1}. \psi_{n}(\beta_{1}) < \dfrac{\big|g_{1}(u_{x})\big|}{2}.$$ Thus, we get $$\widehat{g}_{2,\widehat{h}_2}(u_{x}) . g_{2}(u_{x}) > 0\quad\mbox{and}\quad \widehat{g}_{1,\widehat{h}_1}(u_{x}) . g_{1}(u_{x}) > 0.$$ Therefore,
\begin{align}	
R_n= & \hspace{0.1cm} 2 \E \left [ \sin^2 \left ( \frac 1 2  \left (   \arctan \Big( \dfrac{\widehat{g}_{1,\widehat{h}_1}(u_{x})}{\widehat{g}_{2,\widehat{h}_2}(u_{x})}  \Big)  -  \arctan \Big( \dfrac{g_{1}(u_{x})}{g_{2}(u_{x})} \Big) \right ) \right). \ind_{\widetilde{E}_{2,n}(u_{x}, \widehat{h}_{2}) \cap \widetilde{E}_{1,n}(u_{x}, \widehat{h}_{1}) } \right]    
	\nonumber \\
		\leq & \hspace{0.1cm} \frac 1 2 \E \left [  \left |     \arctan \Big( \dfrac{\widehat{g}_{1,\widehat{h}_1}(u_{x})}{\widehat{g}_{2,\widehat{h}_2}(u_{x})}  \Big)  -  \arctan \Big( \dfrac{g_{1}(u_{x})}{g_{2}(u_{x})} \Big)  \right |^2. \ind_{\widetilde{E}_{2,n}(u_{x}, \widehat{h}_{2}) \cap \widetilde{E}_{1,n}(u_{x}, \widehat{h}_{1}) } \right]    
	\nonumber  \\
	\leq& \hspace{0.1cm} \E  \Big[ \Big| \arctan \Big( \dfrac{\widehat{g}_{1,\widehat{h}_1}(u_{x})}{\widehat{g}_{2,\widehat{h}_2}(u_{x})}  \Big)  -  \arctan \Big( \dfrac{g_{1}(u_{x})}{\widehat{g}_{2,\widehat{h}_2}(u_{x})} \Big) \Big|^{2} . \ind_{\widetilde{E}_{2,n}(u_{x}, \widehat{h}_{2}) \cap \widetilde{E}_{1,n}(u_{x}, \widehat{h}_{1}) } \Big]     \label{adaptive:thm:pointwise-risk.upper-bound:g:proof:control.on.E2n}
	\\
	&+   \E \Big[ \Big| \arctan \Big( \dfrac{g_{1}(u_{x})}{\widehat{g}_{2,\widehat{h}_2}(u_{x})}  \Big)  -  \arctan \Big( \dfrac{g_{1}(u_{x})}{g_{2}(u_{x})} \Big) \Big|^{2} . \ind_{\widetilde{E}_{2,n}(u_{x}, \widehat{h}_{2}) \cap  \widetilde{E}_{1,n}(u_{x}, \widehat{h}_{1}) } \Big]  \nonumber.
	\end{align}
For $n$ sufficiently large, $\big| \widehat{g}_{2,\widehat{h}_2}(u_{x}) \big| \geq \big| g_{2}(u_{x}) \big| - \big| g_{2}(u_{x}) - \widehat{g}_{2,h}(u_{x}) \big| > \delta_{2} - C_{2}. \psi_{n}(\beta_{2}) \geq \delta_{2}/2$ on the event $\widetilde{E}_{2,n}(u_{x}, \widehat{h}_{2})$, and using the $1$-Lipschitz continuity of  $\arctan$, we get for the first term in~\eqref{adaptive:thm:pointwise-risk.upper-bound:g:proof:control.on.E2n}, since on $\widetilde{E}_{1,n}(u_{x}, \widehat{h}_{1})$ one has $\big| \widehat{g}_{1,\widehat{h}_1}(u_{x}) - g_{1}(u_{x}) \big| \leq C_{1}. \psi_{n}(\beta_{1})$
	\begin{align*}
	&\hspace{-1cm}\E \Big[ \Big| \arctan \Big( \dfrac{\widehat{g}_{1,\widehat{h}_1}(u_{x})}{\widehat{g}_{2,\widehat{h}_2}(u_{x})}  \Big)  -  \arctan \Big( \dfrac{g_{1}(u_{x})}{\widehat{g}_{2,\widehat{h}_2}(u_{x})} \Big) \Big|^{2}. \ind_{\widetilde{E}_{2,n}(u_{x}, \widehat{h}_{2}) \cap \widetilde{E}_{1,n}(u_{x}, \widehat{h}_{1}) }  \Big]
	\\
	&\leq  \dfrac{4}{\delta_{2}^2}. \E \Big[ \Big| \widehat{g}_{1,\widehat{h}_1}(u_{x}) - g_{1}(u_{x}) \Big|^{2} .\ind_{\widetilde{E}_{2,n}(u_{x}, \widehat{h}_{2}) \cap \widetilde{E}_{1,n}(u_{x}, \widehat{h}_{1}) } \Big]
	\\
	&\leq  \dfrac{4}{\delta_{2}^2}. \E \Big[  C_{1}^{2}. \psi_{n}(\beta_{1})^{2} . \ind_{ \widetilde{E}_{2,n}(u_{x}, \widehat{h}_{2}) \cap  \widetilde{E}_{1,n}(u_{x}, \widehat{h}_{1}) }  \Big]   
	\\
	&\leq  \dfrac{4}{\delta_{2}^2}.  C_{1}^{2}. \psi_{n}(\beta_{1})^{2} . \mathbb{P} \big( \widetilde{E}_{2,n}(u_{x}, \widehat{h}_{2}) \cap  \widetilde{E}_{1,n}(u_{x}, \widehat{h}_{1}) \big) \\
	&\leq  \dfrac{4}{\delta_{2}^2}.  C_{1}^{2}. \psi_{n}(\beta_{1})^{2} \leq  \dfrac{4}{\delta^2}.  C_{1}^{2}. \psi_{n}(\beta_{1})^{2}.
	\end{align*}
	Moreover, for the second term in~\eqref{adaptive:thm:pointwise-risk.upper-bound:g:proof:control.on.E2n}, since $\dfrac{g_{1}(u_{x})}{\widehat{g}_{2,\widehat{h}_2}(u_{x})}. \dfrac{g_{1}(u_{x})}{g_{2}(u_{x})} >  0$ on $\widetilde{E}_{2,n}(u_{x}, \widehat{h}_{2})$, we have 
	\begin{align*}
	&\E \Big[ \Big| \arctan \Big( \dfrac{g_{1}(u_{x})}{\widehat{g}_{2,\widehat{h}_2}(u_{x})}  \Big)  -  \arctan \Big( \dfrac{g_{1}(u_{x})}{g_{2}(u_{x})} \Big) \Big|^{2} . \ind_{\widetilde{E}_{2,n}(u_{x}, \widehat{h}_{2}) \cap  \widetilde{E}_{1,n}(u_{x}, \widehat{h}_{1}) } \Big]
	\\
	&=  \E \Big[ \Big| \arctan \Big( \dfrac{\widehat{g}_{2,\widehat{h}_2}(u_{x})}{g_{1}(u_{x})}  \Big)  -  \arctan \Big( \dfrac{g_{2}(u_{x})}{g_{1}(u_{x})} \Big) \Big| ^{2} .  \ind_{\widetilde{E}_{2,n}(u_{x}, \widehat{h}_{2}) \cap  \widetilde{E}_{1,n}(u_{x}, \widehat{h}_{1}) } \Big] 	\\
	&\leq  \dfrac{1}{\big| g_{1}(u_{x}) \big|^2}. \E \Big[ \Big|  \widehat{g}_{2,\widehat{h}_2}(u_{x}) - g_{2}(u_{x})  \Big| ^{2} .  \ind_{\widetilde{E}_{2,n}(u_{x}, \widehat{h}_{2}) \cap  \widetilde{E}_{1,n}(u_{x}, \widehat{h}_{1})} \Big]
	\\
	&\leq  \dfrac{1}{\delta_{1}^2}.  \E \Big[  C_{2}^{2}. \psi_{n}(\beta_{2})^{2} . \ind_{ \widetilde{E}_{2,n}(u_{x}, \widehat{h}_{2}) \cap  \widetilde{E}_{1,n}(u_{x}, \widehat{h}_{1}) }  \Big]\\& \leq  \dfrac{1}{\delta_{1}^2}.  C_{2}^{2}. \psi_{n}(\beta_{2})^{2}\leq  \dfrac{1}{\delta^2}.  C_{2}^{2}. \psi_{n}(\beta_{2})^{2}.
	\end{align*}
	\noindent Therefore, on the event $\widetilde{E}_{2,n}(u_{x}, \widehat{h}_{2}) \cap  \widetilde{E}_{1,n}(u_{x}, \widehat{h}_{1})$, 
	for $n$ sufficiently large such that $C_{2}. \psi_{n}(\beta_{2})  \leq  \delta_{2}/2$ and $C_{1}. \psi_{n}(\beta_{1})  \leq  \delta_{1}/2$, we obtain
	\begin{align*}
	&\hspace{-2cm} \E \Big[ \big| \widehat{g}_{\widehat{h}}(u_{x}) - g(u_{x}) \big|^{2}. \ind_{\widetilde{E}_{2,n}(u_{x}, \widehat{h}_{2}) \cap  \widetilde{E}_{1,n}(u_{x}, \widehat{h}_{1})}  \Big]  
	\\
	&\leq   \dfrac{4}{\delta^2}.  C_{1}^{2}. \psi_{n}(\beta_{1})^{2}+  \dfrac{1}{\delta^2}.  C_{2}^{2}. \psi_{n}(\beta_{2})^{2}.
	\end{align*}
	On the other hand, on the complementary $\big( \widetilde{E}_{2,n}(u_{x}, \widehat{h}_{2}) \big)^{c}  \cup  \big( \widetilde{E}_{1,n}(u_{x}, \widehat{h}_{1}) \big)^{c}$, using the fact that $\big| \Atan (w_{1},w_{2}) \big| \leq \pi$, $\forall (w_{1},w_{2})$, we can simply obtain an upper-bound as follows:
\begin{align*}
	&\E \Big[ d_c( \widehat{g}_{\widehat{h}}(u_{x}) - g(u_{x}) ). \ind_{\big( \widetilde{E}_{2,n}(u_{x}, \widehat{h}_{2}) \big)^{c}  \cup  \big( \widetilde{E}_{1,n}(u_{x}, \widehat{h}_{1}) \big)^{c} }  \Big]
	\\
	&\leq \frac 1 2  \E \Big[ \Big| \Atan \big( \widehat{g}_{1,\widehat{h}_1}(u_{x}) , \widehat{g}_{2,\widehat{h}_2}(u_{x})\big)  -  \Atan \big( g_{1}(u_{x}), g_{2}(u_{x}) \big)  \Big|^{2} . \ind_{\big( \widetilde{E}_{2,n}(u_{x}, \widehat{h}_{2}) \big)^{c}  \cup  \big( \widetilde{E}_{1,n}(u_{x}, \widehat{h}_{1}) \big)^{c} }  \Big]
	\\
	&\leq  2 \pi^{2}. \mathbb{P} \left( \big( \widetilde{E}_{2,n}(u_{x}, \widehat{h}_{2}) \big)^{c} \right)   +  2 \pi^{2}. \mathbb{P} \left( \big( \widetilde{E}_{1,n}(u_{x}, \widehat{h}_{1}) \big)^{c} \right) \leq   4 \pi^{2}. 4. n^{-q}, 
\end{align*}  
by Corollary~\ref{adaptive:corollary:concentration:g2-g1}.
For $q \geq 1$, we get that $n^{-q}$ is negligible in comparison with $C_{1}^{2}. \psi_{n}(\beta_{1})^{2} = C_{1}^{2}. \Big( \dfrac{\log (n)}{n} \Big)^{\frac{2\beta_{1}}{2\beta_{1} + 1}}$ and $C_{2}^{2}. \psi_{n}(\beta_{2})^{2} = C_{2}^{2}. \Big( \dfrac{\log (n)}{n} \Big)^{\frac{2\beta_{2}}{2\beta_{2} + 1}}$. 
	
\bigskip

\noindent
\underline{Case 2: $g_1(u_x)=0$ and $|g_2(u_x)|>0$.}\\ 
In this case $\delta=|g_2(u_x)|$.\\
- If $g_2(u_x)>0$, then, as previously, on the event $\widetilde{E}_{2,n}(u_{x}, \widehat{h}_{2}) \cap \widetilde{E}_{1,n}(u_{x}, \widehat{h}_{1})$, for $n$ large enough, $\widehat{g}_{2,\widehat{h}_2}(u_{x})>0$. Then,
\begin{align*}
R_n&=2 \E \left [ \sin^2 \left ( \frac 1 2  \left (  \Atan \big( \widehat{g}_{1,\widehat{h}_1}(u_{x}) , \widehat{g}_{2,\widehat{h}_2}(u_{x})\big)  -  \Atan \big( g_{1}(u_{x}), g_{2}(u_{x})\big)  \right ) \right). \ind_{\widetilde{E}_{2,n}(u_{x}, \widehat{h}_{2}) \cap \widetilde{E}_{1,n}(u_{x}, \widehat{h}_{1}) } \right]\\
&=2 \E \left [ \sin^2 \left ( \frac 1 2  \left (  \Atan \big( \widehat{g}_{1,\widehat{h}_1}(u_{x}) , \widehat{g}_{2,\widehat{h}_2}(u_{x})\big)  -  0  \right ) \right). \ind_{\widetilde{E}_{2,n}(u_{x}, \widehat{h}_{2}) \cap \widetilde{E}_{1,n}(u_{x}, \widehat{h}_{1}) } \right]\\
&=2 \E \left [ \sin^2 \left ( \frac 1 2  \left (  \arctan \Bigg( \dfrac{\widehat{g}_{1,\widehat{h}_1}(u_{x})}{\widehat{g}_{2,\widehat{h}_2}(u_{x})}  \Bigg) -\arctan \Bigg( \dfrac{ g_{1}(u_{x})}{\widehat{g}_{2,\widehat{h}_2}(u_{x})}  \Bigg)\right ) \right). \ind_{\widetilde{E}_{2,n}(u_{x}, \widehat{h}_{2}) \cap \widetilde{E}_{1,n}(u_{x}, \widehat{h}_{1}) } \right]\\
&\leq\frac{1}{2}\E  \Big[ \Big| \arctan \Big( \dfrac{\widehat{g}_{1,\widehat{h}_1}(u_{x})}{\widehat{g}_{2,\widehat{h}_2}(u_{x})}  \Big)  -  \arctan \Big( \dfrac{g_{1}(u_{x})}{\widehat{g}_{2,\widehat{h}_2}(u_{x})} \Big) \Big|^{2} . \ind_{\widetilde{E}_{2,n}(u_{x}, \widehat{h}_{2}) \cap \widetilde{E}_{1,n}(u_{x}, \widehat{h}_{1}) } \Big],     
\end{align*}
and we conclude as for the first case.\\
- If $g_2(u_x)<0$, then, as previously, on the event $\widetilde{E}_{2,n}(u_{x}, \widehat{h}_{2}) \cap \widetilde{E}_{1,n}(u_{x}, \widehat{h}_{1})$, for $n$ large enough, $\widehat{g}_{2,\widehat{h}_2}(u_{x})<0$. Then,
\begin{align*}
R_n&=2 \E \left [ \sin^2 \left ( \frac 1 2  \left (  \Atan \big( \widehat{g}_{1,\widehat{h}_1}(u_{x}) , \widehat{g}_{2,\widehat{h}_2}(u_{x})\big)  -  \Atan \big( g_{1}(u_{x}), g_{2}(u_{x})\big)  \right ) \right). \ind_{\widetilde{E}_{2,n}(u_{x}, \widehat{h}_{2}) \cap \widetilde{E}_{1,n}(u_{x}, \widehat{h}_{1}) } \right]\\
&=2 \E \left [ \sin^2 \left ( \frac 1 2  \left (  \Atan \big( \widehat{g}_{1,\widehat{h}_1}(u_{x}) , \widehat{g}_{2,\widehat{h}_2}(u_{x})\big)  +\pi  \right ) \right). \ind_{\widetilde{E}_{2,n}(u_{x}, \widehat{h}_{2}) \cap \widetilde{E}_{1,n}(u_{x}, \widehat{h}_{1}) } \right]\\
&=2 \E \left [ \sin^2 \left ( \frac 1 2  \left (  \arctan \Bigg( \dfrac{\widehat{g}_{1,\widehat{h}_1}(u_{x})}{\widehat{g}_{2,\widehat{h}_2}(u_{x})}  \Bigg)+2\pi1_{\{\widehat{g}_{1,\widehat{h}_1}(u_{x})> 0\}}  -\arctan \Bigg( \dfrac{ g_{1}(u_{x})}{\widehat{g}_{2,\widehat{h}_2}(u_{x})}  \Bigg)\right ) \right). \ind_{\widetilde{E}_{2,n}(u_{x}, \widehat{h}_{2}) \cap \widetilde{E}_{1,n}(u_{x}, \widehat{h}_{1}) }\right]\\
&=2 \E \left [ \sin^2 \left ( \frac 1 2  \left (  \arctan \Bigg( \dfrac{\widehat{g}_{1,\widehat{h}_1}(u_{x})}{\widehat{g}_{2,\widehat{h}_2}(u_{x})}  \Bigg) -\arctan \Bigg( \dfrac{ g_{1}(u_{x})}{\widehat{g}_{2,\widehat{h}_2}(u_{x})}  \Bigg)\right ) \right). \ind_{\widetilde{E}_{2,n}(u_{x}, \widehat{h}_{2}) \cap \widetilde{E}_{1,n}(u_{x}, \widehat{h}_{1}) } \right]\\
&\leq\frac{1}{2}\E  \Big[ \Big| \arctan \Big( \dfrac{\widehat{g}_{1,\widehat{h}_1}(u_{x})}{\widehat{g}_{2,\widehat{h}_2}(u_{x})}  \Big)  -  \arctan \Big( \dfrac{g_{1}(u_{x})}{\widehat{g}_{2,\widehat{h}_2}(u_{x})} \Big) \Big|^{2} . \ind_{\widetilde{E}_{2,n}(u_{x}, \widehat{h}_{2}) \cap \widetilde{E}_{1,n}(u_{x}, \widehat{h}_{1}) } \Big],  
\end{align*}
and we conclude as for the first case.

\bigskip

\noindent
\underline{Case 3: $|g_1(u_x)|>0$ and $g_2(u_x)=0$.}\\ 
In this case $\delta=|g_1(u_x)|$.\\
- If $g_1(u_x)>0$, then, as previously, on the event $\widetilde{E}_{2,n}(u_{x}, \widehat{h}_{2}) \cap \widetilde{E}_{1,n}(u_{x}, \widehat{h}_{1})$, for $n$ large enough, $\widehat{g}_{1,\widehat{h}_1}(u_{x})>0$. Then,
\begin{align*}
R_n&=2 \E \left [ \sin^2 \left ( \frac 1 2  \left (  \Atan \big( \widehat{g}_{1,\widehat{h}_1}(u_{x}) , \widehat{g}_{2,\widehat{h}_2}(u_{x})\big)  -  \Atan \big( g_{1}(u_{x}), g_{2}(u_{x})\big)  \right ) \right). \ind_{\widetilde{E}_{2,n}(u_{x}, \widehat{h}_{2}) \cap \widetilde{E}_{1,n}(u_{x}, \widehat{h}_{1}) } \right]\\
&=2 \E \left [ \sin^2 \left ( \frac 1 2  \left (  \arctan \Bigg( \dfrac{\widehat{g}_{1,\widehat{h}_1}(u_{x})}{\widehat{g}_{2,\widehat{h}_2}(u_{x})}  \Bigg)+\pi1_{\{\widehat{g}_{2,\widehat{h}_2}(u_{x})<0\}}  -  \frac{\pi}{2}  \right ) \right). \ind_{\widetilde{E}_{2,n}(u_{x}, \widehat{h}_{2}) \cap \widetilde{E}_{1,n}(u_{x}, \widehat{h}_{1}) } \right]\\
  &=2 \E \left [ \sin^2 \left ( \frac 1 2  \left (  \arctan \Bigg( \dfrac{\widehat{g}_{2,\widehat{h}_2}(u_{x})}{\widehat{g}_{1,\widehat{h}_1}(u_{x})}  \Bigg) -\arctan \Bigg( \dfrac{ g_{2}(u_{x})}{\widehat{g}_{1,\widehat{h}_1}(u_{x})}  \Bigg)\right ) \right). \ind_{\widetilde{E}_{2,n}(u_{x}, \widehat{h}_{2}) \cap \widetilde{E}_{1,n}(u_{x}, \widehat{h}_{1}) } \right],
\end{align*}
where the last equality is obtained by using for $x\in\R^*$, 
$$\arctan(x)+\arctan(1/x)=\pi/2\times sign(x)$$ and by distinguishing the cases according to the sign of $\widehat{g}_{2,\widehat{h}_2}(u_{x})$.\\
- If $g_1(u_x)<0$, then, as previously, on the event $\widetilde{E}_{2,n}(u_{x}, \widehat{h}_{2}) \cap \widetilde{E}_{1,n}(u_{x}, \widehat{h}_{1})$, for $n$ large enough, $\widehat{g}_{1,\widehat{h}_1}(u_{x})<0$. Then, similarly,
\begin{align*}
R_n&=2 \E \left [ \sin^2 \left ( \frac 1 2  \left (  \Atan \big( \widehat{g}_{1,\widehat{h}_1}(u_{x}) , \widehat{g}_{2,\widehat{h}_2}(u_{x})\big)  -  \Atan \big( g_{1}(u_{x}), g_{2}(u_{x})\big)  \right ) \right). \ind_{\widetilde{E}_{2,n}(u_{x}, \widehat{h}_{2}) \cap \widetilde{E}_{1,n}(u_{x}, \widehat{h}_{1}) } \right]\\
&=2 \E \left [ \sin^2 \left ( \frac 1 2  \left (  \arctan \Bigg( \dfrac{\widehat{g}_{1,\widehat{h}_1}(u_{x})}{\widehat{g}_{2,\widehat{h}_2}(u_{x})}  \Bigg)-\pi1_{\{\widehat{g}_{2,\widehat{h}_2}(u_{x})<0\}}  + \frac{\pi}{2}  \right ) \right). \ind_{\widetilde{E}_{2,n}(u_{x}, \widehat{h}_{2}) \cap \widetilde{E}_{1,n}(u_{x}, \widehat{h}_{1}) } \right]\\
  &=2 \E \left [ \sin^2 \left ( \frac 1 2  \left (  \arctan \Bigg( \dfrac{\widehat{g}_{2,\widehat{h}_2}(u_{x})}{\widehat{g}_{1,\widehat{h}_1}(u_{x})}  \Bigg) -\arctan \Bigg( \dfrac{ g_{2}(u_{x})}{\widehat{g}_{1,\widehat{h}_1}(u_{x})}  \Bigg)\right ) \right). \ind_{\widetilde{E}_{2,n}(u_{x}, \widehat{h}_{2}) \cap \widetilde{E}_{1,n}(u_{x}, \widehat{h}_{1}) } \right].
\end{align*}
We conclude by using the second case since $\widehat{g}_{1,\widehat{h}_1}(u_{x})$ (resp. $g_{1}(u_{x})$) and $\widehat{g}_{2,\widehat{h}_2}(u_{x})$ (resp. $g_{2}(u_{x})$) play a symmetric role.

\bigskip

Note that under Assumption~\ref{assumption:c_{low}}, the case $g_1(u_x)=g_2(u_x)=0$ cannot occur.

\end{proof}
\subsubsection{Proof of Theorem~\ref{lower-bound}}
Before tackling the proof of Theorem \ref{lower-bound}, the next lemma shows that the von Mises density satisfies condition (\ref{condition-erreur}).
\begin{lem}\label{preuve-condition-erreur}
The von Mises density with location parameter $\mu$ and concentration parameter $\kappa$ satisfies condition (\ref{condition-erreur}).
\end{lem}
\begin{proof}[Proof of Lemma~\ref{preuve-condition-erreur}]
	We recall the expression of the von Mises density with location parameter  $\mu\in [-\pi, \pi)$ and concentration parameter $\kappa>0$:
	$$
	f_{vM(\mu, \kappa)}(\theta)=c(\kappa)e^{\kappa \cos(\theta-\mu)},\quad \theta\in[-\pi,\pi)
	$$
	with $c(\kappa)$ the normalizing constant. Let us prove that $f_{vM(\mu, \kappa)}$ satisfies:
	\begin{equation*}\label{hyp1}
	\exists\, p_*>0: \int_{-\pi}^\pi f_{vM(\mu, \kappa)}(t) \log \frac{f_{vM(\mu, \kappa)}(t)}{f_{vM(\mu, \kappa)}(t+\theta)} dt \leq p_* \theta^2,
	\end{equation*}
	for all $\theta\in\R$.
	We have that 
	\begin{eqnarray*}
		\int_{-\pi}^\pi f_{vM(\mu, \kappa)}(t) \log \frac{f_{vM(\mu, \kappa)}(t)}{f_{vM(\mu, \kappa)}(t+\theta)} dt &=& c(\kappa)\kappa \int_{-\pi}^{\pi}  e^{\kappa \cos(t-\mu)} (\cos (t-\mu) - \cos(t+\theta-\mu)) dt 
		\\
		&=&2 c(\kappa)\kappa \sin \frac \theta 2 \int_{-\pi}^{\pi}  e^{\kappa \cos (t-\mu)} \sin (t-\mu + \frac \theta 2) dt 
		\\
		&=& 2 c(\kappa)\kappa \sin \frac \theta 2 \int_{-\pi}^{\pi}  e^{\kappa \cos (t-\mu)} \Big( \sin (t-\mu) \cos \frac \theta 2 
		\\
		&\textrm{}& \hspace{4cm} + \sin \frac \theta 2 \cos (t-\mu)\Big) dt 
		\\
		&=& 2 c(\kappa)\kappa \Big(\sin \frac \theta 2\Big)^2 \underbrace{\int_{-\pi}^{\pi} e^{\kappa \cos (t-\mu)} \cos (t-\mu)dt}_{=: C(\kappa)>0} 
		\\
		& \leq & 2c(\kappa)\kappa C(\kappa)\frac{\theta^2}{4}
	\end{eqnarray*}
	for all $\theta\in\R$. Then, with $$p_*= \frac{c(\kappa)\kappa C(\kappa)}{2},$$ we have for any $\theta\in\R$,
	$$\int_{-\pi}^\pi f_{vM(\mu, \kappa)}(t) \log \frac{f_{vM(\mu, \kappa)}(t)}{f_{vM(\mu, \kappa)}(t+\theta)} dt \leq p_* \theta^2.$$
\end{proof}
\begin{proof}[Proof of Theorem~\ref{lower-bound}]
	To prove the lower bound stated in Theorem \ref{lower-bound}, we follow the lines of Section 2.5  in \cite{book:Tsybakov2009} for the regression at a point. The differences with our problem lie in the circular response and the randomness of the $X_i$'s. 
	We consider $m_0(t)=0$ and $m_1(t)=Lh_n^\beta K(\frac{t-x}{h_n})$ with $h_n=c_0n^{-\frac{1}{2\beta+1}}$
	and $K: \R \longmapsto \mathbb{S}^1$ satisfying:
	$$
	K \in \tilde \Sigma(\beta, 1/2) \cap C^{\infty}(\R), \quad K(t)>0 \iff t \in ]-1/2, 1/2[.
	$$
	Such functions $K$ exist. For instance,  for a sufficiently small $a>0$, one can take 
	$$
	K(t)=a  \exp \left (-\frac{1}{1-4t^2} \right ) 1_{[-0.5;0.5]}(t).$$
	We have now to check three points which are developed in the sequel. 
	\begin{enumerate}
		\item Let us prove that $m_1 \in \tilde \Sigma(\beta,L)$ (the function $m_0$ obviously belongs to $\tilde \Sigma(\beta,L)$).
		For $l = \lfloor \beta \rfloor$ we have 
		$$
		m_1^{(l)}(t)= L h_n^{\beta-l} K^{(l)}\left(\frac{t-x}{h_n} \right)
		$$
		then, with $u=\frac{t-x}{h_n}$ and $u'=\frac{t'-x}{h_n}$, 
	\begin{eqnarray*}
			d_c(m_1^{(l)}(t),m_1^{(l)}(t')) &=& 1-\cos(m_1^{(l)}(t)- m_1^{(l)}(t')) \\
			& \leq & 2 \sin^2((m_1^{(l)}(t)- m_1^{(l)}(t'))/2) \\
			&\leq &  \frac 1 2  |m_1^{(l)}(t)- m_1^{(l)}(t')|^2 \\
			&=& \frac{L^2}{2} h_n^{2(\beta-l)} |K^{(l)}(u) -K^{(l)}(u')|^2 \\
			& \leq & \frac{L^2}{8} h_n^{2(\beta -l)} |u-u'|^{2(\beta-l)} = \frac{L^2}{8} |t-t'|^{2(\beta-l)}.
		\end{eqnarray*}
		Then, $m_1 \in \tilde \Sigma(\beta,L)$.
		\item Let us show that  $d_c(m_0(x),m_1(x)) \geq 4 s_n^2$. \\
		We have that $m_1(x) = Lh_n^\beta K(0)=Lc_0^\beta K(0)n^{-\frac{\beta}{2\beta+1}}$, hence for $n$ sufficiently large, $m_1(x) \in [0, \frac \pi 2]$. Hence using that $\sin(t) \geq \frac 2 \pi t$ for $t \in[0, \frac \pi 2]$, we get
		\begin{eqnarray*}
		d_c(m_0(x),m_1(x)) &=& 1-\cos(m_1(x)) = 2\sin^2(m_1(x)/2) \\
		& \geq &2 \left (\frac 2 \pi \right )^2 \frac{m_1^2(x)}{4} \\
		&= & \frac {2}{\pi ^2} L^2c_0^{2\beta} K^2(0) n^{-\frac{2\beta}{2\beta +1}}
		\end{eqnarray*}
		then the condition is fulfilled with $s_n=\frac{ 1 }{ \sqrt{2} \pi } L c_0^\beta K(0) n^{-\frac{\beta}{2 \beta +1}}=: A\psi_n.$ 
		\item Using the classical reduction to a two test hypotheses problem for the pointwise regression problem, we get for any estimator $T_n$:
\begin{eqnarray}\label{reduction}
		&\textrm{}& \hspace{-1.5cm} \sup_{m\in \tilde \Sigma(\beta,L)} \E_m[\psi_n^{-2} d_c(T_n,m(x))]
		\\
		&\geq& A^2 \max_{m \in \{ m_{0}, m_{1}\}} \mathbb{P}_m ( d_c(T_n,m(x)) \geq A^2 {\psi_n}^2) \nonumber \\
		&\geq & \frac{A^2}{ 2} \E_{X_1,\dots, X_n}  \Big [ \mathbb{P}_{m_{0}}(d_c( T_n, m_0(x)) \geq A^2 \psi_n^2 | X_1, \dots, X_n) \nonumber \\
		&&\hspace{3cm}+ \mathbb{P}_{m_{1}}(d_c( T_n, m_1(x)) \geq A^2 \psi_n^2 | X_1, \dots, X_n) \Big ] \nonumber \\
		& \geq & \frac{A^2}{2}  \E_{X_1,\dots, X_n} \Big[\inf_{\psi}\Big\{\mathbb{P}_{m_0}( \psi \neq 0 | X_1, \dots, X_n)  \nonumber \\
		&&\hspace{3cm}+ \mathbb{P}_{m_1}( \psi \neq 1| X_1, \dots, X_n)\Big\}\Big], 
		\end{eqnarray}
		where $\inf_{\psi}$ denotes the infimum over all tests $\psi$ taking values in $\{0,1\}$. We have used that $\sqrt{d_c}$ is a true distance on $\mathbb{S}^1$, so that it satisfies the triangular inequality.  
		
		Now let us fix the $X_i$'s. The minimum average probability $\overline p_{e,1}$ is defined as (see page 116 in \cite{book:Tsybakov2009}):
		$$
		\overline p_{e,1}:= \frac 1 2 \inf_{\psi} \Big\{\mathbb{P}_{m_0}( \psi \neq 0 | X_1, \dots, X_n) + \mathbb{P}_{m_1}( \psi \neq 1| X_1, \dots, X_n)\Big\}.
		$$
		
		We have for the Kullback Leibler divergence (still with the $X_i$'s fixed)
		\begin{equation}\label{Kullback}
		\bold{K}(\mathbb{P}_{m_0}, \mathbb{P}_{m_1}) =\int \log \left(\frac{d\mathbb{P}_{m_0}}{d\mathbb{P}_{m_1}}\right)d\mathbb{P}_{m_0} =\sum_{i=1}^n  \int \log \frac{p_\zeta(y)}{p_\zeta(y-m_1(X_i))}p_{\zeta}(y) dy.
		\end{equation}
		
		There exists $n_0$ such that $\forall n>n_0$, $Lh_n^{\beta} K_{\max} \leq y_0$ where $K_{\max} = \max_t |K(t)|$. Using (\ref{Kullback}) and (\ref{condition-erreur}),  we have:
		\begin{eqnarray*}
			\bold{K}(\mathbb{P}_{m_0}, \mathbb{P}_{m_1})
						& \leq & p_* \sum_{i=1}^n    m_1^2(X_i) \\
			& \leq & p_*L^2  h_n^{2\beta} K^2_{\max}   \sum_{i=1}^n 1_{ \left\{\left |\frac{X_i-x}{h_n} \right | \leq \frac 1 2\right\}}. \\
		\end{eqnarray*}
		Now taking the expectation and using that the density of the $X_i$'s is bounded by $\mu_0$, we get:
		\begin{align*}
		\E_{X_1,\dots, X_n} \bold{K}(\mathbb{P}_{m_0}, \mathbb{P}_{m_1}) &\leq p_*L^2   K^2_{\max}  h_n^{2\beta} n \mathbb{P} \left (\left  |\frac{X_1-x}{h_n} \right | \leq \frac 1 2 \right  ) 
		\\
		&\leq p_*L^2 K^2_{\max} \mu_0 h_n^{2\beta+1} n.
		\end{align*}
		For $\alpha<2\log(2)$, since $h_n=c_0n^{-\frac{1}{2\beta+1}}$, setting 
		$$
		c_0 =\left ( \frac{\alpha}{p_*\mu_0L^2 K^2_{\max}}\right )^{\frac{1}{2\beta +1}},
		$$ 
		we get that 
		$$
		\E_{X_1,\dots, X_n} \bold{K}(\mathbb{P}_{m_0}, \mathbb{P}_{m_1}) \leq \alpha.
		$$
		{As in Lemma~2.10 of \cite{book:Tsybakov2009}, we introduce the function $\mathcal{H}(t)=-t\log(t)-(1-t)\log(1-t)$ for $t\in(0,1)$ and $\mathcal{H}(0)= \mathcal{H}(1)=0$. Inequality (2.70) of \cite{book:Tsybakov2009} with $M=1$ gives
			$$\E_{X_1,\dots, X_n}[\mathcal{H}(\overline p_{e,1})]\geq \log(2)-\frac{1}{2}\E_{X_1,\dots, X_n} \bold{K}(\mathbb{P}_{m_0}, \mathbb{P}_{m_1})\geq \log(2)-\frac{\alpha}{2},$$
			since $\E_{X_1,\dots, X_n} \bold{K}(\mathbb{P}_{m_0}, \mathbb{P}_{m_1}) \leq \alpha$. Since $\mathcal{H}$ is concave, $\mathcal{H}(\E_{X_1,\dots, X_n}[\overline p_{e,1})])\geq \E_{X_1,\dots, X_n}[\mathcal{H}(\overline p_{e,1})]$ and
			$$
			\E_{X_1,\dots, X_n} [  \overline p_{e,1}] \geq  \mathcal{H}^{-1}\big(\log 2 -\frac \alpha 2\big)>0,
			$$
			with, for any $t>0$, $\mathcal{H}^{-1}(t)=\min \{ p \in (0, \frac 1 2 ] : \mathcal{H}(p) \geq t \}$. Hence we deduce  using (\ref{reduction})
			\begin{eqnarray*}
				\sup_{m\in \tilde \Sigma(\beta,L)} \E_m\big[\psi_n^{-2} |d_c(T_n,m(x))\big]\geq A^2  \mathcal{H}^{-1}\big(\log 2 -\frac \alpha 2\big),
			\end{eqnarray*}  
			where the right hand side is a positive constant. This concludes the proof of Theorem~\ref{lower-bound}.
		} 
	\end{enumerate}
\end{proof}
\section{Conclusion}
Considering nonparametric regression for circular data, we derive minimax convergence rates and prove near optimal properties of our kernel estimate combined with a warping strategy on anisotropic H\"older classes of functions for pointwise estimation. The bandwidth parameter is selected by using a data-driven Goldenshluger-Lepski type procedure. After tuning hyperparameters of our estimate, we show that it remains very competitive with respect to existing methods.

As a natural extension, it could be very challenging to investigate our regression problem with a response on the sphere $\mathbb{S}^2$ or more generally on the unit hypersphere $\mathbb{S}^{d-1}$. The case of predictors $X \in \mathbb{S}^{d-1}$ and a response $\Theta \in \mathbb{S}^{d-1}$ has been tackled in~\cite{Marzio-Panzera-Taylor:2014}.
The spherical context is of course more complicated than the circular one and the arctangent function approach used here is not easily generalizable in the spherical setting. In~\cite{Marzio-Panzera-Taylor:2014}, Di Marzio et al. proposed a local constant estimator by smoothing on each component of the response. Once again no rates of convergence were obtained. Hence, in a future work, a first task would be to obtain convergence rates and then investigate adaptation issue. 

%
%
%
%
%
%
\begin{acks}[Acknowledgments]
The authors would like to warmly thank the anonymous referees for very valuable comments and suggestions.
\end{acks}
%

\end{document}